\documentclass[10pt]{amsart}

\usepackage{amssymb, amsmath, amsthm, amsfonts, stmaryrd, mathrsfs}
\usepackage{graphicx}
\usepackage{tikz}
\usetikzlibrary{shapes,arrows}
\usepackage{float}
\usepackage{hvfloat}
\usepackage{caption}
\usepackage{url}
\usepackage[all,arc,2cell]{xy}
\UseAllTwocells
\usepackage{enumerate}
\usepackage{color}

\usepackage{tabu}

\usepackage{hyperref}
  \definecolor{dark-red}{rgb}{0.6,0.15,0.15}
   \definecolor{dark-blue}{rgb}{0.15,0.15,0.6}
   \definecolor{medium-blue}{rgb}{0,0,0.5}

\hypersetup{
    colorlinks, 
    linkcolor=dark-red,
    citecolor=dark-blue, urlcolor=medium-blue
}

\usepackage[nameinlink,capitalise,noabbrev]{cleveref}

\numberwithin{equation}{section}

\newtheorem{thm}{Theorem}[section]
\newtheorem{theorem}{Theorem}[section]
\newtheorem{Theorem}{Theorem}[section]
\newtheorem{cor}{Corollary}[section]

\newtheorem{prop}{Proposition}[section]

\newtheorem{lem}{Lemma}[section]
\newtheorem{lemma}{Lemma}[section]

\theoremstyle{definition}
\newtheorem{defn}{Definition}[section]

\newtheorem{ex}{Example}[section]

\newtheorem{rem}{Remark}[section]

\makeatletter
\let\c@equation=\c@thm
\let\c@lem=\c@thm
\let\c@theorem=\c@thm
\let\c@lemma=\c@thm
\let\c@Theorem=\c@thm
\let\c@Lemma=\c@thm
\let\c@cor=\c@thm
\let\c@corollary=\c@thm
\let\c@Corollary=\c@thm
\let\c@conj=\c@thm
\let\c@conjecture=\c@thm
\let\c@prop=\c@thm
\let\c@proposition=\c@thm
\let\c@Proposition=\c@thm
\let\c@defn=\c@thm
\let\c@definition=\c@thm
\let\c@Definition=\c@thm
\let\c@notation=\c@thm
\let\c@note=\c@thm
\let\c@exmp=\c@thm
\let\c@ex=\c@thm
\let\c@exmps=\c@thm
\let\c@rem=\c@thm
\let\c@warn=\c@thm
\let\c@claim=\c@thm
\let\c@convention=\c@thm
\let\c@conventions=\c@thm
\let\c@quest=\c@thm
\let\c@facts=\c@thm
\let\c@slogan=\c@thm
\makeatother

\newcommand{\F}{\mathbb{F}}

\newcommand{\Z}{\mathbb{Z}}
\newcommand{\W}{\mathbb{W}}

\newcommand{\G}{\mathbb{G}}
\newcommand{\R}{\mathbb{R}}
\newcommand{\ZZ}{\mathbb{Z}}
\newcommand{\FF}{\mathbb{F}}

\newcommand{\GG}{\mathbb{G}}
\newcommand{\wchi}{\widetilde{\chi}}

\def\SS{\mathbb{S}}

\newcommand{\xra}{\xrightarrow}
\newcommand{\cC}{\mathcal{C}}
\newcommand{\cS}{\mathcal{S}}

\newcommand{\sE}{\mathscr{E}}

\def\makeop#1{\expandafter\def\csname #1\endcsname{\mathop{\mathrm{#1}}\nolimits}}

\makeop{Gal}
\makeop{id}
\makeop{Mod}
\makeop{Hom}
\makeop{Tot}
\makeop{gr}
\makeop{Out}
\makeop{Hom}
\makeop{Ext}
\makeop{End}
\makeop{Aut}
\makeop{Tor}
\makeop{ev}
\makeop{hocolim}
\makeop{holim}
\makeop{map}
\makeop{id}
\makeop{colim}
\makeop{im}
\makeop{pic}
\newcommand{\kappabar}{\bar{\kappa}}

\newcommand{\WW}{\mathbb{W}}

\def\KEn{\Lambda} %Vesna trying to be innovative with notation
\def\KE{\Lambda} 
\def\KIn{\Lambda^\infty}

\def\longr{{\longrightarrow}}
\def\RP{{\mathbb{R}\mathrm{P}}}
\def\FF{{\mathbb{F}}}
\def\bX{{{\mathbf{X}}}}
\def\bY{{\mathbf{Y}}}

\newcommand{\Einf}{\mathcal{E}_{\infty}} % for E_\infty rings

\usepackage{xargs} 
\usepackage[textwidth=2.5cm]{todonotes}
\newcommandx{\irina}[2][1=]{\todo[linecolor=red,backgroundcolor=red!25,bordercolor=red,#1]{#2}}
\newcommandx{\cuong}[2][1=]{\todo[linecolor=orange,backgroundcolor=orange!25,bordercolor=orange,#1]{#2}}
\newcommandx{\agnes}[2][1=]{\todo[linecolor=blue,backgroundcolor=blue!25,bordercolor=blue,#1]{#2}}
\newcommandx{\agnesinline}[2][1=]{\todo[linecolor=blue,backgroundcolor=blue!25,bordercolor=blue,#1,inline]{#2}}

\newcommandx{\paul}[2][1=]{\todo[linecolor=green,backgroundcolor=green!25,bordercolor=green,#1]{#2}}
\newcommandx{\vesna}[2][1=]{\todo[linecolor=purple,backgroundcolor=purple!25,bordercolor=purple,#1]{#2}}
\newcommandx{\hanswerner}[2][1=]{\todo[linecolor=cyan,backgroundcolor=cyan!25,bordercolor=purple,#1]{#2}}
\makeatletter
\newcommand{\mylabel}[2]{#2\def\@currentlabel{#2}\label{#1}}
\makeatother

\newcommand{\Pic}{\mathrm{Pic}}

\def\Gl{{\mathrm{Gl}}}
\def\Mor{{\mathfrak{Mor}}}

 \newcommand{\tmf}{E^{hG_{48}}}

\title[Exotic $K(2)$-Local Picard Group]{The Exotic $K(2)$-Local Picard Group at the Prime $2$}
\date{\today}

\author[Beaudry]{Agn\`es Beaudry}
\address{Department of Mathematics, University of Colorado, Boulder, Boulder, CO, Campus Box 395,
Boulder, CO, 80309, USA}
\email{agnes.beaudry@colorado.edu}

\author[Bobkova]{Irina Bobkova}
\address{Department of Mathematics, Texas A\&M University, College Station, TX, 77843, USA}
\email{ibobkova@tamu.edu}

\author[Goerss]{Paul G. Goerss}
\address{Department of Mathematics\\ Northwestern University \\  2033 Sheridan Road 
\\ Evanston \\ Illinois \\ 60208 }
\email{pgoerss@math.northwestern.edu}

\author[Henn]{Hans-Werner Henn}
\address{Institut de Recherche Math\'ematique Avanc\'ee, C.N.R.S. et 
Universit\'e de Strasbourg \\ 7, rue Ren\'e Descartes \\ 67084 Strasbourg Cedex \\ France}
\email{henn@math.unistra.fr}

\author[Pham]{Viet-Cuong Pham}
\address{Lycée Charles Jully \\ 59 rue Maréchal Foch \\ 57500 Saint-Avold \\France}
\email{phamvietcuonga2@gmail.com}

\author[Stojanoska]{Vesna Stojanoska}
\address{Department of Mathematics, University of Illinois, Urbana-Champaign, 273 Altgeld Hall
1409 W. Green Street, 
Urbana, IL 61801}
\email{vesna@illinois.edu}

\begin{document}

\begin{abstract} We calculate the group $\kappa_2$ of exotic elements in the $K(2)$-local Picard group at the prime
$2$ and find it is a group of order $2^9$ isomorphic to $(\ZZ/8)^2 \times (\ZZ/2)^3$. In order to do this we must define and
exploit a variety of different ways of constructing elements in the Picard group, and this requires a significant
exploration of the theory. The most innovative technique, which so far has worked best at the prime $2$, is
the use of a $J$-homomorphism from the group of real representations of finite quotients of the Morava stabilizer
group to the $K(n)$-local Picard group. 
\end{abstract}
 
\maketitle

\setcounter{tocdepth}{1}
\tableofcontents

\renewcommand{\thepart}{\Roman{part}}

 % !TEX root = pic-master.tex

\section{Introduction}

The focus of this paper is a computation and analysis of the subgroup $\kappa_n$ of exotic elements in
the Picard group of the $K(n)$-local stable homotopy category at a prime $p$. We give quite a few
structural results and outline some significant techniques, all culminating in the  
main result \Cref{thm:kappa}, an isomorphism at $n=p=2$
\[
\kappa_2 \cong (\ZZ/8)^2 \times (\ZZ/2)^3.
\]
We immediately note that this isomorphism, in itself, gives little insight into the importance of this group, where
the elements arise, or what they mean. In fact, we might be better served writing
\[
\kappa_2 \cong \big[\ZZ/8 \times (\ZZ/2)^2\big] \times \ZZ/2 \times \ZZ/8
\]
reflecting the fact that the elements in this group arise in three different ways and reflect three fundamentally different
aspects of the $K(2)$-local category. Making this clear and precise requires some explanation,
and we begin there. 

Let $(\cC,\otimes)$ be a symmetric monoidal category with unit object $I$.  An object $X\in \cC$ is invertible if
there exists an object $Y$ and an isomorphism $X\otimes Y\cong I$. If the collection of isomorphism classes of
invertible objects is a set, then $\otimes$ defines a group structure on this set. 
This basic invariant of the category $\cC$ is the {\it Picard group} $\Pic(\cC)$. 

We are particularly interested in examples from stable homotopy theory. If we consider the stable homotopy category
of  spectra, the answer turns out to be simple, even if the proof is not:
the only invertible spectra are the sphere spectra $S^n$, $n \in \ZZ$  and the map $\ZZ \to \Pic(\cS)$ sending $n$
to $S^n$ is an isomorphism. See \cite{HopMahSad}. 

However, it is an insight due to Mike Hopkins 
\cite{StricklandInterpolation} that we don't need to stop there and that, in fact, we can get a great deal
of information if we pass to the localized stable homotopy categories which arise in chromatic stable
homotopy theory. The basic example is then  $K(n)$-local spectra, where $K(n)$ is the $n$th Morava $K$-theory
at a fixed prime $p$. 

Thus, the aim of this paper is to study $\Pic_n$, the Picard group of the $K(n)$-local stable homotopy category at a
fixed prime $p$. The first observation is that the $K(n)$-local category as a whole and its Picard group 
in particular can be accessed using Morava $E$-theory.

We fix a formal group law $\Gamma_n$ of height $n$ over the field $\FF_{p^n}$ and let
$E = E_n = E(\FF_{p^n},\Gamma_n)$ be the associated Lubin-Tate or Morava $E$-theory. (For a few more details,
see \Cref{sec:background-kn}.) For a spectrum $X$ we  then define
\[
E_\ast X = \pi_\ast L_{K(n)}(E \wedge X).
\]
This graded $E_\ast$-module has a continuous action by the Morava stabilizer group
$\GG_n = \Aut(\FF_{p^n},\Gamma_n)$, which gives it the structure of a Morava module in the sense
of \cite[Definition 5.2.30]{BarthelBeaudry}. The completed tensor product over $E_\ast $ endows the category of Morava  modules with
a symmetric monoidal structure with unit $E_\ast  S^0$. Thus, the category of Morava modules also has 
a Picard group, which we write $(\Pic_n)_{\mathrm{alg}}$ and call the \emph{algebraic Picard group}. It consists of isomorphism classes of
Morava modules which are free of rank $1$ over $E_\ast$. A basic fact from the foundational paper 
\cite{HopMahSad} of Hopkins, Mahowald, and Sadofsky says that a $K(n)$-local spectrum
$X$ is in $\Pic_n$ if and only if $E_\ast X$ is in $(\Pic_n)_{\mathrm{alg}}$. Thus we obtain 
a group homomorphism
\begin{equation}\label{eq:top-to-alg-intro}
\varepsilon: \Pic_n \to (\Pic_n)_{\mathrm{alg}} \ .
\end{equation}
Elements of the kernel of $\varepsilon$, denoted by $\kappa_n$, are the \emph{exotic} invertible $K(n)$-local spectra and $\kappa_n$ is the \emph{exotic Picard group}. 
%his paper is focused on the group of \emph{exotic} invertible $K(n)$-local spectra, denoted by $\kappa_n$, and defined as the 
%kernel of $\varepsilon$. 
Thus, an invertible element $X$ is in $\kappa_n$ if and only if
\[
E_\ast X \cong E_\ast S^0 = E_\ast
\] 
as Morava modules. 

To understand $\Pic_n$, one must
\begin{enumerate}[(1)]
\item\label{it:1intro} compute the algebraic Picard group $ (\Pic_n)_{\mathrm{alg}}$,
\item\label{it:3intro} compute the exotic Picard group $\kappa_n$, 
\item\label{it:2intro} determine the image of $\varepsilon$, and
\item\label{it:4intro} resolve the extension
$ 0 \to \kappa_n \to \Pic_n \to \im(\varepsilon) \to 0$.
\end{enumerate}
This paper studies the second of these items at $n=p=2$ and establishes structural results 
at all heights and primes. But before getting into the details of our results, it is worth summarizing what we know about
some of the other items.

The computation of the algebraic Picard group is equivalent to the computation of the continuous cohomology 
of $\GG_n$ with coefficients in the units of the ring $E_0 = E_0S^0$. This was widely regarded as impossible.
That all changed with the work of Barthel--Schlank--Stapleton--Weinstein \cite{BSSW}, which used deep
techniques and computations from arithmetic geometry to give a complete calculation of the algebraic Picard group
and, except in one case, the image of  $\varepsilon$. This link to arithmetic geometry goes back to
Gross--Hopkins \cite{GrossHopkins2}, among others, but there has been revolutionary work in that field in the last decade
or so, and the observation of Barthel, Schlank, Stapleton, and Weinstein is that these new directions made the
connection very concrete. The paper \cite{BSSW} offers at once a fundamental insight and a tour de force of technique. 

More specifically, in Theorem 2.4.3 of \cite{BSSW}, the authors show that at primes $p \ne 2$, the algebraic Picard group is 
isomorphic to $\ZZ_p^2 \oplus \ZZ/2(p-1)$. It is topologically generated by $E_\ast S^1$ and
$E_\ast S\langle \mathrm{det}\rangle$, where $S\langle \mathrm{det}\rangle$ is the determinant sphere,
as in \eqref{eq:Gross-Hopkins-relation}.  The prime $2$ is only slightly more complicated.  A number of special cases
were known before the results of  \cite{BSSW}. See, for example,  \cite{HopMahSad} for $n=1$ and \cite{Lader} and 
\cite{Karamanov} for $n=2$ and $p \geq 3$. Furthermore, Hopkins had long suggested that the algebraic Picard group would be 
generated by $E_\ast S^1$ and $E_\ast S\langle \mathrm{det}\rangle$. 

Except possibly in the case $n=p=2$, the map $\varepsilon$ of \eqref{eq:top-to-alg-intro}
is onto, and in the exceptional case it is onto or the image has index $2$. See Corollary 2.4.5 of \cite{BSSW}; the exceptional 
case was also known to Henn. If $p$ is large with respect to $n$ then $\varepsilon$ is an isomorphism;  for example,
if $n=2$, then we need $p > 3$. This is a reflection of the fact that in those cases the homotopy theory of $K(n)$-local spectra is
algebraic, in a sense which can be made completely precise. See \cite{PstPic,HeardPicAlg,PstAlgebraic} 
for the state of the art. Now that we have \cite{BSSW}, some results of this sort follow immediately
from older sources, such as \cite{HopMahSad} and \cite{HoveySadofsky}.

Systematic results on $\kappa_n$ at low primes remain elusive. It was known right from the start that $\kappa_n$
could be non-zero; for example, if $p=2$, $\kappa_1 \cong \ZZ/2$ by \cite{HopMahSad}. If $p=3$, then
$\kappa_2 \cong \ZZ/3 \times \ZZ/3$ by \cite{GHMRPicard}. In earlier work  \cite{ShimomuraKamiya}, Kamiya and
Shimomura had shown $\kappa_2 \ne 0$ at $p=3$, and had given an upper bound as well. For all primes $p$, 
the group $\kappa_{p-1}$ contains non-trivial elements of  order $p$, by \cite{BGHS}, and if $p=2$ then $\kappa_n \ne 0$
by \cite{HeardLiShi}, who also produce elements of $2$-power order which grows with $n$. There are other qualitative results of 
this type when $n$ is large (but not too large) with  respect  to $p$; see, for example, \cite{CulverZhang} and \cite{BLLSZ}. 
General considerations using \cite{HovStrick} and the vanishing line results of \cite{DH}
show that there is an integer $N$ so that $p^N\kappa_n=0$, but we have no good bounds on $N$.
It is possible to conjecture that $\kappa_n$ is finite, just as it is  possible to conjecture that $\pi_\ast L_{K(n)}S^0$
is topologically finitely generated in each degree. Both conjectures could be deduced from finiteness results for
the cohomology of the group $\GG_n$. The developing connections between chromatic homotopy theory 
and arithmetic geometry may bring such finiteness results within reach.

Our basic technique for studying $\kappa_n$ uses the subgroup structure of $\GG_n$. By Devinatz, Goerss, Hopkins, and
Miller, the group $\GG_n$ acts continuously on the spectrum $E$ through $\sE_\infty$-ring maps; 
thus if 
$H \subseteq \GG_n$ is a closed subgroup, the homotopy fixed point spectrum $E^{hH}$ is
a ring spectrum with an associated category of modules, which has a Picard group $\Pic(E^{hH})$ and
we get a homomorphism $\Pic_n \to \Pic(E^{hH})$. We would like to define $\kappa(H) \subseteq 
\kappa_n$ to be those elements in the kernel of this map; that is, those $X \in \kappa_n$ so
that there is an equivalence $E^{hH} \simeq E^{hH} \wedge X$ of $E^{hH}$-modules. For various reasons, we
need a more rigid definition of trivialization, which we call an $H$-orientation; see \Cref{defn:subgroupfilt}.
However, in every case we consider the weaker definition will suffice. See \Cref{lem:old-is-new} and
\Cref{rem:old-is-new-at-2}.  In any case, if $H_1 \subset H_2 \subseteq \GG_n$ is a sequence
of subgroups will get inclusions $\kappa(H_2) \subseteq \kappa(H_1)$. 

At $n=2$ and $p=2$ we will focus on two subgroups of $\GG_2$. The first of these is $G_{48}$, a maximal
finite subgroup with the property that $E^{hG_{48}}$ is the $K(2)$-localization of the Hopkins-Miller spectrum
of topological modular forms. The other is $\GG_n^1$ defined using the determinant map $\det:\GG_n \to \ZZ_p^\times$.
See \Cref{defn:det-defined-here}. Namely, if $\mu \subseteq \ZZ_p^\times$ is the finite subgroup, then we get a map
\[
\xymatrix{
\GG_n \ar[r]^-{\det} &\ZZ_p^\times \ar[r] & \ZZ_p^\times/\mu \cong \ZZ_p
}
\]
and $\GG_n^1 \subseteq \GG_n$ is the kernel of this map. This gives a filtration of $\kappa_2$ at $p=2$
\begin{equation}\label{eq:sbgrp-filt-one-more-time}
\kappa_2 \supseteq \kappa(G_{48}) \supseteq \kappa(\GG_2^1)
\end{equation}
and then in  \Cref{thm:kappag21}, \Cref{thm:Qfirst}, and \Cref{thm:first-quotient} we compute the filtration quotients as
\begin{align}\label{eq:div-in-3pts}
\kappa(\GG_2^1)&\cong \Z/8 \oplus (\Z/2)^{ 2}\nonumber\\
\kappa(G_{48})/\kappa(\GG_2^1) &\cong \Z/2\\
\kappa_2/ \kappa(G_{48}) &\cong \Z/8.\nonumber
\end{align}
We then show that this filtration splits in \Cref{prop:kappaG48} and \Cref{thm:kappa}.  To help with this,
we compare this subgroup filtration with the descent filtration, the one arising from the
Adams-Novikov spectral sequence. See \Cref{defn:descent-filt} for the details on this filtration. We remark immediately
that this splitting is by no means canonical.  A roadmap and a more detailed synopsis is given after \Cref{defn:subgroup-filt}, 
near the beginning of  Part II and after we have all the definitions in place.  

Echoing the three-part division of $\kappa_2$ displayed in \eqref{eq:div-in-3pts}, we have three ways of producing
exotic elements in the Picard group. The first is a twisting construction, which appeared in \cite{GHMRPicard}
and was studied in some detail in \cite{Westerland}. Let $\KEn \subseteq \pi_0E^{h\GG_n^1}$ be in the kernel of
the Hurewicz map to $E_0E^{h\GG_n^1}$. Then we can form the multiplicative subgroup $1 + \KEn \subseteq (\pi_0E^{h\GG_n^1})^\times$.
Note the quotient group $\GG_n/\GG_n^1 \cong \ZZ_p$ acts on $1 + \KEn$.
Now if $\alpha \in 1 + \KEn$ and $\psi \in \ZZ_p$ is a topological generator
we can form a fiber sequence
\[
\xymatrix{
X(\alpha) \ar[r] & E^{h\GG_n^1} \ar[r]^-{\psi-\alpha} & E^{h\GG_n^1}.
}
\]
In \Cref{prop:twist-constr} we prove this assignment $\alpha \mapsto X(\alpha)$ defines an isomorphism
\[
H^1(\ZZ_p,1+\KEn) \cong \kappa(\GG_n^1).
\]
We then implement this isomorphism at $n=2$ and $p=2$. The group $\pi_0E^{h\GG_2^1}$ was
computed in \cite{BGH}; the difficulty here is to determine the action of $\GG_2/\GG_2^1$ on this group. 
It turns out to be trivial, and the final calculation $\kappa(\GG_2^1)$ is given in \Cref{thm:kappag21}. 

The second construction also uses the determinant map. At the prime $2$, $\mu = \{\pm 1\} \subseteq 
\ZZ_2^\times$ is a cyclic group of order $2$ and we have a homomorphism $\chi:\GG_n \to C_2$ defined
as the composition 
\[
\xymatrix{
\GG_n \ar[r]^-{\det} &\ZZ_2^\times \ar[r] & \ZZ_2^\times/(1+4\ZZ_2) \cong \mu = C_2.
}
\]
If $V$ is a virtual real representation of $C_2$ with one-point compactification $S^V$, $\GG_n$ acts on $S^V$ through this quotient, Giving $E \wedge S^V$ the diagonal $\GG_n$-action, we can form
\[
J(V) = (E \wedge S^V)^{h\GG_n}.
\]
If $V$ has virtual dimension zero and $C_2$ acts trivially on $H_0S^V \cong \ZZ$,
then $J(V) \in \kappa_n$. Now let $\sigma$ be the sign representation of $C_2$. We show in \Cref{thm:Qfirst}
and \Cref{thm:2Q} that if $n=2$
\[
J(2\sigma-2) \in \kappa(G_{48})
\]
is an element of order $2$ and generates the quotient group $\kappa(G_{48})/\kappa(\GG_2^1) \cong \ZZ/2$.
As a preliminary to all this, we devote
\Cref{sec:j-construction} to discussing the general categorical properties of this $J$-construction. Before getting
this far, we need to show $\kappa(G_{48})/\kappa(\GG_2^1)$ is no larger than $\ZZ/2$. This requires
a separate suite of ideas and uses the topological resolutions pioneered in \cite{ghmr} and made
explicit at $p=2$ in \cite{BobkovaGoerss}. So far these techniques are very special to height $2$.

Note that we might conjecture that $J(2\sigma-2)$ is a non-trivial element of $\kappa_n$ for all $n$ at $p=2$.
This is true for $n=1$ by a simple calculation and here we show it at $n=2$. It is possible to prove 
this for $n$ odd using \cite{HeardLiShi} and the fact that $\chi$ has a splitting. For larger even $n$, this would follow
if we  could show a certain cohomology class in $H^3(\GG_n,E_2)$ was non-zero. The precise
class at $n=2$ is given in \Cref{prop:keypropd2meansd3}. This is quite plausible, but we make no attempt to prove it here. 

The third technique uses the fact that the $K(n)$-local category has two dualities: Spanier-Whitehead 
duality $D_n$ and Gross-Hopkins duality $I_n$. These are related by the formula
\[
I_nX \cong D_nX \wedge I_n
\]
where $I_n$ is the Gross-Hopkins dual of the sphere. By \cite{GrossHopkins} and \cite{StrickGrossHop} we have
an equation 
\[
I_n = \Sigma^{n^2-n} S\langle \det \rangle \wedge P
\]
where $S\langle \det \rangle$ is a determinant sphere and $P \in \kappa_n$. Again see \eqref{eq:Gross-Hopkins-relation} 
and also \cite{BBGS}. The technique then is to find an $X$ so that we know $D_nX$ and $I_nX$, and then see
that the equation for $I_n$ forces $P \ne 0$. This is by now a classical idea, and has been used in 
\cite{BehrensModular,GHMRPicard,BGHS,HeardLiShi}, and probably elsewhere. In our case we take
$X = E^{hG_{48}}$ where we can use a combination of \cite{Pham,MahowaldRezkBCDual, BrunerGreenleesRognes, Stojanoska,TmfAt2} 
for $I_2E^{hG_{48}}$ and \cite{Bobkova_DTMF,BGHS}
for $D_2E^{hG_{48}}$. Then not only is $0 \ne P \in \kappa_2$, but $P$ generates
$\kappa_2/\kappa(G_{48}) \cong \ZZ/8$. See \Cref{thm:first-quotient}.

There is a drawback to this last technique. The spectrum $I_n$ is the Brown-Comenetz dual of the $n$th 
monochromatic  layer of the sphere and thus is a product of localization theory and Brown Representability.
That we know $E_\ast I_n$ at all relies on a deep result in the algebraic geometry of formal groups:
the identification of the dualizing sheaf of Lubin-Tate space, as in \cite{GrossHopkins2}. 
This is all very indirect, and gives no detailed information on the homotopy type of $I_n$ or, by extension,
of the spectrum $P$. We only know that $P$ is non-trivial because of its effects on much larger spectra. 
It is as if we've observed the perturbations of the  orbit of Uranus, but have not yet
discovered Neptune. Put another way, Gross-Hopkins duality remains a tantalizing mystery 
when the prime is low with respect to the height. For further thoughts on $P$, see \Cref{rem:as-close-as-we-get}. 

It is possible to ask how much all of this generalizes to higher heights. 
The filtration of \eqref{eq:sbgrp-filt-one-more-time} is rather short, but we suspect that this is an artifact of the fact that 
we are working at a relatively low height.  More generally, we hypothesize that we would need at least need know
$\kappa(\GG_n^1)$ and $\kappa(H)$ where $H$ runs over the entire lattice of finite subgroups of $\GG_n$.
At low heights, this lattice is very simple; for example at $n=p=2$, $G_{48}$ is the unique (up to conjugacy) maximal
finite subgroup of $\GG_2$, so it is sufficient to calculate $\kappa(G_{48})$. But knowing $\kappa(\GG_n^1)$ and $\kappa(H)$
for all finite $H$ would not be enough: even in our case the inclusion $\kappa(\GG_2^1) \subseteq \kappa(G_{48})$ 
is strict.

\setcounter{subsection}{1}

\subsection*{Acknowledgements}  The research for this paper has been supported by the National Science
Foundation through the grants DMS-2005627, DMS-1906227 and DMS-1812122, the Simons Foundation
through a Fellowship to Stojanoska and a Travel Grant to Goerss, and the L'Agence Nationale de la Recherche
through the grant ANR-16-CE40-0003 Chrok ``Chromatic homotopy and K-theory''.

This project began in 2015 during the Hausdorff Research Institute for Mathematics (HIM) Trimester 
Program ``Homotopy Theory, Manifolds, and Field Theories'', and was completed in 2022  during the HIM Trimester 
program ``Spectral Methods in Algebra, Geometry, and Topology''. We would like to thank HIM for the hospitality, with 
acknowledgement of support from the Deutsche  Forschungsgemeinschaft (DFG, German Research Foundation) under 
Germany's Excellence Strategy -- EXC-2047/1 -- 390685813.

Over this period we have worked on these ideas in many places. In addition to the Hausdoff Institute
we'd like to particularly thank the University of Colorado Boulder, the Isaac Newton Institute for
Mathematical Sciences in Cambridge, the Max Planck Institute for Mathematics in Bonn,
and the Institut de Recherche Math\'ematique Avanc\'ee at the Universit\'e de Strasbourg.

We are grateful to Mike Hill for insightful and supportive discussions throughout the duration of this project, and 
especially during the second HIM visit. 

Finally, we'd like to reaffirm our enduring debt to Mark Mahowald, his inspiration, his ideas, his generosity, and his
friendship. Long ago, Mark had very concrete ideas on how to produce elements of $\kappa_2$ at $p=2$ as the 
localization of finite complexes. We don't know which elements of $\kappa_2$ can be constructed 
this way, and this question remains very interesting. Any affirmative answer would tell us a great deal about the global 
structure of stable homotopy theory.

\subsection*{Notation} In this paper, concepts defined one place can reappear in others, often
many pages distant. We collect here some of the less standard items, along with references to where they are defined.

\subsubsection*{Subgroups and cohomology classes of the Morava stabilizer group $\GG_n$}
\begin{enumerate}

\item $\det \colon \GG_n \to \ZZ_p^\times$ is
defined in \Cref{defn:det-defined-here}

\item $\zeta:\GG_n \to \ZZ_p$ is defined in \Cref{defn:zeta}

\item $\chi: \GG_n \to \ZZ/2$ is defined in \Cref{defn:chi-defined}

\item $S\GG_n$, $\GG^1_n$, and $\GG_n^0$ are the kernels of $\det$, $\zeta$ and $\chi$ respectively

\item $\zeta \in H^1(\GG_n,E_0)$ and $\wchi \in H^2(\GG_n,E_0)$ are
defined in \Cref{defn:what-is-wchi-anyway1}

\item $G_{48}$, $G_{24}$ and other finite subgroups of $\GG_2$ are defined in \Cref{defn:finite-subgrps}

\item $e$ and $k$ in $H^\ast(\GG^1_2,E_0)$ appear in \Cref{fig:HFPSSG2}
\end{enumerate}

\subsubsection*{Subgroups of $\kappa_n$}
\begin{enumerate}

\item $\kappa(K)$ and orientations are discussed in \Cref{defn:subgroupfilt} and \Cref{lem:incl-kappa-K}

\item $\kappa_{n,r}$ is defined in \Cref{defn:descent-filt}

\item $\phi_r \colon\kappa_{n,r} \longr E_r^{r,r-1}(S^0)$ is defined in \Cref{defn:descent-filt}

\item $\kappa_r(K) = \kappa(K) \cap \kappa_{n,r}$

\item $\phi_r^1$ (we need only $\phi_3^1$) is defined in \eqref{eq:phi31}
\end{enumerate}

\subsubsection*{Various constructions and concepts} 

\begin{enumerate}

\item algebraic maps of spectra are defined in \Cref{def:alg-mor}

\item $\KEn \subseteq \pi_0E^{h\GG_n^1}$ is defined in \eqref{eq:lambda-defined}

\item $\KEn_s$ is defined in \Cref{def:kens-filt}

\item $X(\alpha) \in \kappa(\GG_n^1)$ is defined in \Cref{defn:super-e-defn}

\item $J(q,f,K)$ and $J(V)$ are defined in \eqref{eq:defnJQFK} and  \Cref{exam:rep-spheres}

\item $\bX_{\leq M}$ and $\bX_K^M$ are defined in \Cref{defn:trunc-towers-ss} and \Cref{defn:trunc-towers-rel}

\item $\RP^n_k$ is defined in \eqref{eq:thom-m}
\end{enumerate}

\part{$K(n)$-Local Results for General $n$ and $p$}

This first part of the paper focuses on machinery and results used to study the $K(n)$-local category, its Picard group,
and $\kappa_n$ for general heights $n$ and prime $p$. We will specialize to the case $n=p=2$ in the second part.

 % !TEX root = pic-master.tex

\section{The $K(n)$-local category}\label{sec:background-kn}

Detailed introductions to the $K(n)$-local category can be found in many places; for example,
much of the foundations can be found in \cite{HovStrick}. A precise summary of what is needed here can be
found in Section 2 of \cite{BGH} and we will use the language and notation of that reference.  Here we give
a short summary to establish the context. 

We begin with the selection of a formal group $\Gamma_n$ of height $n$. We will assume $\Gamma_n$ is defined
over $\FF_p$ and that for any extension $\FF_{p^n} \subseteq \FF_q$ of finite fields the inclusion of automorphism
groups
\[
\Aut(\Gamma_n/\FF_{p^n}) \subseteq \Aut(\Gamma_n/\FF_q)
\]
is an equality. Examples include the Honda 
formal group of height $n$ and the formal group arising from the standard supersingular elliptic curves at $p=2$ and 
$p=3$. See \cite{StrickGrossHop} and \cite{HennCent}.

We will write $\GG_n$ for the  automorphisms of the pair $(\FF_{p^n},\Gamma_n)$;  there is a semi-direct
product decomposition
\[
\GG_n \cong \Aut(\Gamma_n/\FF_{p^n}) \rtimes \Gal(\FF_{p^n}/\FF_p).
\]
It is customary to define $\SS_n = \Aut(\Gamma_n/\FF_{p^n})$ and we will write $\Gal = \Gal(\FF_{p^n}/\FF_p)$.

We define $K(n)$ to be a $2$-periodic complex orientable homology theory with
$K(n)_\ast \cong \FF_{p^n}[u^{\pm 1}]$ with $u$ in degree $-2$ and associated formal group $\Gamma_n$. This
homology theory has the same Bousfield localization functor as the $2(p^n-1)$-periodic version 
historically labelled $K(n)$. The 2-periodic version used here is better related to Morava $E$-theory.
See \eqref{eq:Morava-E-theory-def}.

The $K(n)$-local category has a symmetric monoidal structure with product given by
\[
X \wedge Y = L_{K(n)}(X \wedge Y). 
\]
As this equation indicates, we adorn the smash product with the localization only when emphasis is needed; normally,
we will leave it understood. 

The most important and most basic algebraic invariant of the $K(n)$-local category is given by Lubin-Tate or Morava
$E$-theory. In fact, $E$-theory is a functor that associates a $K(n)$-local $\Einf$-ring spectrum to a pair of a perfect field and a formal group law over it \cite{GHModuli}. We will write 
\begin{equation}\label{eq:Morava-E-theory-def}
E = E_n = E(\FF_{p^n},\Gamma_n).
\end{equation} 
There is a non-canonical isomorphism
\[
E_\ast \cong \WW[[u_1\ldots,u_{n-1}]][u^{\pm 1}],
\]
where $\WW = W(\FF_{p^n})$ is the Witt vectors on $\FF_{p^n}$. The power series rings is in degree zero and $u$
is in degree $-2$. This is a Landweber exact complex orientable  theory with formal group given by a universal
deformation of $\Gamma_n$. Our version of Morava $K$-theory is chosen so that there is a map $E \to K(n)$;
on coefficients this map is given by the quotient by the maximal ideal $\mathfrak{m} = (p,u_1,\ldots,u_{n-1})$.

We define
\[
E_\ast X = \pi_\ast L_{K(n)} (E \wedge X).
\]
Since the formal group for $E$ is the universal deformation of $\Gamma_n$, the group $\GG_n$ acts on $E_\ast$ and,
in fact, this lifts to an action on $E$ through maps of $\Einf$-ring spectra. 
Thus $E_\ast X$ is a continuous $E_\ast$-module with a continuous action by $\GG_n$; more precisely, it is a Morava
module in the sense of Definition~5.3.20 of \cite{BarthelBeaudry}.
We will write $\Mor$ for the category of Morava modules.

A fundamental fact is that the action of $\GG_n$ on the right factor of $E \wedge E=L_{K(n)}(E \wedge E)$
defines an isomorphism of Morava modules
\begin{equation}\label{eq:basic-eg}
 E_\ast E \cong \map(\GG_n,E_\ast) 
\end{equation}
to the ring of continuous functions. See \cite{StrickGrossHop},  \cite{DH}, or \cite{HennCent}. From this it follows,
as in \cite{DH}, that for many reasonable spectra -- including all the spectra of this paper -- the $K(n)$-local
Adams-Novikov Spectral Sequence has the form
\begin{equation}\label{eq:ANSS-first-app}
E_2^{s,t}(X) = H^s(\GG_n,E_tX) \Longrightarrow \pi_{t-s}L_{K(n)}X. 
\end{equation}

{\bf Note well:} Here, and throughout the paper, group cohomology will be understood
to be {\it continuous} cohomology whenever this makes sense. Likewise, maps of sets will be understood to be continuous whenever that makes sense.
We will have more to say about this spectral sequence and its construction in \Cref{subsec:ANSS}. 

\begin{rem}\label{rem:mormodulehom}
The Devinatz-Hopkins fixed point theory \cite{DH} supplies a functorial assignment $K \mapsto E^{hK}$
from closed subgroups of $\GG_n$ to $K(n)$-local $\Einf$-ring spectra. We have that
$L_{K(n)}S^0 \simeq E^{h\GG_n}$ and if $K = \{e\}$ then $E^{hK} = E$. This construction has
the following properties. First, there is an isomorphism of
Morava modules
\begin{equation}\label{eq:basic-eg-2}
E_\ast E^{hK} \cong \map(\GG_n/K,E_\ast)
\end{equation}
and, second, for dualizable $K(n)$-local spectra $X$, there is a homotopy fixed point or Adams-Novikov spectral sequence 
\begin{equation}\label{eq:ANSS-first-app-2}
E_2^{s,t}(K, X) = H^s(K,E_tX) \Longrightarrow \pi_{t-s}(E^{hK} \wedge X). 
\end{equation} 
The spectral sequence of \eqref{eq:ANSS-first-app} is the case of $K = \GG_n$. Further, the spectral sequence \eqref{eq:basic-eg-2} for any $X$ is a module over the analogous spectral sequence for $X=S^0$, a fact which will be exploited several times in this paper.
\end{rem}

 % !TEX root = pic-master.tex

\section{Orientations and Filtrations}

Let $K \subseteq \GG_n$ be a closed subgroup. In this section we introduce the concept of an $E^{hK}$-orientation of
an exotic element in the Picard group. This allows us to introduce a decomposition on $\kappa_n$,  
which  reflects the subgroup lattice of $\GG_n$. We also introduce a second, more classical, descent filtration which, in 
essence,  comes from the Adams--Novikov filtration.
 
\subsection{Orientations} \label{sec:subgroupfilt}

We begin with a basic definition. 

\begin{defn}\label{defn:Gninvariantgen} Let $X \in \kappa_n$. Then a \emph{$\GG_n$-invariant generator}
for $X$  is a choice of a $\GG_n$-invariant element $\iota_X$ in $E_0X$ which generates $E_\ast X$ as an
$E_\ast$-module. 
\end{defn}

\begin{rem}\label{rem:really-basic-siso}
For $X \in \kappa_n$, a $\GG_n$-invariant
generator $\iota_X$ determines, and is determined by, a choice of isomorphism of Morava modules $\varphi_\ast: E_\ast \to E_\ast X$. This $\varphi_\ast$ defines an isomorphism
\[
H^\ast(\GG_n,E_\ast) \cong H^\ast(\GG_n,E_\ast X).
\]
Two choices of $\GG_n$-invariant generators differ by an element of $\ZZ_p^\times \cong H^0(\GG_n,E_0)^\times$.
The latter isomorphism was historically known to the experts; one proof is in \cite[Lemma 1.33]{BobkovaGoerss}.
\end{rem}

\begin{defn}\label{defn:subgroupfilt}
Let $K \subseteq \GG_n$ be a closed subgroup, and suppose that $X \in \kappa_n$ is an exotic invertible spectrum. 
\begin{enumerate}
\item We say that $z \in \pi_0(E^{hK} \wedge X)$ is \emph{$\GG_n$-invariant $E^{hK}$-orientation}, or briefly an \emph{$E^{hK}$-orientation} of $X$ if $z$ maps to a
$\GG_n$-invariant generator under the map
\[
\pi_0(E^{hK} \wedge X) \to \pi_0 (E\wedge X) = E_0X.
\]
\item Define $\kappa(K) \subseteq \kappa_n$ as the subgroup of $X\in \kappa_n$ such that $X$ has an $E^{hK}$-orientation $z$ as in (1).
\end{enumerate}
%Then $\kappa(K) \subseteq \kappa_n$ is the subgroup
%of invertible spectra $X$ such that there is a class  $z \in \pi_0(E^{hK} \wedge X)$ which maps to a
%$\GG_n$-invariant generator under the map
%\[
%\pi_0(E^{hK} \wedge X) \to \pi_0 (E\wedge X) = E_0X.
%\]
%We call the class $z$ an $E^{hK}$-\emph{orientation} of $X$, or simply an \emph{orientation}
%if $K$ and $X$ are understood.
\end{defn}
\begin{rem}
A more natural definition for an $E^{hK}$-orientation for $X$ might be to require an equivalence $E^{hK}\wedge X \simeq E^{hK}$ of $E^{hK}$-modules. We will see shortly in \Cref{prop:detect-kappa-ANSS-bis} that our definition of a $\GG_n$-equivariant orientation gives such an equivalence. The converse is not necessarily true, but holds under a mild technical assumption, as proved in \Cref{lem:old-is-new}. The stronger notion of orientation is needed already in \Cref{prop:orientations-nat-4}, and is of key importance in the proof of \Cref{lem:orientations-nat-2}.
\end{rem}

\begin{rem}[Subgroup filtrations]\label{lem:incl-kappa-K}
If $K_1 \subseteq K_2$, then $\kappa(K_2) \subseteq \kappa(K_1)$, so the assignment $K\mapsto \kappa(K)$
defines a filtered diagram of subgroups of $\kappa_n$. In particular, any nested sequence of  closed subgroups 
\[
 \{e\} = K_0 \subseteq K_1 \subseteq \cdots \subseteq K_k =\GG_n ,
\] 
gives an associated filtration
\[
\kappa_n= \kappa(\{e\}) \supseteq \kappa(K_{1}) \supseteq \cdots \supseteq 
\kappa(\GG_n) =\{L_{K(n)}S^0\} \ ,
\]
We call this a \emph{subgroup filtration} of $\kappa_n$. 
\end{rem}

In order to compute the subgroups $\kappa(K)$ it is necessary to be able to find $E^{hK}$-orientations of elements $X$ 
of $\kappa_n$. The goal of the remainder of this section is to give recognition principles for $E^{hK}$-orientations which 
use the homotopy fixed point spectral sequence \eqref{eq:ANSS-first-app-2}. 

\begin{prop}\label{prop:detect-kappa-ANSS} Let $X \in \kappa_n$ and let $\iota_X \in E_0X$ be a
$\GG_n$-invariant generator. Then $X$ is in $\kappa_n(K)$ if and only if $\iota_X$ is a permanent cycle
in the homotopy fixed point spectral sequence
\[
H^s(K,E_tX) \Longrightarrow \pi_{t-s}(E^{hK} \wedge X).
\]
\end{prop}

\begin{proof} First, suppose that $\iota_X$ is a permanent cycle detecting a
class $z \in \pi_0(E^{hK}\wedge X)$. Then since $\iota_X$ is a $\GG_n$-equivariant of $E_\ast X$, the class $z$ is an
$E^{hK}$-orientation.

Conversely, assume $X \in \kappa(K)$. Then we have $z \in \pi_0(E^{hK} \wedge X)$, whose image in $E_0X$ is a $
\GG_n$-invariant generator, so it must be a unit multiple of $\iota_X$. Thus $z$ is represented by a permanent cycle 
equaling $a\iota_X$ for some $a \in \ZZ_p^\times$. This implies that $\iota_X$ itself is also a permanent cycle. 
\end{proof}

We now examine some consequences of having an $E^{hK}$-orientation.  

\begin{prop}\label{prop:detect-kappa-ANSS-bis} Let $z \in \pi_0(E^{hK} \wedge X)$ be an $E^{hK}$-orientation
for $X$. Then the extension of $z$ to an $E^{hK}$-module map
\[
\varphi_z\colon E^{hK} \longrightarrow E^{hK} \wedge X
\]
is an equivalence.
\end{prop}

\begin{proof} As mentioned in \Cref{rem:mormodulehom}, the homotopy fixed point spectral sequence \eqref{eq:ANSS-first-app-2}
\[
E_2^{s,t}(K,S^0) = H^s(K,E_t) \Longrightarrow \pi_{t-s}E^{hK}
\]
for $S^0$ acts on the spectral sequence $E_2^{s,t} (K, X)$,
and calculates the action of $E_\ast^{hK}$ on $E_\ast^{hK}(X)=\pi_\ast(E^{hK} \wedge X)$.
Suppose $X \in \kappa_n(K)$, and suppose that we have chosen a $\GG_n$-invariant generator $\iota_X  \in E_0X$. If $d_q(\iota_X) = 0$ for $q < r$,
then the map $a \mapsto a\iota_X$ defines an isomorphism 
\[
E_r^{\ast,\ast}(K,S^0) \xrightarrow{\cong} E_r^{*,*} (K,X).
\]
In particular, if $\iota_X$ is a permanent cycle detecting the class
$z \in \pi_0(E^{hK} \wedge X)$, then the module map $E^{hK} \to E^{hK} \wedge X$ extending $z$
is an equivalence. 
\end{proof} 

The equivalence of \Cref{prop:detect-kappa-ANSS-bis} also works well in $E$-homology.

\begin{prop}\label{prop:orientations-nat-4} Let $z \in \pi_0(E^{hK} \wedge X)$ be an $E^{hK}$-orientation
for $X$ and let $\varphi_\ast:E_\ast \to E_\ast X$
be the Morava module isomorphism from \Cref{rem:really-basic-siso}. Then there is a commutative diagram of Morava modules
\begin{equation*}
\xymatrix{
E_\ast E^{hK} \ar[d]_{E_\ast (\varphi_z)} \ar[r]^-\cong & \map(\GG/K,E_\ast)\ar[d]^{\map(\GG/K,\varphi_\ast)} \\
E_\ast (E^{hK} \wedge X) \ar[r]_-\cong & \map(\GG/K,E_\ast X).
}
\end{equation*}
\end{prop}
%\vesnachanges{\begin{rem}Having an equivalence $E^{hK} \simeq E^{hK}\wedge X$ gives a $\GG_n$-equivariant map on the left hand side of this diagram, so by the horizontal isomorphisms, that corresponds to an equivariant map on the right hand side. But to know that that right-hand vertical map is of the form claimed, we need a $\GG_n$-equivariant orientation.\end{rem}}

\begin{proof} The horizontal equivalences are obtained as the adjoint of the composition
\[
\GG_n/K \times E_\ast (E^{hK} \wedge X) \to E_\ast (E \wedge X) \to E_\ast X
\]
where the first map is given by the action map $\GG_n/K \times E^{hK} \to E$ and the second by
the multiplication $E\wedge E \to  E$. Now it is a matter of chasing around the diagram. 
\end{proof}

%\Cref{prop:detect-kappa-ANSS-bis} suggests an alternative, potentially more natural, characterization of the elements 
%of $\kappa(K)$ which only requires an equivalence of $E^{hK}$-modules
%$E^{hK} \to E^{hK} \wedge X$. 

We can now give a partial converse to \Cref{prop:detect-kappa-ANSS-bis}.% i.e. a criterion for when having an equivalence $E^{hK}\wedge X \simeq E^{hK}$ is sufficient for producing a $\GG_n$-invariant $E^{hK}$-orientation of $X$. 

\begin{lem}\label{lem:old-is-new} Suppose $X \in \kappa_n$ and that we have an $E^{hK}$-module equivalence
\[
f \colon E^{hK} \xrightarrow{\simeq }E^{hK} \wedge X,
\] 
not assumed to be induced by a $\GG_n$-equivariant $E^{hK}$-orientation of $X$.
Suppose further that the edge homomorphism
\[
\pi_0E^{hK} \longrightarrow H^0(K,E_0)
\]
is onto. Then $X$ has a $\GG_n$-equivariant $E^{hK}$-orientation, and $X \in \kappa(K)$. 
\end{lem}

\begin{proof} We need to find a class $z \in \pi_0(E^{hK} \wedge X)$ which maps to a $\GG_n$-invariant
generator of $E_0X$. The $E^{hK}$-module equivalence $f$ determines a $K$-invariant
generator $f_\ast(1) = y \in E_0X$. Since $X \in \kappa_n$, we also can choose a $\GG_n$-invariant generator $x$
of $E_0X$. There is an element $a \in (E_0)^\times$ so that $ay = x$, since both are $E_0$-module
generators. Since $x$ and $y$ are both $K$-invariant generators of a free module, $a$ is also $K$-invariant.
Choose an element $\alpha \in \pi_0E^{hK}$ which maps to $a$ and
let $\phi:E^{hK} \to E^{hK}$ be the $E^{hK}$-module map extending $\alpha$. We then have a new
$E^{hK}$-module equivalence $g = (\phi \wedge 1_X)f: E^{hK} \to E^{hK} \wedge X$ and $g_*(1) = x$ as needed.
\end{proof}  

\begin{ex}\label{rem:old-is-new-at-2} The hypothesis on $K$ in \Cref{lem:old-is-new} is equivalent to the
statement that $E_\infty^{0,0} = E_2^{0,0}$ in the homotopy fixed point spectral sequence
\[
E_2^{s,t}(K,S^0)=H^s(K,E_t) \Longrightarrow \pi_{t-s} E^{hK}. 
\]
This happens in all the examples for which we have complete calculations, including all the examples
at $n=p=2$ we will consider in this paper. 
\end{ex}

The following result will allow for some flexibility in later sections, as it tells us that the Galois group will not cause 
complications when it comes to orientations.

\begin{prop}\label{prop:galois-doesnt-matter}Let $K \subseteq \GG_n$ be a closed subgroup and
$K_0 = \SS_n \cap K$. Then the inclusion $\kappa_n(K) \subseteq \kappa_n(K_0)$ is an equality. 
\end{prop}

\begin{proof} 
Let
$G = K/K_0 \subseteq \GG_n/\SS_n \cong \Gal.
$
For any dualizable spectrum $X$, which in particular includes invertible spectra, the same argument as for
\cite[Lemma 1.32]{BobkovaGoerss} shows that
\[
\WW \otimes_{\WW^G} E_2^{s,t}(K,X) \cong E_2^{s,t}(K_0,X).
\]
But for such $X$, the differentials in $E_r^{s,t}(K_0,X) $
are automatically $\WW$-linear since the spectral sequence is a module over the spectral sequence
$E_r^{s,t}(K_0,S^0)$  and the latter has $\WW$-linear differentials. This follows from Remark 1.35 of
\cite{BobkovaGoerss} and the fact  that the unit $\ZZ_p \to \WW$ is \'etale. Thus, for any dualizable
spectrum $X$, we have
\[
\WW \otimes_{\WW^G} E_r^{s,t}(K,X) \cong E_r^{s,t}(K_0,X)
\]
and the isomorphism preserves differentials. Now apply \Cref{prop:detect-kappa-ANSS}. 
\end{proof}

\subsection{The naturality of orientations}
The material in this section, specifically the notion of an algebraic map introduced in \cref{def:alg-mor} and their 
relationship with orientations presented in \cref{lem:orientations-nat-2}  will be used in an essential way in 
\cref{sec:untwistduality}, specifically in \cref{prop:untwist-the-TDSS-1}.

For a closed subgroup $K$ of $ \GG_n$, let $X\in \kappa(K)$, and let $z \in \pi_0(E^{hK} \wedge X)$ be a chosen $E^{hK}$-orientation of $X$.
Let $\varphi_z:E^{hK} \to E^{hK} \wedge X$ be the equivalence of $E^{hK}$-modules obtained by extending $z$. 

If $g \in \GG_n$ and $H \subseteq gKg^{-1}$,
the composite
\[
\xymatrix{E^{hK} \ar[r]^-{g} &  E^{h(gKg^{-1})} \ar[r] & E^{hH}}
\]
gives a map which we also call $g:E^{hK} \to E^{hH}$. The following result is immediate from the definitions.

\begin{lem}\label{lem:orientations-nat-1} Let $g \in \GG_n$ and $H \subseteq gKg^{-1}$.
If $z \in \pi_0(E^{hK} \wedge X)$ is an $E^{hK}$-orientation then 
\[
gz = (g \wedge 1)_\ast (z) \in \pi_0(E^{hH} \wedge X)
\] 
is an $E^{hH}$-orientation and the following
diagram commutes
\begin{align}\label{eq:invgen}
\xymatrix@C=45pt{
E^{hK} \ar[d]_-{\varphi_z}^-\simeq \ar[r]^g & E^{hH} \ar[d]^-{\varphi_{gz}}_-\simeq\\
E^{hK} \wedge X \ar[r]^{g\wedge 1_X}  & E^{hH} \wedge X\ . 
}
\end{align}
\end{lem}

In this result, the map $g$ is induced from an element in the Morava stabilizer group. We would like to
expand the class of maps for which we have a similar diagram. Our key result is \Cref{lem:orientations-nat-2}, but to 
state and prove it, we first need to develop some language.

First, here is some material originally due to Devinatz--Hopkins \cite{DH}; it was also
reviewed in \cite{ghmr} and \cite{BBGS}. For any dualizable $X$ in the $K(n)$-local category,
there is an isomorphism of Morava modules
\begin{equation}\label{eq:mormoduleiso}
E_\ast (E^{hK}\wedge X) \cong \map(\GG_n/K,E_\ast X),
\end{equation}
where $\GG_n$ acts on the right-hand side by conjugation: $(g\phi)(hK) = g\phi(g^{-1}hK)$.
We also have an isomorphism of Morava modules
\begin{equation}\label{eq:alg-vs-sets}
\map(\GG_n/K,E_\ast X) \cong \Hom_{\ZZ_p}(\ZZ_p[[\GG_n/K]],E_\ast X),
\end{equation}
where the right-hand side is the group of continuous $\ZZ_p$-module maps. 

We let $X=S^0$. Given closed subgroups $K_1, K_2$ of $\G_n$, let $h:\ZZ_p[[\GG_n/K_2]] \to \ZZ_p[[\GG_n/K_1]]$
be a continuous $\GG_n$-module map. Using the  identification of \eqref{eq:alg-vs-sets} we obtain a map
\[
\Hom(h,E_\ast)\colon \map(\GG_n/K_1,E_\ast) \longr \map(\GG_n/K_2,E_\ast)
\]
of Morava modules and so we get a homomorphism 
\begin{align}\label{eq:algebra-to-mor}
\Psi \colon \Hom_{\ZZ_p[[\GG_n]]}(\ZZ_p[[\GG_n/K_2]],\ZZ_p[[\GG_n/K_1]]) &\longr\\
\Hom_{\Mor}(\map(\GG_n/K_1,E_\ast),&\map(\GG_n/K_2,E_\ast)).\nonumber
\end{align}
The map $\Psi$ is always an injection; however, it is not always surjective. For example, if $K_1 = \GG_n$; then,
\[
\Hom_{\ZZ_p[[\GG_n]]}(\ZZ_p[[\GG_n/K_2]],\ZZ_p) \cong \ZZ_p,
\]
while
\[
\Hom_{\Mor}(E_\ast,\map(\GG_n/K_2,E_\ast)) \cong  H^0(K_2,E_0) \cong E_0^{K_2}. 
\]
If $n > 1$ and $K_2 = \{e\}$, then $E_0$ is a power series ring strictly containing $\ZZ_p$. More critically
for us, if $p=2$ and $K_2$ is finite, then  $E_0^{K_2}$ is typically also a power series ring larger than $\ZZ_2$. 

This discussion suggests that we single out a class of Morava module maps
\[
E_\ast E^{hK_1} \cong \map(\GG_n/K_1, E_*) \longr \map(\GG_n/K_2, E_*)\cong E_\ast E^{hK_2}
\]
that includes those the image of $\Psi$; that is, those that can be written $\Hom(h,E_\ast)$ for some continuous
$\GG_n$-map $h\colon \ZZ_p[[\GG_n/K_2]] \to \ZZ_p[[\GG_n/K_2]]$. We will also need to incorporate the periodicity of 
$E^{hK}$ into the picture. Thus, we start by a brief discussion on the interplay of Morava modules and periodicities.

\begin{defn}\label{defn:alg-periodicity}  Let $K \subseteq \GG_n$ be a closed subgroup. An {\it algebraic periodicity 
class} for $K$ is a unit
\[
d \in H^0(K,E_*)
\]
in the ring of invariants. We say $E^{hK}$ has {\it algebraic period} $k$ if $k$ is the smallest positive integer
for which there is an algebraic periodicity class in degree $k$, i.e. $d \in H^0(K,E_k)$.
\end{defn}

Let $K \subseteq \GG_n$ be a closed subgroup, $X$ be a spectrum, and suppose we have an element
$d \in H^0(K,E_k)$. Then for all $m$ we get an induced map of Morava modules
\begin{align*}
f_d^X \colon \map(\GG_n/K,E_mX) &\longr \map(\GG_n/K,E_{m+k}X)\\
\varphi&\longmapsto \psi
\end{align*}
with
\[
\psi(gK) =(gd) \varphi(gK) .
\]
This is a map of $\map(\GG_n/K,E_0)$-modules in the category of Morava modules. \Cref{defn:alg-periodicity} makes the following result obvious, since if $d$ is invertible, then $f_{d^{-1}}^X$ is inverse to $f_d^X$.

\begin{lem} Let $d \in H^0(K,E_k)$ be a periodicity class, then for all $m$ and all $X$, the map
\[
f_d^X: \map(\GG_n/K,E_mX) \longr \map(\GG_n/K,E_{m+k}X)
\]
is an isomorphism. 
\end{lem}

\begin{rem} It could happen that the only algebraic periodicity classes are in $H^0(K,E_0)$, in which case $E^{hK}$ doesn't have an algebraic period. This is the case for $K = \GG_n$ itself, for example. We are more interested in this notion when $K$ is a finite subgroup of $\GG_n$. Indeed, if $K$ is finite
of order $m$ then 
\[
d = \prod_{g \in K} gu^{-1} \in H^0(K,E_{2m})
\]
is an algebraic periodicity class, although perhaps not one of minimal positive degree.
\end{rem} 

\begin{rem} There is a corresponding notion of a topological periodicity class, namely a unit $x \in \pi_*E^{hK}$ in positive dimension $k$. Then the induced $E^{hK}$-module map $\Sigma^kE^{hK} \to E^{hK}$ is an equivalence of $E^{hK}$-module spectra.
The corresponding notion of topological period is the minimal $k$ for which such an $x$ exists.  An algebraic periodicity class determines a topological 
periodicity class if and only if it is a permanent cycle in the homotopy fixed point spectral sequence; hence, the algebraic
and topological periods can and often do differ. For example, $E^{hC_2}$ at height one (which is the  $2$-completion of the real $K$-theory spectrum $KO$) has algebraic period $4$ and topological period $8$. At height two and the prime
2, $E^{hG_{48}}$ has algebraic period $24$ and topological period $192 = 8\cdot 24$. (Details on the subgroup  $G_{48}$
can be found in \Cref{defn:finite-subgrps}.)
\end{rem}

For any spectrum $X$, an algebraic periodicity class $d \in H^0(K,E_k)$ determines an isomorphism
\[
\xymatrix{
P_d^X: E_\ast( \Sigma^k E^{hK}\wedge X )\ar[r]^-\cong & \map(\GG_n/K,E_\ast X)
}
\]
as the composition
\[
\xymatrix{
E_\ast( \Sigma^k E^{hK}\wedge X) \ar[r]^-\cong & \map(\GG_n/K,E_{\ast-k}X) \ar[r]^-{\cong}_-{f_d^X} & \map(\GG_n/K,E_\ast X).
}
\]

\begin{rem}\label{rem:fXPXtensor}
We have isomorphisms
\[\xymatrix{
E_*E^{hK}\otimes_{E_*}E_* X  \ar[d]^-{\cong} \ar[r]^-{\cong} & E_*(E^{hK}\wedge X)  \ar[d]^-{\cong} \\
 \map(\GG_n/K, E_*)\otimes_{E_*}E_*X \ar[r]^-{\cong}  & \map(\GG_n/K, E_*X) 
}\]
Using these identifications and letting $f_d=f_d^{S^0}$ and $P_d =P_d^{S^0}$, we have that $f_d^X  = f_d \otimes E_*X$ and $P_d^X \cong P_d \otimes E_*X$.
\end{rem}

We now single out a particularly useful set of maps.

\begin{defn}\label{def:alg-mor} Let $K_1,K_2$ be closed subgroups of $\GG_n$, and let $d \in H^0(K_2,E_k)$ be an algebraic periodicity class for $K_2$. A map
$f: E^{hK_1} \to \Sigma^k E^{hK_2}$ of spectra is {\it algebraic} for $d$ if the induced
map $g$ of Morava modules
\[
\xymatrix{
\map(\GG_n/K_1,E_\ast) \ar[r]^g& \map(\GG_n/K_2,E_\ast)
}
\]
defined by the commutative square
\[
\xymatrix{
E_\ast E^{hK_1} \ar[r]^-{E_\ast f} \ar[d]_\cong & E_\ast \Sigma^k E^{hK_2} \ar[d]_\cong^{P_d}\\
\map(\GG_n/K_1,E_\ast) \ar[r]_g& \map(\GG_n/K_2,E_\ast)
}
\]
is in the image of the map $\Psi$ of \eqref{eq:algebra-to-mor}. Thus, $g=\Hom(h,E_\ast)$ for some continuous
$\GG_n$-module map
\[
h\colon \ZZ_p[[\GG_n/K_2]] \longr \ZZ_p[[\GG_n/K_1]].
\]
\end{defn} 

The next result mixes orientations and algebraic maps, so requires some set-up. We fix the following data
\begin{enumerate}

\item a closed subgroup $K \subseteq \GG_n$ and element $X \in \kappa(K)$;

\item an $E^{hK}$-orientation $z \in \pi_0(E^{hK} \wedge X)$; 

\item two elements $g_i \in \GG_n$, for $1\leq i \leq 2$ and two closed subgroups $H_i \subseteq  g_iKg_i^{-1}$; and,

\item an algebraic periodicity class $d \in H^0(H_2,E_k)$. 
\end{enumerate} 

Then by \Cref{lem:orientations-nat-1} we have induced orientations  $g_iz \in  \pi_0(E^{hH_i} \wedge X)$. Let
\[
\varphi_i = \varphi_{g_iz} \colon E^{hH_i} \longr E^{hH_i} \wedge X
\]
be the induced $E^{hH_i}$-module equivalences. In this context, we have the following result. 

\begin{prop}\label{lem:orientations-nat-2} Let $f \colon E^{hH_1} \to \Sigma^k E^{hH_2}$ be any map which is algebraic for $d$. Then the following diagram of Morava modules commutes
\begin{equation}\label{eq:orientations-nat-3}
\xymatrix@C=45pt{
E_\ast E^{hH_1} \ar[d]_{E_\ast(\varphi_1)} \ar[r]^-{E_\ast f} & E_\ast \Sigma^kE^{hH_2}\ar[d] \ar[d]^{E_\ast (\Sigma^k\varphi_2)}\\
E_\ast (E^{hH_1} \wedge X) \ar[r]_-{E_\ast (f \wedge 1_X)} & E_\ast (\Sigma^k E^{hH_2} \wedge X).
}
\end{equation} 
\end{prop}

\begin{proof} By \Cref{prop:orientations-nat-4}, the induced $H_i$-orientations for $X$ give us commutative diagrams
\begin{equation}\label{eq:orientations-nat-4}
\xymatrix{
E_\ast E^{hH_i} \ar[d]_{E_\ast (\varphi_{z_i})} \ar[r]^-\cong & \map(\GG/H_i,E_\ast)\ar[d]^{\map(\GG/H_i,\varphi_\ast)} \\
E_\ast (E^{hH_i} \wedge X) \ar[r]_-\cong & \map(\GG/H_i,E_\ast X).
}
\end{equation}

We now use the hypothesis that $f$ is an algebraic map, as then \Cref{def:alg-mor} gives a corresponding map $g$ with
\[
g = \Hom(h,E_\ast)
\]
for some continuous $\GG_n$-module map $h:\ZZ_p[[\GG_n/H_2]] \to \ZZ_p[[\GG_n/H_1]]$. 
Using \cref{rem:fXPXtensor}, we also have the following commutative diagram
\[
\xymatrix{
E_\ast (E^{hH_1}\wedge X) \ar[r]^-{E_\ast (f\wedge X)} \ar[d]_\cong &
E_\ast (\Sigma^k E^{hH_2}\wedge X) \ar[d]_\cong^{P_d^X}\\
\map(\GG_n/H_1,E_\ast X) \ar[r]_-{g_X}& \map(\GG_n/H_2,E_\ast X)
}
\]
with $g_X=\Hom(h,E_\ast X)$. 

Thus,  we have an isomorphism from the diagram of \eqref{eq:orientations-nat-3}
to the diagram
\[
\xymatrix@C=60pt{
\Hom(\ZZ_p[[\GG_n/H_1]],E_\ast )  \ar[r]^-{\Hom(h,E_\ast) } \ar[d]_{\Hom(\GG/H_1,\varphi_\ast)}
&\Hom(\ZZ_p[[\GG_n/H_2]],E_\ast ) \ar[d]^{\Hom(\GG/H_2,\varphi_\ast)}\\
\Hom(\ZZ_p[[\GG_n/H_1]],E_\ast X)  \ar[r]_-{\Hom(h,E_\ast X)} &\Hom(\ZZ_p[[\GG_n/H_2]],E_\ast X).
}
\]
Note that we can use the same $\varphi_\ast$ for both $H_1$ and $H_2$ because the $E^{hH_i}$-orientations $g_iz$
of $E^{hH_i} \wedge X$ map to the {\it same} $\GG_n$-invariant generator of $E_0X$. 
This last diagram evidently commutes, finishing the argument.
\end{proof}
 
 \subsection{The descent filtration}
 
Another  commonly used filtration to study $\kappa_n$ comes from the $K(n)$-local Adams--Novikov spectral 
sequence. See, for example \cite{HoveySadofsky} and \cite{ShimomuraKamiya}.

Recall from \Cref{rem:mormodulehom} that if $X$ is a dualizable object in the $K(n)$-local category, then the $E$-based Adams--Novikov spectral sequence for $E^{hK} \wedge X$ is given by
\begin{align*}
E_2^{s,\ast }(K,X) \cong H^s(K, E_tX) \Longrightarrow \pi_{t-s}L_{K(n)}(E^{hK}\wedge X).
\end{align*}
We will use the following key and deep property of this spectral sequence.

\begin{rem}\label{rem:props-of-ANSS} The $K(n)$-local $E$-based Adams-Novikov spectral sequence has a uniform and 
horizontal vanishing line at $E_\infty$; that is, there is an integer $N$, depending only on $n$, $p$, and $K$, so that in the 
Adams-Novikov spectral sequence we have
\[
E_\infty^{s,\ast }(K,X) = 0, \qquad s > N.
\]
This can be found in the literature in several guises; for example, it can be put together from the material in Section
5 of \cite{DH}, especially Lemma 5.11. See also \cite{BBGS} for even further explanation. 
If $p-1 > n$, there is often a  horizontal vanishing line at $E_2$, but we are decidedly not in that case in the second
part of this paper. 
\end{rem}

 Let $X \in \kappa_n$ and choose  a $\GG_n$-invariant generator $\iota_X \in E_0X$ as in
 \cref{defn:Gninvariantgen}. The Adams--Novikov Spectral Sequence  for the sphere acts on the 
 Adams--Novikov spectral sequence $X$; thus, if $d_q(\iota_X)=0$ for $r < q$, then we have an
 isomorphism of $E_r$-terms
\begin{align}\label{eq:isoEr}
E_r^{\ast,\ast}(\GG_n,S^0) &\cong E_r^{\ast,\ast}(\GG_n,X)\\
a &\longmapsto a\iota_X . \nonumber
\end{align}
Similar ideas were deployed in the proof of \Cref{prop:detect-kappa-ANSS-bis}. 

In \Cref{rem:really-basic-siso} we also observed that any two choices for a $\GG_n$-invariant
generator of $E_\ast X$ differ by multiplication by a unit in $\ZZ_p^\times$. Thus the following
definition is independent of the choices. 

\begin{defn}[Descent filtration]\label{defn:descent-filt} 
For $r\geq 2$, let
\[
\kappa_{n,r}= \{\ X \in \kappa_2 \mid d_k(\iota_X)=0,\ 2\leq k < r\  \} ,
\]
where $d_k$ denotes the differential in the homotopy fixed point spectral sequence
\[E_2^{s,t}(\GG_n,X) = H^s(\GG_n,E_tX )\Longrightarrow \pi_{t-s}X.\]
Define $\phi_r:\kappa_{n,r} \to E_r^{r,r-1}(S^0)$ by the equation
\[
d_r(\iota_X) = \phi_r(X)\iota_X,
\]
where the right-hand side uses the identification \eqref{eq:isoEr}.
\end{defn}
  
\begin{lem}\label{lem:desfilt} 
The subsets $\kappa_{n,r}$ of $\kappa_n$ 
satisfy the following basic properties.
\begin{enumerate}\item The subset $\kappa_{n,r}$ is a subgroup of $\kappa_n$ and $\phi_r$
is a homomorphism.

\item The kernel of $\phi_r$ is identified with $\kappa_{n,r+1}$, and so we have an exact sequence
\[
\xymatrix{
0 \ar[r] &  \kappa_{n,r+1} \ar[r] &  \kappa_{n,r} \ar[r]^-{\phi_r} &E_r^{r,r-1}(\GG_n, S^0)\ .
}
\]

\item There exists an integer $N$ so that $\kappa_{n,N}  = \{L_{K(n)}S^0\}$.
\end{enumerate}
\end{lem}

\begin{proof}  Part (1) follows from the observation that $\iota_X \in E_0X$ and $\iota_Y \in E_0Y$ are
$\GG_n$-invariant generators then
\[
\iota_X \wedge \iota_Y \in E_0(X \wedge Y)
\]
is a $\GG_n$-invariant generator. Part (2) is built into the definitions and part (3) follows from the horizontal
vanishing line of \Cref{rem:props-of-ANSS}.
\end{proof}

\begin{rem}\label{rem:phi-not-onto} The map $\phi_r:\kappa_{n,r} \to E_r^{r,r-1}(\GG_n, S^0)$ need not be onto.
In \Cref{rem:phi-not-surj} we note that the class $\eta^2e$ is not in the image of $\phi_5$ at $n=p=2$.
\end{rem}

 % !TEX root = pic-master.tex

\section{The $\GG_n^1$-orientable elements of the Picard group}\label{sec:third-graded}

For all $n$ and all $p$, the group $\GG_n$ has a closed subgroup $\GG_n^1$ defined as the kernel of a reduced
determinant map. In this section we give general results on $\kappa(\GG_n^1) \subseteq \kappa_n$, the
subgroup of exotic invertible elements which have an $E^{h\GG_n^1}$-orientation, and give some remarks on the 
interaction between $\kappa(\GG_n^1)$ and the filtration coming from the descent filtration of \Cref{defn:descent-filt}. 
Most of the key ideas are already present in Section 5 of \cite{GHMRPicard}, and then adapted and generalized
in \cite{Westerland}.

\subsection{The determinant, the subgroup $\GG_n^1$ and the class $\zeta$}
We briefly introduce the subgroup $\GG_n^1$ and the closely related homotopy class $\zeta$. 

We have already defined  $\SS_n = \Aut(\Gamma_n/\FF_{p^n}) \subseteq \GG_n$ and indeed, we have
a semi-direct product decomposition $\SS_n \rtimes \Gal \cong \GG_n$ where  $\Gal = \Gal(\FF_{p^n}/\FF_p)$. 
The group $\SS_n$ is the group of units in the endomorphism ring $\mathcal{O}_n$ of $\Gamma_n$ over $\FF_{p^n}$. 
The inclusion $\ZZ_p \to \mathcal{O}_n$ sending $n$ to the multiplication by $n$ extends to an inclusion
of the Witt vectors $\WW \to \mathcal{O}_n$ and $\mathcal{O}_n$ is a left $\WW$ module of rank $n$. The right action
of $\SS_n$ on $\mathcal{O}_n$ then defines a map
\[
\SS_n \longrightarrow \mathrm{Gl}_n(\WW).
\]
The image of this map has enough symmetry that the determinant restricts to a map
\[
\det \colon \SS_n \to \ZZ_p^\times \subseteq \WW^{\times}. 
\]
We then extend the determinant to all of $\GG_n$ as follows. 

\begin{defn} \label{defn:det-defined-here}
The \emph{determinant} $\det \colon \GG_n \to \ZZ_p^{\times}$ is defined to be
 the composite
\[
\xymatrix{
\GG_n \cong \mathbb{S}_n \rtimes \Gal \ar[rr]^-{\det \times \Gal} && \Z_{p}^{\times} \times \Gal \ar[r]  &\Z_{p}^{\times}
}
\]
where the second map is the projection. The kernel of $\det$ is denoted by $S\GG_n$.
\end{defn}
We can now define one of the key players in this story, the reduced determinant $\zeta$.

\begin{defn}\label{defn:zeta}
The homomorphism $\zeta=\zeta_n \colon \GG_n \to \Z_p$ is the composite
\[ \GG_n \xrightarrow{\det} \Z_p^{\times} \to \Z_p^{\times}/\mu \cong \Z_p\]
where the second map is the quotient by the subgroup $\mu$ of roots of unity. 
The kernel of $\zeta$ is denoted by $\G_n^1$. 
\end{defn}

\begin{rem}
We note that $\zeta$ is always a split surjection so that $\GG_n \cong \GG_n^1 \rtimes \ZZ_p$. Throughout, we will let
\[
\psi \in \GG_n/\GG_n^1 \cong \Z_p
\]
be a  topological generator.
\end{rem}

When $p=2$, we can define another important homomorphism derived from the determinant, which will not be used
in this section but later in \Cref{part:two} to study the case $n=p=2$.

\begin{defn} \label{defn:chi-defined}
Let $p=2$. The homomorphism $\chi \colon \GG_n \to \Z/2$ is defined as the composite 
\[
\xymatrix{
\GG_n \ar[r]^-{\det}& \Z_2^{\times} \ar[r] & \ZZ_2^\times/(1+4\ZZ_2) \cong (\Z_2/4\Z_2)^{\times} \cong \ZZ/2.
}
\]
We define $\GG_n^0$ to be the kernel of $\chi$.
\end{defn}

\begin{rem}
The fixed point spectrum $E^{h\GG_n^0}$ is a mysterious object, unusual even at $p=2$ and $n=1$. However,
the other fixed point spectrum $E^{h\GG_n^1}$ is much more familiar. If $n=1$, the determinant $\GG_1 \to
\ZZ_2^\times$ is the identity, so $\GG_1^1 = \{\pm 1\}$ and  $E^{h\GG_1^1} \simeq KO$, the $2$-complete
real $K$-theory spectrum. 
\end{rem}

We use the homomorphisms $\zeta$ and $\chi$ to define some of our key cohomology classes.

\begin{defn}\label{defn:what-is-wchi-anyway1} The homomorphisms
\begin{align*}
\zeta\colon \GG_n &\longr \ZZ_p\\
\chi  \colon \GG_n &\longr \ZZ/2
\end{align*}
define classes 
\[
\zeta \in H^1(\GG_n,\ZZ_p)\qquad\mathrm{and}\qquad \chi \in H^1(\GG_n,\ZZ/2)
\]
in cohomology with trivial coefficients. Since the inclusion $\ZZ_p \to E_0$ of the submodule
generated by the unit is $\GG_n$-invariant we can also write 
\[
\zeta \in H^1(\GG_n,E_0)\qquad\mathrm{and}\qquad \chi \in H^1(\GG_n,E_0/2)
\]
for the image of these classes under the inclusion map. Finally, let
\[
\wchi \in H^2(\GG_2,E_0)
\]
be the Bockstein on $\chi$. Note that $\chi$ and $\wchi$ are only defined if $p=2$. 
\end{defn}

We now can record a standard result for recovering $L_{K(n)}S^0$ from
$E^{h\GG_n^1}$.  See, for example, \cite[Proposition 8.1]{DH}.
It follows easily from the isomorphism of Morava modules
$\map(\GG_n/\GG_n^1,E_\ast) \cong E_\ast E^{h\GG_n^1}$ of \cref{rem:mormodulehom}. 

\begin{prop}\label{prop:fund-above} Let $\psi$ be a topological generator of $\GG_n/\GG_n^1$.
Then there is a fiber sequence
\[
\xymatrix{
L_{K(n)}S^0 \ar[r]^-i & E^{h\GG_n^1} \ar[r]^-{\psi-1} & E^{h\GG_n^1}
}
\]
that gives a short exact of Morava modules
\[
\xymatrix{
E_\ast S^0 \ar[r]^-{i_\ast} & E_\ast E^{h\GG_n^1} \ar[rr]^{(\psi - 1)_\ast} && E_\ast E^{h\GG_n^1}.
}
\]
The map $i_\ast$ is an injection onto the sub-Morava module of rank one generated by the
unit of the ring $E_\ast E^{h\GG_n^1}$. 
\end{prop} 

We have already conflated the cohomology class
\[
\zeta = \zeta_n \in H^1(\GG_n,\ZZ_p).
\]
with its image under the map
\[
H^1(\GG_n,\ZZ_p) \longrightarrow H^1(\GG_n,E_0) 
\]
defined by the inclusion of rings $\ZZ_p \to E_0$. As is customary, we will
further confuse this element with the homotopy class in 
\[
\zeta \in \pi_{-1}L_{K(n)}S^0
\] 
defined as the image of the unit $i\colon S^0 \to E^{h\GG_n^1}$ under the boundary homomorphism
\[
\partial \colon \pi_0E^{h\GG_n^1} \longrightarrow \pi_{-1}L_{K(n)}S^0.
\]
The homotopy class $\zeta$ is detected by the cohomology class $\zeta$ in the Adams--Novikov Spectral
Sequence
\[
E_2^{s,t}=H^s(\GG_nE_t) \Longrightarrow \pi_{t-s}L_{K(n)}S^0.
\]
This follows from  \cref{prop:fund-above}, the long
exact sequence in cohomology
\[
\xymatrix{
\cdots \ar[r] & H^s(\GG^1_n,E_t) \ar[r]^-{\psi-1} & H^s(\GG^1_n,E_t) \ar[r]^-{\partial_2}&
H^{s+1}(\GG_n,E_t) \ar[r] &\cdots\ ,
}
\]
and the Geometric Boundary Theorem. This is all due to Hopkins and Miller, see Theorem 6 in \cite{DH}.

We immediately have 
\[
\xymatrix{
\pi_kL_{K(n)} S^0 \ar[r]^-{i_\ast} & \pi_kE^{h\GG_n^1}  \ar[r]^-{\partial} & \pi_{k-1}L_{K(n)}S^0
}
\]
is multiplication by $\zeta$. This extends to the maps on Adams--Novikov Spectral Sequences: the composition 
\[
\xymatrix{
E_r^{s,t}(\GG_n,S^0)\ar[r]^{i_\ast} & E_r^{s,t}(\GG^1_n,S^0)  \ar[r]^{\partial_r} & E_r^{s+1,t}(\GG_n,S^0)
}
\]
is multiplication by $\zeta \in E_r^{1,0}(\GG_n,S^0)$.

\subsection{A description of $\kappa(\GG_n^1)$ in terms of homotopy groups}\label{sec:desckGn1}
There is a standard way to produce spectra in
$\kappa(\GG_n^1)$, which we now review. The main result of this subsection is \Cref{prop:twist-constr}. 

Define
\begin{equation}\label{eq:lambda-defined}
\KEn = \mathrm{Ker}\{\ \pi_0E^{h\GG_n^1} \longrightarrow E_0E^{h\GG_n^1}\ \} 
\end{equation}
to be the kernel of the $E$-Hurewicz homomorphism; this is the kernel of the edge homomorphism of the 
spectral sequence
\[
E_2^{s,t}(\GG_n, S^0)=H^s(\GG_n^1,E_t) \Longrightarrow \pi_{t-s} E^{h\GG_n^1}
\]
when $t=s$.
Since this spectral sequence has a horizontal vanishing line (see \Cref{rem:props-of-ANSS}) any 
element in $\KEn$ is nilpotent. We thus have a subgroup
\[
1 + \KEn \subseteq (\pi_0E^{h\GG_n^1})^\times
\]
of the units in the ring $\pi_0E^{h\GG_n^1}$.

\begin{defn}[Twisting Construction]\label{defn:super-e-defn}
For $\psi \in \GG_n/\GG_n^1$ a topological generator
and $\alpha \in 1 + \KEn$, we define $X(\alpha)$ by the fiber sequence
\[
\xymatrix{
X(\alpha) \ar[r]^{i_\alpha} & E^{h\GG_n^1} \ar[rr]^{\psi - \alpha} && E^{h\GG_n^1}.
}
\]
Here, in an abuse of notation, we also denote by $\alpha$ the unique $E^{h\GG_n^1}$-module map
$\varphi_\alpha:E^{h\GG_n^1} \to E^{h\GG_n^1}$ 
obtained by extending $\alpha \colon S^0 \to E^{h\GG_n^1}$. 
\end{defn}

By \Cref{prop:fund-above},
\[
L_{K(n)}S^0 \simeq X(1)
\] 
and there is a short exact sequence of Morava modules
\begin{equation}\label{eq:fund-above-alt}
\xymatrix{
E_\ast S^0 \ar[r]^-{(i_1)_\ast} & E_\ast E^{h\GG_n^1} \ar[rr]^{(\psi - 1)_\ast} && E_\ast E^{h\GG_n^1}.
}
\end{equation}
The map $(i_1)_\ast$ is an injection onto the sub-Morava module of rank one generated by the
unit of the ring $E_0E^{h\GG_n^1}\cong \map(\GG_n/\GG_{n}^1, E_0)$. This observation is extended in the following result.

\begin{prop}\label{prop:westerland-works} Let $\alpha \in 1 + \KEn$. Then the map
\[
(i_\alpha)_\ast : E_\ast X(\alpha) \longr E_\ast E^{h\GG_n^1}
\]
is an injection onto to the sub-Morava module of rank one generated by the unit.
The assignment $\alpha \mapsto X(\alpha)$ defines a homomorphism
\[
X(-):1+\KEn \longr \kappa(\GG_n^1).
\]
\end{prop} 

\begin{proof}  Since $E_\ast \alpha = E_\ast (1)$ the first statement follows from the short exact sequence
\eqref{eq:fund-above-alt}. It follows that the map $i_\alpha: X(\alpha) \to E^{h\GG_n^1}$ extends to an equivalence
\[
E^{h\GG_n^1} \wedge X(\alpha) \xrightarrow{\simeq} E^{h\GG_n^1}.
\]
If we let $z \in \pi_0(E^{h\GG_n^1} \wedge X(\alpha))$ be the element which maps to the unit in
the ring $\pi_0E^{h\GG_n^1}$ under this equivalence, then $z$ is an $E^{h\GG_n^1}$-orientation in
the sense of \Cref{defn:subgroupfilt}. Thus, $X(\alpha) \in \kappa(\GG_n^1)$. Proposition 3.17 of \cite{Westerland}
shows that the function $\alpha \mapsto X(\alpha)$ has the property that
there are canonical pairing maps
\[
X(\alpha)\wedge X(\beta) \longrightarrow X(\alpha\beta).
\]
and that this map is an equivalence. Thus we have a homomorphism. 
\end{proof}

\begin{rem} There's an omission in the statements of Propositions 3.15 and 3.17 of \cite{Westerland}. They are stated for a general unit in
$\pi_0E^{h\GG_n^1}$; however, the proofs in  \cite{Westerland} work only if $\alpha \equiv 1$ modulo the maximal
ideal in $\pi_0E^{h\GG_n^1}$. In all our applications, and indeed in all the applications in \cite{Westerland}, this
additional hypothesis holds. 
\end{rem} 

\begin{prop}\label{lem:cohom-inv} Let $\alpha \in 1 +  \KEn$ and let $\psi$ be a topological generator
of $\GG_n/\GG_n^1$. There is an equivalence $L_{K(n)}S^0 \to X(\alpha)$ if and only if
\[
\alpha\beta = \psi\beta
\]
for some $\beta \in 1+\KEn$.
\end{prop}

\begin{proof} First suppose $\psi\beta = \alpha\beta$. Then we have a factoring
\[
\xymatrix{
&L_{K(n)}S^0 \ar[d]^\beta \ar@{-->}[dl]_f\\
X(\alpha) \ar[r]_-{i_{\alpha}} & E^{h\GG_n^1} \ar[rr]^{\psi - \alpha} && E^{h\GG_n^1}.
}
\]
Since $\beta \in 1+\KEn$, the map $E_\ast \beta$ is injection onto the sub-Morava module of $E_\ast E^{h\GG_n^1}$ 
generated by the unit. Then \Cref{prop:westerland-works} implies the  map $f$ is an $E_\ast$-isomorphism. 

Conversely, suppose we are given an equivalence
$h \colon L_{K(n)}S^0 \to X(\alpha)$. We define $\gamma$ to be the image of $h$ in $\pi_0E^{h\GG_n^1}$ under
the map $i_\alpha \colon X(\alpha) \to E^{h\GG_n^1}$. Then we have a diagram
\[
\xymatrix{
&E_\ast S^0 \ar[d]^{\gamma_\ast} \ar[dl]^\cong_{h_\ast}\\
E_\ast X(\alpha) \ar[r]_-{(i_\alpha)_\ast} & E_\ast E^{h\GG_n^1}  
}
\]
Let $1 \in E_0S^0$ be the tautological generator.  Since the ring of Morava module endomorphisms of $E_\ast S^0$
is isomorphic to $\ZZ_p$ and $h_*$ is such an endomorphism, \Cref{prop:westerland-works} implies that
\[
\gamma_\ast(1) = a (i_\alpha)_\ast(1)
\]
for some $a \in \ZZ_p^\times$. Define $\beta = a^{-1}\gamma$.
Then $\beta \in 1 + \KEn$ and $\beta$ is in the kernel of $\psi-\alpha$. 
\end{proof} 

The quotient group $\ZZ_p \cong \GG_n/\GG_n^1$ acts on $1 + \KEn$ and we have an exact sequence
\[
\xymatrix{
1+\KEn \ar[r]^-{\partial} & 1+\KEn \ar[r] & H^1(\ZZ_p,1+\KEn) \ar[r] &0
}
\]
where $\partial(\beta) = \beta^{-1}\psi(\beta)$. Thus \Cref{lem:cohom-inv} implies that $X(\alpha)$ is trivial in $\kappa_n$
if and only if $\alpha$ is a coboundary in $1+\KEn$. Thus we have an injection
\[
\xymatrix{
H^1(\GG_n/\GG_n^1,1+\KEn) \ar[r] & \kappa(\GG_n^1)
}
\]
sending the coset of $\alpha$ to $X(\alpha)$. We will show in \Cref{prop:twist-constr} that this is an isomorphism, 
but we need some preliminaries. 

Let $\KIn \subseteq [E^{h\GG_n^1},E^{h\GG_n^1}]$ be the set of maps $f$ such that $E_\ast f = 0$; note that $\KIn$ is
an ideal in the endomorphism ring. The map $f \mapsto f(1_E)$ defines a split surjection
\begin{align}\label{eq:splita}
\xymatrix{\KIn \ar[r]^-\epsilon  & \KEn \ar@/^1.0pc/@{.>}[l]^-{s}}
\end{align}
The splitting sends $a \in \KEn$ to the $E^{h\GG_n^1}$-module map defined by extending $a$. 

Now let $g \in 1 + \KIn$ and form the cofiber sequence
\[
\xymatrix{
X(g) \ar[r]^-{i_g} & E^{h\GG_n^1} \ar[rr]^{\psi - g} && E^{h\GG_n^1}.
}
\]
The following is a generalization of \Cref{prop:westerland-works} and of \cref{lem:cohom-inv}.

\begin{prop}\label{prop:westerland-works-1}
Let $g \in 1 + \KIn$. Then the map
\[
(i_g)_\ast : E_\ast X(g) \longr E_\ast E^{h\GG_n^1}
\]
is an injection onto to the sub-Morava module of rank one generated by the unit.
The assignment $g \mapsto X(g)$ defines a homomorphism
\[
X(-):1+\KIn \longr \kappa(\GG_n^1).
\]
For $g \in 1 +  \KIn$ and $\psi$ the chosen topological generator
of $\GG_n/\GG_n^1$, there is an equivalence $L_{K(n)}S^0 \to X(g)$ if and only if
\[
g \beta = \psi \beta
\]
for some $\beta \in 1+\KEn$. 
\end{prop}
\begin{proof}
The proof of \Cref{prop:westerland-works} goes through
without change for the first part, and that of \cref{lem:cohom-inv} for the second part.
\end{proof}
\begin{rem}
As in \cref{prop:westerland-works-1}, we have in particular that if $g(1_E)=1_E$, then $X(g)$ is trivial.
\end{rem}

If we write $g \in 1 + \KIn$ as $g = 1 + f$ with $f \in \KIn$, then $g(1_E) = 1_E$ if and only if $f(1_E) = 0$. We then have
a diagram
\begin{equation}\label{eq:replace-mod}
\xymatrix{
1+ \KIn \ar[r] \ar[d]_{1+\epsilon}& \kappa(\GG_n^1)\\
1+ \KEn \ar[ur]
}
\end{equation} 
where the vertical map sends $g$ to $g(1_E)$. Thus we have the following result.

\begin{lem}\label{lem:replace-mod} Suppose $g \colon E^{h\GG_n^1} \to E^{h\GG_n^1}$ is any self map so that
$E_\ast g$ is the identity and let $\alpha = g(1_E) \in \pi_0E^{h\GG_n^1}$. Then
\[
X(g) =  X(\alpha) \in \kappa(\GG_n^1).
\]
\end{lem}
\begin{proof}
Let $s$ be the splitting of \eqref{eq:splita}, then $X(s(\alpha))=X(\alpha)$ by definition. Therefore, $X(s(\alpha^{-1})g)= 
X(s(\alpha^{-1}))\wedge X(g) = X(\alpha^{-1})\wedge X(g)$. But \cref{prop:westerland-works-1} implies that 
$X(s(\alpha^{-1})g)$ is trivial since $(s(\alpha^{-1})g)(1_E)=1_E$. 
\end{proof}

We can now prove the following result.

\begin{thm}\label{prop:twist-constr} The assignment $\alpha \mapsto X(\alpha)$ defines an isomorphism
\[
\xymatrix{
H^1(\GG_n/\GG_n^1,1+\KEn) \ar[r]^-\cong & \kappa(\GG_n^1). 
}
\]
In particular, if $\GG_n/\GG_n^1$ acts trivially on $\KEn$, then we have an isomorphism $1+\KEn \cong \kappa(\GG_n^1)$.
\end{thm} 

\begin{proof} \Cref{lem:cohom-inv} shows that the map is well-defined and an injection; thus we must 
show it is onto.

Let $X\in \kappa(\GG_n^1)$. Choose an $E^{h\GG_n^1}$-orientation $z \in  \pi_0(E^{h\GG_n^1} \wedge X)$;
see \Cref{defn:subgroupfilt}. Let $\varphi_1 \colon E^{h{\GG_n^1}} \to E^{h{\GG_n^1}}\wedge X$ and
$\varphi \colon  E \to E\wedge X$ be the equivalences induced by extending $z$. The
image of $z$ in $E_0X$ is denoted $\iota_X$; it is a $\GG_n$-invariant generator of $E_0X$.  

For any spectrum $A$ we have an isomorphism $\phi_A$ of Morava modules
\[
\xymatrix{
E_\ast A \ar[r]^-{\cong} & E_\ast A \otimes_{E_0} E_0X \ar[r]^-{\cong} & E_\ast (A \wedge X).
}
\]
The first isomorphism sends $a$ to  $a \otimes \iota_X$, and the second is the K\"unneth isomorphism.
Both maps are natural is $A$; thus, so is $\phi_A$. If $A = E^{h\GG_n^1}$ we will simply write $\phi = \phi_A$. 
The composition 
\[
\xymatrix@C=45pt{
E_\ast E^{h\GG_n^1} \ar[r]^-{(\varphi_1)_*} & E_\ast(E^{h\GG_n^1}\wedge X)    \ar[r]^-{\phi^{-1}} & E_\ast E^{h\GG_n^1}.
}
\]
is then the identity. Let $\psi \in \GG_n/\GG_n^1$ be a topological generator. (See \Cref{prop:fund-above}.) 
We next define a self-map $\widetilde{\psi}\colon E^{h\GG_n^1} \to E^{h\GG_n^1}$ by requiring the following diagram to commute
\[
\xymatrix@C=55pt{
E^{h\GG_n^1} \ar[d]_-{\varphi_1}^-\simeq \ar[r]^-{\widetilde{\psi}} &E^{h\GG_n^1} \ar[d]_-{\varphi_1}^-\simeq\\
E^{h\GG_n^1} \wedge X \ar[r]_-{\psi \wedge X} & E^{h\GG_n^1} \wedge X.
}
\]
Applying $E$-homology and prolonging with $\phi^{-1}$, we obtain the commutative diagram
\[
\xymatrix@C=55pt{
E_\ast E^{h\GG_n^1} \ar[d]_{(\varphi_1)_\ast} \ar[r]^-{\widetilde{\psi}_\ast} &
E_\ast E^{h\GG_n^1} \ar[d]^-{(\varphi_1)_\ast}\\
E_\ast ( E^{h\GG_n^1} \wedge X )\ar[d]_-{\phi^{-1}} \ar[r]_-{(\psi \wedge X)_\ast } &
E_\ast ( E^{h\GG_n^1} \wedge X)\ar[d]^-{\phi^{-1}}  \\
E_\ast E^{h\GG_n^1} \ar[r]^-{\psi_\ast} &E_\ast E^{h\GG_n^1},
}
\]
where the vertical composites are the identity.
From this we conclude that
\[
\widetilde{\psi}_\ast = \psi_\ast\colon E_\ast E^{h\GG_n^1} \longrightarrow E_\ast E^{h\GG_n^1}.
\]
Defining $F$ to be the fiber of $\widetilde{\psi}-1$, we get a diagram of fiber sequences
\[
\xymatrix@C=55pt{
F \ar[r] \ar[d]_\simeq &E^{h\GG_n^1} \ar[d]_-{\varphi_1}^-\simeq \ar[r]^-{\widetilde{\psi}-1} &E^{h\GG_n^1} \ar[d]_-{\varphi_1}^-\simeq\\
X \ar[r] & E^{h\GG_n^1} \wedge X \ar[r]_-{(\psi-1) \wedge X} & E^{h\GG_n^1} \wedge X.
}
\]
Define $f=\psi-\widetilde{\psi} \colon E^{h\GG_n^1} \to E^{h\GG_n^1}$. By construction $\widetilde{\psi} -1 = \psi - (1+f)$.
Furthermore, $E_\ast f=0$. Hence $f \in \KIn$ and $F \simeq X(1+f)$. The result now follows from
\Cref{lem:replace-mod}.
\end{proof}

\subsection{A comparison of filtrations}

In \Cref{defn:descent-filt}  we defined a filtration on $\kappa_n$ using the Adams-Novikov spectral sequence;
specifically $\kappa_{n,s} \subseteq \kappa_n$ is the subgroup of elements so that $d_r(\iota_X) = 0$ for $r < s$
for any choice of $\GG_n$-invariant generator $\iota_X$ of $E_0X$. We also discussed homomorphisms
\[
\phi_s \colon \kappa_{n,s} \to E_{r}^{s,s-1}(\GG_n,S^0)
\]
determined by the formula
\[
d_s(\iota_X) = \phi_s(X) \iota_X.
\]

Define
\[
\kappa_s(\GG_n^1) = \kappa_{n,s} \cap \kappa(\GG_n^1) \subseteq \kappa(\GG_n^1) \cong 
H^1(\GG_n/\GG_n^1,1+\KEn).
\]
These subgroups give a filtration of $\kappa(\GG_n^1)$. In this section, we compare it with
the Adams--Novikov filtration on $\KEn$. This does not appear to be formal; indeed, it is not clear
how the two filtrations can be compared without some additional hypotheses. We provide such a result with hypotheses that will suffice for our purposes in \Cref{lem:filtcompare}.

\begin{defn}\label{def:kens-filt}
Let $\KEn_s \subseteq \pi_0E^{h\GG_n^1}$ be the subgroup of elements 
whose Adams--Novikov filtration is greater than or equal to $s$ in the spectral sequence
\[
E_2^{s,t}(\GG_n^1,S^0) = H^s(\GG_n^1,E_t) \Longrightarrow \pi_{t-s}E^{h\GG_n^1}.
\]
\end{defn}

Note that by definition, $\KEn = \KEn_1$, and furthermore, we get a corresponding filtration $1+\KEn_s \subseteq 1+\KEn$. The first observation about this latter filtration will be the cause of much technical complication.

\begin{lem}\label{lem:filtstructure} Let $\alpha \in 1+\KEn_s$ and $\beta \in 1+ \KEn_{s'}$ and $s < s'$. Then $\alpha\beta \in 1+\KEn_s$. If $\alpha$ has exact filtration $s$, then so does $\alpha\beta$.
\end{lem}

\begin{proof} Write $\alpha = 1 + x$ and $\beta = 1+ y$ with $x$ and $y$ of filtration $s$ and $s'$ respectively 
in the Adams-Novikov spectral sequence for $E^{h\GG_n^1}$. Then
$\alpha\beta -1 $ is congruent to $ x$ modulo elements of filtration greater than $s$. 
\end{proof} 

\begin{rem}\label{rem:filtstructure} Beyond the complications implied by \Cref{lem:filtstructure}, the
filtration $1+\KEn_s$ of $1 + \KEn$ also does not fit particularly well with the cohomological description in \Cref{prop:twist-constr}.
While $1+\KEn_s$ is closed under the action of  $\GG_n/\GG_n^1$, the map 
\[
H^1(\GG_n/\GG_n^1,1+\KEn_s) \longr H^1(\GG_n/\GG_n^1,1 + \KEn) 
\]
induced by the inclusion is not obviously one-to-one. These facts complicate the analysis of the relationship
between the Adams-Novikov filtration $\kappa_r(\GG_n^1)$ and the Adams-Novikov filtration on
$\KEn \subseteq \pi_0E^{h\GG_n^1}$, at least in the absence of further hypotheses. This observation explains the excessive
(even in the context of this paper) technicality of \Cref{lem:filtcompare}.
\end{rem}

\begin{rem} A first hypothesis will be to require that $\GG_n/\GG_n^1$ act trivially  on $\pi_0E^{h\GG_n^1}$.
Then \cref{prop:twist-constr} specifies an isomorphism
$1+\KEn \cong \kappa(\GG_n^1)$, sending $\alpha$ to $X(\alpha)$. We then get a filtration
\[
\kappa(\GG_n^1) \cong 1+\KEn = 1+\KEn_1 \supseteq 1+ \KEn_2 \supseteq \ldots
\]
of $\kappa(\GG_n^1)$. In that case, we have two filtrations of $\kappa(\GG_n^1)$, namely $1+\KEn_s$ and
$\kappa_s(\GG_n^1)$ and we wish to compare them. This is the goal of the rest of the section. Note
that if $\alpha = 1 + x$ with $x \in \KEn_{s-1}$, then we hope to have a formula of the form
\[
\phi_{s}(X(\alpha)) = \zeta \overline{x} \in E_{s}^{s,s-1}(\GG_n,S^0)
\]
where $\overline{x} \in E_{s}^{s-1,s-1}(\GG^1_n,S^0)$ detects $x$ and $\zeta \in H^1(\GG_n,E_0)$ is the
cohomology class of \Cref{defn:what-is-wchi-anyway1}. If this holds, then $X(\alpha) \in \kappa_{s}(\GG_n^1)$.
Thus we would be comparing $1+\KEn_{s-1}$ and $\kappa_{s}(\GG_n^1)$. 
\end{rem}

\begin{lem}\label{lem:input-for-all-below-general} 
Suppose that $\GG_n/\GG_n^1$ acts trivially on $\pi_0E^{h\GG_n^1}$, and let $\alpha \in 1+\KEn$. The boundary map
$\pi_0E^{h\GG_n^1} \to\pi_{-1}X(\alpha)$ of the long exact sequence on homotopy for the fibration
\begin{equation}\label{eq:xalpha-again}
X(\alpha) \to E^{h\GG_n^1}\xrightarrow{\psi-\alpha} E^{h\GG_n^1}
\end{equation}
induces an injection
\[
\pi_0E^{h\GG_n^1}/(\alpha-1) \hookrightarrow \pi_{-1}X(\alpha) .
\]
\end{lem}

\begin{proof} For $a \in \pi_0E^{h\GG_n^1}$, $\psi(a) = a$, so
\[
(\psi - \alpha)(a) = (1-\alpha)a.
\]
The claim then follows from the long exact sequence on homotopy groups for the fiber sequence
\eqref{eq:xalpha-again}. 
\end{proof} 

We make two basic observations to help organize the assumptions in the following results.
First, recall from \Cref{def:kens-filt} that
$\KEn_s \subseteq \pi_0E^{h\GG_n^1}$ is the subgroup of elements of filtration at least $s$ in the $\GG_n^1$ homotopy fixed point spectral sequence.
Since $H^1(\GG_n^1,E_1) = 0$, we have $\KEn_2 = \KEn_1 = \KEn$. Thus many of 
statements begin with $s=2$. 

Second, because of the uniform horizontal vanishing of the homotopy fixed point spectral sequence,
we know that there is an integer $N$ so that for all $s \geq N$ we have that  $1+ \KEn_s$ 
and $\kappa_s(\GG_n^1) \subseteq \kappa_{n,s} $ are trivial. One of our goals is to show that, at least under
certain hypotheses, the assignment $\alpha \mapsto X(\alpha)$ defines an isomorphism
$1+\KEn_{s-1} \cong \kappa_s(\GG_n^1)$. If $s > N$, this is obvious, so we are free to concentrate
on smaller $s$. 

The first step for proving our comparison result is the following factorization of the homomorphism $\phi_s$ restricted to $\kappa_s(\GG_n^1)$.

\begin{lemma}\label{lem:factoringkappas}
Assume that $\GG_n/\GG_n^1$ acts trivially on $\pi_0E^{h\GG_n^1}$, that $\KEn_{s-1}$ is trivial for $s> N$, and
assume further that for all $2\leq s  \leq N$ the following conditions are satisfied:

\begin{enumerate}[(i)]

\item  $E_{\infty}^{s-1,s-1}(\GG_n^1, S^0) \cong E_{s}^{s-1,s-1}(\GG_n^1, S^0)$; and
\smallskip

\item there is an exact sequence 
\[ 
\xymatrix{ 
0 \ar[r] & E_{s}^{s-1,s-1}(\GG_n^1, S^0) \ar[r]^-{\zeta} & E_{s}^{s,s-1}(\GG_n,S^0) \ar[r]^-{i_*} & E_s^{s,s-1}(\GG_n^1,S^0)
}
\]
where $i_*$ is the restriction.
\end{enumerate}
Then, the map $\phi_s$ restricts to a homomorphism
\[ \phi_s: \kappa_s(\GG_n^1) \to \zeta E_s^{s-1,s-1}(\GG_n^1,S^0)\]
for all $2 \leq s \leq N$.
\end{lemma}
\begin{proof}

By \Cref{prop:twist-constr}, $\kappa(\GG_n^1) \cong 1+\KEn$ and from (i) and (ii), for $2\leq s\leq N$ we have
an injection
\[
\xymatrix{
0 \ar[r] & \KEn_{s-1}/\KEn_{s} \cong E_{\infty}^{s-1,s-1}(\GG_n^1, S^0) \ar[r]^-{\zeta} & E_s^{s,s-1}(\GG_n,S^0).
}
\]
If $X \in \kappa(\GG_n^1)$ then we can choose a $\GG_n$-invariant generator $\iota_X \in E_0X$ which
is a permanent cycle in the $\GG_n^1$-homotopy fixed point spectral sequence
\[
H^s(\GG_n^1,E_tX) \Longrightarrow \pi_{t-s}(E^{h\GG_n^1} \wedge X).
\]
For this reason, any differential on $\iota_X$ in the $\GG_n$-homotopy fixed point spectral sequence must lie in the kernel of 
the restriction $i_\ast: E_s^{s,s-1}(\GG_n,S^0) \to E_s^{s,s-1}(\GG_n^1,S^0)$ for some $s$. 
We conclude from (ii) that the possible differentials are of the form
\[
d_s(\iota_X) = (\zeta y)\iota_X
\]
for some element $y \in E_{\infty}^{s-1,s-1}(\GG_n^1, S^0)$. Rephrased, we have shown that the
map $\phi_s$ factors as a map
\[
\xymatrix{
\kappa_s(\GG_n^1) \ar[r] &\zeta E_s^{s-1,s-1}(\GG^1_n,S^0) \ar[r]^-{\subseteq} &E_s^{s,s-1}(\GG_n,S^0).
}
\]
We will continue to write $\phi_s$ for the first of these maps, and this is the map from the statement of the lemma.
\end{proof}

The stage is now set for our main comparison result. Note that all hypotheses will be checked in 
\Cref{sec:subgroupsk2G21} in the case $p=n=2$. 

\begin{thm}\label{lem:filtcompare}
Assume that $\GG_n/\GG_n^1$ acts trivially on $\pi_0E^{h\GG_n^1}$, that $\KEn_{s-1}$ is trivial for $s> N$, and
assume further that for all $2\leq s  \leq N$ the following conditions are satisfied:

\begin{enumerate}[(i)]

\item  $E_{\infty}^{s-1,s-1}(\GG_n^1, S^0) \cong E_{s}^{s-1,s-1}(\GG_n^1, S^0)$;
\smallskip

\item there is an exact sequence 
\[ 
\xymatrix{ 
0 \ar[r] & E_{s}^{s-1,s-1}(\GG_n^1, S^0) \ar[r]^-{\zeta} & E_{s}^{s,s-1}(\GG_n,S^0) \ar[r]^-{i_*} & E_s^{s,s-1}(\GG_n^1,S^0)
}
\]
where $i_*$ is the restriction; and
\smallskip

\item  all elements in  $E_{s}^{s-1,s-1}(\GG_n^1, S^0)$ are torsion.

\end{enumerate}

Then, we conclude that for all $s\geq 2$,

\begin{enumerate}[(1)]
\item the homomorphism $\alpha \mapsto X(\alpha)$ defines
an isomorphism $1+\KEn_{s-1} \cong \kappa_s(\GG_n^1)$;
\smallskip

\item the map $\phi_s$ induces an isomorphism
\[
\kappa_s(\GG_n^1)/\kappa_{s+1}(\GG_n^1) \cong \zeta E_s^{s-1,s-1}(\GG_n^1, S^0);
\]

\item if $x \in \KEn_{s-1}$, $\alpha = 1 + x$, and
$\overline{x} \in E_{s}^{s-1,s-1}(\GG_n^1,S^0)$ is the equivalence class of $x$, then
\[
d_{s}(\iota_{X(\alpha)}) = a\, \overline{x} \zeta
\]
for some $a\in \Z_p^{\times}$. 
\end{enumerate}
\end{thm}

\begin{proof}  

We set up an inductive argument, with the following induction hypothesis. For a fixed integer $2\leq s\leq N$ and
all integers $r$ with $2 \leq r \leq s$

\begin{enumerate}[(a)]
\item The homomorphism $\alpha \mapsto X(\alpha)$ defines
an injection $1+\KEn_{r-1} \to \kappa_r(\GG_n^1)$.
\smallskip

\item Let $x \in \KEn_{r-1}$ and $\alpha = 1 + x $. If
$0 \ne \overline{x} \in E_{r}^{r-1,r-1}(\GG_n^1,S^0)$ then
\[
d_{r}(\iota_{X(\alpha)}) = a\, \overline{x} \zeta
\]
for some $a\in \Z_p^{\times}$.
\smallskip

\item The map $\phi_r$ induces an isomorphism
\[
\kappa_r(\GG_n^1)/\kappa_{r+1}(\GG_n^1) \cong \zeta E_r^{r-1,r-1}(\GG_n^1, S^0).
\]
Furthermore, for any $0 \ne y \in E_r^{r-1,r-1}(\GG_n^1, S^0)$ there is a class $\beta \in 1+\KEn_{r-1}$ so
that $\phi_r(X(\beta)) = \zeta y$ and $\beta$  is non-trivial modulo $1+\KEn_r$.
\smallskip

\item Let $x \in \KEn_{r}$ and $\alpha = 1 + x $. Then $ d_{r}(\iota_{X(\alpha)}) = 0 $.
\end{enumerate}

To deduce the result from this, note that since $\KEn_N=0$ and $\kappa(\GG_n^1) = 1+\KEn$, (a)--(d) imply that
$\kappa_{N+1}(\GG_n^1)=0$. So for $s\geq N+1$, (1)--(3) are trivial statements.
For $2\leq s\leq N$, part (1) of the result can be deduced from (a) and (c), part (2) from (c) and, and part (3) from
(b) and (d). 

We now begin the induction argument. The base case is $s=2$. The statements follow from the sparseness of the 
spectral sequence. Specifically, we have that
\begin{align*}
E_2^{1,1}(\GG_n^1,S^0) &\cong H^1(\GG^1_n,E_1) = 0\\
E_2^{2,1}(\GG_n,S^0) &\cong H^2(\GG_n,E_1) = 0.
\end{align*}

We now proceed with the induction step. So assume the statements (a)--(d) hold for $s-1$. 

We first show (a). Since the homomorphism
\[
X(-)\colon 1 + \KEn \longr \kappa(\GG_n^1)
\]
is an isomorphism, we need only show $X(F_{s-1}) \subseteq \kappa_s(\GG_n^1)$. By (d) for $s-1$, we have 
that if $\alpha \in F_{s-1}$, then $d_r(\iota_{X(\alpha)}) = 0$ for all $r < s$, and the assertion follows.

We now prove (b). Let $X=X(\alpha)$ where $\alpha=1+x$ for $x \in \KEn_{s-1}$ and
$\overline{x} \ne 0 \in E_s^{s-1,s-1}(\GG_n^1,S^0)$. We have just shown that
$X(\alpha) \in \kappa_s(\GG_n^1)$. In particular 
$\iota_X$ survives to the $E_{s}$-page of the $\GG_n^1$-homotopy fixed point spectral sequence. By (ii) in our hypotheses,  $0 \ne  \overline{x} \zeta \iota_X$.
By \cref{lem:input-for-all-below-general}, the element $\overline{x} \zeta \iota_X$ must be hit by a differential;
otherwise, it would survive to a non-trivial element detecting the boundary of $\alpha-1$. Since
$H^0(\GG_n^1,E_0) \cong \ZZ_p$ generated by $\iota_X$, the only possibility is
\[
d_{s}(\iota_{X} )= a\, \overline{x} \zeta  \iota_{X}, \qquad a\in \Z_p^{\times}.
\]
This proves (b) and shows
\[
\phi_{s}(X(\alpha))=a\, \overline{x} \zeta \neq 0.
\]

We can now move to (c). By definition $\phi_s$ induces an injection
\[
\kappa_s(\GG_n^1)/\kappa_{s+1}(\GG_n^1) \longr E_s^{s,s-1}(\GG_n,S^0).
\]
Since $\zeta E_s^{s-1,s-1}(\GG_n^1,S^0) \subseteq  E_s^{s,s-1}(\GG_n,S^0)$, the map of (c) remains
an injection: we have only changed the target. So we need to show that it is onto. Let $0\ne y \in E_s^{s-1,s-1}(\GG_n^1,S^0)$ and chose
an $x \in \KEn_{s-1}$ with $\overline{x}= y$. Let $\alpha = 1 + x$. Then we have just shown
\[
\phi_{s}(X(\alpha))=a\, \overline{x} \zeta \neq 0
\]
for some $a \in \ZZ_p^\times$. Since we assumed in (iii) that every element of $E_{s}^{s-1,s-1}(\GG_n^1,S^0)$ has
finite order,   we can choose a positive integer $b$ coprime to $p$ such that $ba  \overline{x}  =  \overline{x}$. 
Since $\phi_s$ is a homomorphism,
\[
\phi_{s}(X(\alpha^b)) = b a\, \overline{x} \zeta = y\zeta
\]
as needed. Note we have proved the final statement of (c) as well: the class $\alpha^b = (1+x)^b \in 1+\KEn_{s-1}$ has
the property that the residue class of $\alpha^b-1$ in $E_s^{s-1,s-1}(\GG_n^1,S^0)$ is $by$ and, hence, non-zero. 

We are left with (d). The only case we need to prove is $r=s$. So suppose that
$x \in \KEn_{s}$ and $\alpha = 1 + x $. Then $d_r(\iota_X) = 0$ for $r < s$ and we have
\[
\phi_{s}(X(\alpha)) =  y\zeta
\]
for some $y \in E_s^{s-1,s-1}(\GG_n^1,S^0)$. We will show $y=0$ by contradiction. So assume $ y \ne 0$.

By (c), there is a class $\beta = 1 + z\in 1+\KEn_{s-1}$ so that the coset $\overline{z}$ of $z$ in
$E_s^{s-1,s-1}(\GG_n^1,S^0)$ is non-zero and
\[
\phi_{s}(X(\beta)) =  - y\zeta.
\]
Then $\phi_{s}(X(\alpha\beta)) = 0$, $\alpha\beta \in 1+\KEn_{s-1}$, and $\alpha\beta \equiv 1 + z$ modulo $\KEn_{s}$. This
contradicts (b). 

This finishes the induction step, and the result follows.
\end{proof}

 % !TEX root = pic-master.tex

\section{The $J$-construction}\label{sec:j-construction} 
We now come to a fundamental construction which allows us to produce invertible $K(n)$-local spectra from virtual 
representations of quotients of $\GG_n$.

Let $q:\GG_n \to H$ be a continuous map to a finite group. This
will usually be surjective.  Suppose we are given an action of $H$ on $L_{K(n)}S^k$  specified by a map of spaces
\[
f:BH \to \{k\} \times B\Gl_1(L_{K(n)}S^0) \subseteq \ZZ \times B\Gl_1(L_{K(n)}S^0).
\]
Here $\Gl_1(L_{K(n)}S^0)$ is simply the topological monoid of self-equivalences of $L_{K(n)}S^0$. 
We will write $S(f)$ for $L_{K(n)}S^k$ with the  action defined by $f$. We can form the spectrum $E \wedge S(f)$
with the diagonal $\GG_n$-action. 

It is straightforward to check that this action has a continuous refinement in the sense of \cite[Definition 2.5]{BBGS}. The requisite map 
\[ E \wedge S(f)  \to F_c({\GG_n}_+, E\wedge S(f))\]
to the continuous function spectrum \cite[Definition 2.2]{BBGS} can be built from the analogous maps for $E$ and $S(f)$: the first coming from \cite{DH}, the second from the fact that $H$ is finite.
Then we can define
\[
J(f) = J(q,f,\GG_n)=  (E \wedge S(f))^{h\GG_n}
\]
as the continuous homotopy fixed point spectrum, in the style of Devinatz-Hopkins; see \cite[Definition 2.8]{BBGS}.

The notation $J(f)$ is under-decorated since $J(f)$ depends on the map $q: \GG_n \to H$ as well
as the map the map $f$. In context, we hope that $q$ is clear. 
More generally, if $K \subseteq \GG_n$ is any closed subgroup we then define the $E^{hK}$-module spectrum
\begin{equation}\label{eq:defnJQFK}
J(q,f,K) = (E \wedge S(f))^{hK}.
\end{equation}

The following basic case explains the choice of $J$ for this notation. 

\begin{ex}\label{exam:rep-spheres} Suppose $H$ is finite and $V$ is a virtual representation of $H$ of dimension $k$.
Then the one-point compactification $S^V$ of $V$ has an $H$ action and has underlying sphere $S^k$. 
The localization $L_{K(n)}S^V$ inherits this action and we obtain
\[
J(V) = (E \wedge S^V)^{h\GG_n}. 
\] 
Write $RO(H)$ for the real representation ring and $RO(H)^\wedge$ for the
completion at the augmentation ideal. Then the induced map 
\[
RO(H) \to [BH,\ZZ \times B\Gl_1(L_{K(n)}S^0)]
\]
factors as 
\begin{align*}
RO(H) \to RO(H)^\wedge &\cong [BH,\ZZ \times BO]\\
&\to  [BH,\ZZ \times B\Gl_1(S^0)]\\
&\to  [BH,\ZZ \times B\Gl_1(L_{K(n)}S^0)].
\end{align*}
Then $J(V) = J(f)$ where $f$ is the image of $V$ under this map. This example is also discussed in \cite[Section 3]{BBHS} and \cite[Section 12.1]{BGHS}.
\end{ex} 

Because of this example, we make the following definition.

\begin{defn} We refer to $f \colon BH \to \ZZ \times B\Gl_1(L_{K(n)}S^0)$ as a $K(n)$-local spherical
representation of $H$ or simply as a {\it spherical representation.} The integer $k$ obtained by projection onto the $\Z$
factor is the \emph{virtual dimension} of this spherical representation. 
\end{defn}
In the rest of the section, we will want to study properties of this construction. But we first need a technical result in
$K(n)$-local homotopy theory that allows us to untwist certain homotopy fixed points. The result uses the language of 
Devinatz--Hopkins \cite{DH} which is also reviewed in depth in \cite{BBGS}. See also \cref{rem:mormodulehom}.
The initial version is a slight generalization on Devinatz--Hopkins's determination of $E_\ast E$, and is as follows.
 
 \begin{lem}[{\cite[Corollary 2.18]{BBGS}}]\label{prop:basic-j-00}
  Let $X$ be a spectrum with a $\GG_n$-action, dualizable in the $K(n)$-local category. Give
$E \wedge X$ the diagonal $\GG_n$-action and suppose this action is continuous. 
Then there is an equivalence
\[
E \wedge (E \wedge X)^{h\GG_n} \simeq E \wedge X,
\]
inducing an isomorphism of Morava modules $E_\ast (E \wedge X)^{h\GG_n} \cong E_\ast X$, where
$\GG_n$ acts on $E_\ast X \cong \pi_\ast L_{K(n)}(E \wedge X)$ diagonally. 

 \end{lem}

We will use the following upgrade.

\begin{prop}[{\bf Untwisting Equivalence}]\label{prop:basic-j-0} Let $X$ be a spectrum with a $\GG_n$-action, dualizable in the $K(n)$-local category. Give
$E \wedge X$ the diagonal $\GG_n$-action and suppose this action is continuous. 
Then for all closed subgroups $K \subseteq \GG_n$ the natural map
\[
(E \wedge X)^{h\GG_n} \longr (E \wedge X)^{hK}
\]
extends to an equivalence of $E^{hK}$-module spectra
\[
E^{hK} \wedge (E \wedge X)^{h\GG_n} \simeq (E \wedge X)^{hK}.
\]
\end{prop} 

\begin{proof}
Let $A$ be any spectrum, $Y$ any $K(n)$-locally dualizable spectrum with a $\GG_n$-action such that the diagonal $\GG_n$-action on $E\wedge Y$ is continuous, and let $K $ be any closed subgroup of $ \GG_n$. Then by \cite[Proposition 2.17]{BBGS}, there is a $K(n)$-local equivalence
\[ A \wedge (E\wedge Y)^{hK} \simeq (A \wedge E\wedge Y)^{hK}.\]
For example, we could take $A = (E\wedge X)^{h\GG_n}$ and $Y = S^0$ to obtain
\[  (E\wedge X)^{h\GG_n} \wedge E^{hK} \simeq ((E\wedge X)^{h\GG_n} \wedge E)^{hK}.\]
Now it suffices to show that  $E \wedge X$ with its diagonal $\GG_n$-action is equivalent to $(E\wedge X)^{h\GG_n} \wedge E $, where the $\GG_n$-action on the latter is from the right-hand factor. This is immediate from \Cref{prop:basic-j-00}.
\end{proof}
 
We can now explore the consequences of this untwisting result for spherical representations.

\begin{prop}\label{prop:basic-j-1} Let $f:BH \to \ZZ \times B\Gl_1(L_{K(n)}S^0)$ be a spherical representation and
$S(f)$ the associated representation sphere. Suppose we are given a map $q\colon \GG_n \to H$ and
$J(f) = J(q,f,\GG_n)$. Then for all closed $K \subseteq \GG_n$ there is a natural equivalence
\[
\xymatrix{
E^{hK}\wedge J(f) = E^{hK} \wedge J(q,f,\GG_n)  \ar[r]^-\simeq &J(q,f,K)= (E \wedge S(f))^{hK} 
}
\]
of $E^{hK}$-module spectra,
where $K$ acts diagonally on $E \wedge S(f)$.
If, in addition $K$ is in the kernel of $q$, we have 
\[
E^{hK} \wedge J(f)  \simeq \Sigma^k E^{hK},
\]
where $k$ is the virtual dimension of $f$.
\end{prop} 

\begin{proof} This follows from \Cref{prop:basic-j-0} and the equation $J(f) = (E \wedge S(f))^{h\GG_n}$. If $K$
is in the kernel of $q$, then $K$ acts trivially on $S(f) \simeq S^k$, and the second statement
follows. 
\end{proof}

\begin{rem}\label{rem:basic-j-1b} 
Given a spherical $H$-representation $f$, the action of $H$ on the right factor of $E \wedge S(f)$ induces an action on the Morava module
$E_\ast S(f)$. Note that, as Morava modules, $E_\ast S(f) \cong E_\ast  S^k$.
We thus get a map
\[
\chi_f :H \longr \Aut_{\Mor}(E_\ast S^k) \cong \Aut_{\Mor}(E_\ast S^0) \cong \ZZ_p^\times
\]
where the automorphism group is the Morava module automorphisms. We refer to $\chi_f$ as the {\it character
defined by} $f$. Then, using the diagonal action of $\GG_n$ on $E \wedge S(f)$ we have a
$\GG_n$-isomorphism
\begin{equation}\label{eq:basic-j-1pic-for}
E_\ast S(f) = E_\ast S^k \otimes \ZZ_p\langle \chi_f \rangle = \Sigma^k E_\ast\langle \chi_f \rangle.
\end{equation}
Note that if $f$ is defined by a virtual real representation of $H$, as in  \Cref{exam:rep-spheres},
then $\chi_f$ is trivial or acts through $\{ \pm 1 \}$. We now have the following result. 
\end{rem}

\begin{prop}\label{rem:basic-j-1pic} Let $q\colon\GG_n \to H$ be a homomorphism and let $f$ be a 
spherical $H$-representation of virtual dimension $0$.  If $\chi_f$ is trivial, then $J(f) \in \kappa_n$.
\end{prop}

\begin{proof} This follows from \Cref{prop:basic-j-00} and \eqref{eq:basic-j-1pic-for}.
\end{proof}

Since $\ZZ \times B\Gl_1(L_{K(n)}S^0)$ is an
infinite loop space, the set $[BH,\ZZ \times B\Gl_1(L_{K(n)}S^0)]$ has the structure of an abelian group and
$S(f+g)$ is equivalent to $ S(f) \wedge S(g)$ with the diagonal $H$ action. This is compatible with the $J$-construction in the following sense.

\begin{prop}\label{prop:basic-j-2} Let $f,g\colon BH \to \ZZ \times B\Gl_1(L_{K(n)}S^0)$ be two maps and let $K \subseteq \GG_n$ 
be a closed subgroup. Then the natural map
\[
(E \wedge S(f))^{hK} \wedge_{E^{hK}} (E \wedge S(g))^{hK} \to (E \wedge S(f+g))^{hK}
\]
is an equivalence. 
\end{prop} 

\begin{proof} 
It is sufficient to show that the map induces an isomorphism on $K(n)_\ast$ homology. 
As a consequence of  \eqref{eq:mormoduleiso} we have that
\[
K(n)_\ast  (E \wedge S(f))^{hK} \cong \map(\GG_n/K,K(n)_\ast S(f))
\]
and, hence, $K(n)_\ast  (E \wedge S(f))^{hK}$ is a free $K(n)_\ast E^{hK}$ module of rank $1$. 
We can apply the spectral sequence
\[
\Tor_p^{K(n)_\ast E^{hK}}(K(n)_\ast X,K(n)_\ast Y)_q \Longrightarrow K(n)_{p+q} (X \wedge_{E^{hK}} Y)
\]
with $X = (E \wedge S(f))^{hK}$ and $Y = (E \wedge S(g))^{hK}$. Since the higher $\Tor$ terms vanish,
the result follows. 
\end{proof}

\Cref{prop:basic-j-2} has the following immediate consequence.

\begin{prop}\label{prop:basic-j-3} Let $K \subseteq \GG_n$ be a closed subgroup and let
$q:\GG_n \to H$ be a map to a finite group. The $J$-construction defines a homomorphism
\[
J(q,-,K): [BH,B\Gl_1(L_{K(n)}S^0)] \to \Pic(E^{hK}).
\]
\end{prop}

For later use we record  the following result, which can be proved
using a similar $\Tor$-spectral sequence, as in \Cref{prop:basic-j-2}. 

\begin{prop}\label{prop:shearing-iso} Let $S(f)$ be a $K(n)$-local spherical representation.
Then for any sequence of closed subgroups $K_1 \subseteq K_2 \subseteq \GG$ the natural map
\[
E^{hK_1} \wedge_{E^{hK_2}} (E \wedge S(f))^{hK_2} \longrightarrow (E \wedge S(f))^{hK_1}
\]
is an equivalence.
\end{prop} 

\begin{rem}\label{rem:compare-ss}
By construction, we have a map of $\GG_n$-spectra $S(f) \to E \wedge S(f)$, where the $\GG_n$-action
on $S(f)$ is defined via the map $q:\GG_n \to H$. The $H$-equivariant homotopy theory may be of
interest in its own right, and we'd like to make a comparison.
 
For $K$ a closed subgroup of $\GG_n$, the map $S(f) \to E \wedge S(f)$ and the composition of the inclusion $K\subseteq \GG_n$ with the map $q:\GG_n \to H$ define a map of augmented
cosimplicial $\GG_n$-spectra
\[
\xymatrix{
S(f) \ar[r] \ar[d] & F(H^{\bullet +1},S(f))\ar[d]\\
E \wedge S(f) \ar[r] & F_c(K^{\bullet +1},E \wedge S(f))
}
\]
and hence a diagram of spectral sequences
\[
\xymatrix{
H^s(H,\pi_tS(f)) \ar@{=>}[r]\ar[d] & \pi_{t-s}S(f)^{hH} \ar[d]\\
H^s(K,E_tS(f) \ar@{=>}[r] & \pi_{t-s}J(q,f,K)
}
\]
\end{rem}

\begin{ex}\label{ex:for-sec-9} There are variants of these constructions, for example $H$ doesn't have to be a finite group. Consider the case
of the reduced determinant map $q=\zeta:\GG_n \to \GG_n/\GG_n^1 \cong \ZZ_p$. We can take $B\ZZ_p$ to be the $p$-completion of $S^1$, as was done in \cite{BBGS} in constructing the determinant sphere.  We set up notions of continuity of actions can as done in that paper. Then the above results will have analogous counterparts. For example, consider 
\[
f\in [B\ZZ_p,B\Gl_1(L_{K(n)}S^0)] \subseteq  \pi_1 B\Gl_1(L_{K(n)}S^0) \cong (\pi_0L_{K(n)}S^0)^\times.
\] 

If $f$ maps to $1 \in E_0S^0$ under the Hurewicz map, then $\chi_f$ is the trivial character
and \Cref{prop:basic-j-00} implies that
\[
E_\ast J(f) = E_\ast(E \wedge S(f))^{h\GG_n} \cong E_\ast S^0
\]
as a Morava module. Hence $J(f) \in \kappa_n$. Using \Cref{prop:shearing-iso}, we have an equivalence of 
$E^{h\GG_n^1}$-modules
\[
E^{h\GG_n^1} \wedge J(f) \simeq (E \wedge S(f))^{h\GG_n^1} \simeq E^{h\GG_n^1} \wedge S(f) 
\simeq E^{h\GG_n^1}
\]
since $\GG_n^1$ is the kernel of $\zeta$. We now apply \Cref{lem:old-is-new} and the fact
that $H^0(\GG_n^1,E_0) \cong \ZZ_p$ to conclude that $J(f) \in \kappa(\GG_n^1)$.
\end{ex}

We end with a result relating the $J$-construction and the subgroups $\kappa(K)$ discussed in \cref{sec:subgroupfilt}, analogous to \Cref{ex:for-sec-9} but involving spherical representations of finite groups as set up at the beginning of this section.

\begin{prop}\label{prop:when-in-kappa} 
Suppose that $K\subseteq \GG_n$ is a closed subgroup with the property that the edge homomorphism
\[
\pi_0E^{hK} \longrightarrow H^0(K,E_0)
\]
is onto.
Let $H$ be a finite group, and suppose we are given a homomorphism
$q:\GG_n \to H$. Let
 $f\colon BH \to B\Gl_1(L_{K(n)}S^0)$ be a spherical representation of virtual dimension 0
 with $\chi_f$ trivial. Further suppose that the composite $K \to \GG_n \to H$ is trivial. Then
\[
J(f) \in \kappa(K).
\]
\end{prop} 

\begin{proof} By \Cref{rem:basic-j-1pic} we have $J(f) \in \kappa_n$.
We also have from \Cref{prop:basic-j-1} that
\[
E^{hK} \wedge J(f) \simeq (E \wedge L_{K(n)}S^0)^{hK} \simeq E^{hK}.
\]
Now apply \Cref{lem:old-is-new}. 
\end{proof}

 % !TEX root = pic-master.tex

\section{Truncating spectral sequences}\label{sec:trunc-ss}\label{subsec:trun-ss}
Here we explain a technique for taking apart some spectral sequences. This will be used in
\Cref{sec:subgroupsk2G21} and \Cref{sec:second-existence}, and any close reading could be
postponed until (or even if) it is needed. 

Let
\[
\cdots \longrightarrow X_s \longrightarrow X_{s-1} \longrightarrow \cdots \longrightarrow X_1 \longrightarrow X_0
\]
be a tower of fibrations in spectra. Write $\bX =\{X_s\}$ for this tower and $X = \holim X_s$ for the homotopy inverse limit.
If we write $F_s$ for the fiber of $X_s \to X_{s-1}$, with $X_{-1} = \ast$, then we have a spectral sequence
with
\[
E^{s,t}_1(\bX) = \pi_{t-s}F_s \Longrightarrow \pi_{t-s}X.
\]
Convergence is an issue which we won't address in this generality. 

\begin{ex}\label{rem:BKSS-cosmp} 
Such towers often arise from a Reedy fibrant cosimplicial spectrum
$A^{\bullet} = \{A^s\}$.  
Using the partial totalization functors $\mathrm{Tot}_s$, this writes $X=\Tot(A^\bullet)$
as a homotopy inverse limit  of a tower of fibrations $\bX= \{X_s \}$ 
with $X_s = \mathrm{Tot}_s(A^\bullet)$. The resulting spectral sequence is the Bousfield-Kan spectral
sequence. This begins with
\[
E^{s,t}_1(A^\bullet) = E^{s,t}_1(\bX) = \pi_{t-s}F_s \cong N^s \pi_t(A^{\bullet}) := N^s_t(A^\bullet),
\]
where $N^\bullet$ is the normalization functor on cosimplicial abelian groups. Write
\[
B^s_\ast(A^{\bullet}) \subseteq Z^s_\ast(A^{\bullet}) \subseteq N^s_\ast(A^{\bullet})
\]
for the coboundaries and cocyles of this cochain complex and
\[
\pi^s\pi_tA^{\bullet} = Z^s_t(A^{\bullet})/B^s_t(A^{\bullet}).
\]
The next page of the Bousfield-Kan spectral sequence then reads
\[
E_2^{s,t}(A^\bullet) \cong \pi^s\pi_tA^{\bullet} \Longrightarrow \pi_{t-s}\Tot(A^\bullet) = \pi_{t-s}X.
\]
\end{ex}

\begin{ex}\label{subsec:ANSS}
An important sub-example of \cref{rem:BKSS-cosmp} is the Adams-Novikov spectral sequence and its
variants. For us, the $K(n)$-local $E$-based Adams-Novikov spectral sequence is the most important, so we
expand some details about it here. When working in the $K(n)$-local category we write $X \wedge Y$ for
$L_{K(n)}(X \wedge Y)$.

If $Y$ is any spectrum, let $E^{\wedge \bullet}\wedge Y$ denote the standard cobar construction; that is, the augmented 
cosimplicial spectrum in the $K(n)$-local category
\begin{equation}\label{eq:ANSS-cobar-1}
\xymatrix{
Y \ar[r] & E \wedge Y \ar@<.75ex>[r] \ar@<-.75ex>[r] & \ar[l] E \wedge E \wedge Y \ar@<1.2ex>[r] \ar[r]\ar@<-1.2ex>[r]
& \ar@<.6ex>[l] \ar@<-.6ex>[l]\cdots\ .
}
\end{equation}
If $Y = E\wedge X$ for $X$ dualizable in the $K(n)$-local category, this becomes
\begin{equation}\label{eq:cobar-for-e}
\xymatrix{
E\wedge X \ar[r] & F_c({\GG_n}_+,E\wedge X) \ar@<.75ex>[r] \ar@<-.75ex>[r] &
\ar[l] F_c({\GG_n^2}_+,E\wedge X) \ar@<1.2ex>[r] \ar[r]\ar@<-1.2ex>[r]
& \ar@<.6ex>[l] \ar@<-.6ex>[l]\cdots\ ,
}
\end{equation}
where $F_c$ denotes the continuous fixed points as defined in \cite[Definition 2.2]{BBGS}, for example.
If $K \subseteq \GG_n$ is any closed subgroup, the Devinatz-Hopkins definition of $E^{hK}$ is
\[
(E\wedge X)^{hK} = \Tot\left( F_c({\GG_n}_+^{\bullet+1},E\wedge X)^K\right) \simeq \Tot\left( F_c({\GG_n}^{\bullet} \times \GG_n/K_+,E\wedge X) \right)
\]
and the associated Bousfield-Kan spectral sequence is isomorphic to a spectral sequence  
\[
E_2^{s,t}(K,X) = H^s(K,E_tX) \Longrightarrow \pi_{t-s}(E^{hK} \wedge X). 
\]
If $K = \GG_n$, then $E^{h\GG_n} \simeq L_{K(n)}S^0$ and this is the $K(n)$-local $E$-based
Adams-Novikov spectral sequence.   If $K$ is finite, then the spectral sequence is isomorphic to the classical
homotopy fixed point spectral sequence.  A concise reference for all this is
\cite[Section 2]{BBGS}, which is a synopsis of many sources, especially \cite[Appendix A]{DH}. 
\end{ex}

\begin{defn}\label{defn:trunc-towers-ss} Let $\bX$ be a tower of fibrations and fix an integer $0 \leq M \leq \infty$.
Define a new tower $\bX_{\leq M} = \{Y_s\}$
by
\[
Y_s =
\begin{cases}
X_s, & s \leq M;\\
X_M, & s \geq M.
\end{cases}
\]
\end{defn}

There is an evident map of towers $\bX = \{X_s\} \to \bX_{\leq M} = \{Y_s\}$ and the induced map on homotopy
inverse limits is the projection map $X \to X_M$. We will call the tower $\bX_{\leq M}$ the $M$th {\it truncated-tower}
and the resulting spectral sequence the $M$th {\it truncated spectral sequence.}

Note that the fiber of $Y_s \to Y_{s-1}$ is $F_s$ if $s \leq M$ and contractible if $s > M$. Write
\[
E_{r,\leq M}^{s,t}(\bX) = E_r^{\ast,\ast}(\bX_{\leq M})
\]
for the truncated spectral sequence. If $\bX$ is the tower of a cosimplicial spectrum $A^{\bullet}$ we will write 
\[
E_{r,\leq M}^{s,t}(\bX) = E_{r,\leq M}^{s,t}(A^\bullet).
\]
The following calculation is immediate.

\begin{lem}\label{lem:truncated-e2} Let $A^\bullet$ be a Reedy fibrant cosimplical spectrum. For the $M$th
truncated spectral sequence we have
\[
E_{2,\leq M}^{s,t}(A^\bullet) \cong
\begin{cases}
\pi^s\pi_t A^\bullet,& s < M;\\
N^M_t(A^\bullet)/B^M_t(A^\bullet),&s=M;\\
0,& s > M.
\end{cases}
\]
The induced map 
\[
E_2^{s,\ast}(A^\bullet) = \pi^s\pi_tA^\bullet \to E_{2,\leq M}^{s,t}(A^\bullet)
\]
is an isomorphism if $s <  M$, and zero for $s > M$. 
If $s = M$, this map is the standard injection. The cokernel is isomorphic to $N_\ast^M(A^\bullet)/Z^M_t(A^\bullet)$, which injects into $N_t^{M+1}(A^\bullet)\subseteq \pi_t A^{M+1}$.
\end{lem}

We also have a relative version of this construction. 

\begin{defn}\label{defn:trunc-towers-rel} Suppose we have two integers $M > K \geq 0$. Then we have 
a map of towers $\bX_{\leq M} \to \bX_{\leq (K-1)}$ and we let $\bX_K^M = \{X_K^s\}$ be the tower obtained
by taking level-wise fibers. Then we have
\[
X^s_K =
\begin{cases}
\hbox{fiber of $X_s \to X_{K-1}$}, & K \leq s \leq M;\\
\ast, & \mathrm{otherwise}.
\end{cases}
\]
\end{defn}
Note that
\[
\hbox{fiber of $X^s_K \to X^{s-1}_K$} =
\begin{cases}
F_s, & K \leq s \leq M;\\
\ast, & \mathrm{otherwise}.
\end{cases}
\]

\begin{lem}\label{lem:truncated-e2-bis} Let $A^\bullet$ be a Reedy fibrant cosimplical spectrum. For the relative truncated spectral sequence we have
\[
E_{2}^{s,t}(\bX_K^M) \cong
\begin{cases}
Z_t^{K}(A^\bullet) & s = K;\\
\pi^s\pi_t A^\bullet,& K < s < M;\\
N^M_t(A^\bullet)/B^M_t(A^\bullet),&s=M;\\
0,& \mathrm{otherwise}.
\end{cases}
\]
\end{lem}

One practical application of these truncations is that they can be used to ``break up'' differentials and turn various $r$-cycles into permanent cycles. Specifically, we will use the following straightforward geometric boundary result. 

\begin{lem}\label{lem:geo-bound}
Let $\bX$ be a tower of fibrations, and let $0 \leq K < M < \infty$  be two integers. Consider the cofiber sequence of towers
\[ \bX_{K}^M \xrightarrow{f} \bX_{\leq M} \xrightarrow{g} \bX_{\leq K-1} \xrightarrow{\delta} \Sigma \bX_{K}^M,\]
with a corresponding cofiber sequence in the limit
\[ X_K^M \xrightarrow{f} X_{M} \xrightarrow{g} X_{K-1} \xrightarrow{\delta} \Sigma X_{K}^M. \]

Let $r \geq 1$, and suppose $x \in E_r^{s,t}(\bX_{\leq M})$ supports a differential
\[ d_r(x) = y \in E_r^{s+r, t+r-1}(\bX_{\leq M}),\]
where $0 \leq s < K \leq s+r$. Then
\begin{enumerate}
\item $g_*x \in E_r^{s,t}(\bX_{\leq K-1})$ is a permanent cycle representing a homotopy class $[g_* x]\in \pi_{t-s} X_{K-1}$, and
\item ${\delta}_* [g_*x] $ is detected by $ y \in E_r^{s+r,t+r-1} (\bX_{\leq M}) \cong E_r^{s+r, t+r-1} (\bX_K^M)$.
\end{enumerate}

Conversely, suppose $0\leq s < K$, and let $x \in E_r^{s,t}(\bX_{\leq M}) $ be such that $g_*x \in E_r^{s,t}(\bX_{\leq K-1})$ is a permanent cycle representing $[g_* x] \in \pi_{t-s} X_{K-1}$. If $\delta_*[g_*x] = 0 \in \pi_{t-s-1} X_K^M $, then $x$ is a permanent cycle.

\end{lem}
\begin{proof}
The proof is a chase of diagrams and definitions, such as those from \cite[VI.2]{GoerssJardine}, which is a gloss on \cite{BK-LectureNotes}. For simplicity of notation, in this proof we let $F_s$ denote the homotopy fiber of $(\bX_{\leq M})_s \to (\bX_{\leq M})_{s-1}$.

So, suppose $x \in E_r^{s,t}(\bX_{\leq M})$ is given; by definition, this means $x$ is represented by a map
\[ x: S^{t-s} \to F_s \]
which has a lift $\tilde x $ making the following diagram commute
\[
\xymatrix{
X_{s+r-1} \ar[r] & X_s \\
S^{t-s}\ar[u]^{\tilde x} \ar[r]^{x} & F_s \ar[u].
}
 \]
 In particular, since $K-1 \leq s+r-1 $, there are no obstructions to lifting $x$ to the intermediate term $X_{K-1} = \lim \bX_{\leq K-1}$, proving item (1).
 
 The formula $d_r(x) = y$ means that $y$ is represented by the (desuspension of the) composite
 \[  S^{t-s} \xrightarrow{\tilde x} X_{s+r-1} \to \Sigma F_{s+r}. \]

Now note that we have a commutative diagram of fiber sequences
\[ 
\xymatrix{
X_M \ar[r]\ar[d] & X_{K-1} \ar[r]\ar[d] & \Sigma X_K^M \ar[d]\\
X_{s+r} \ar[r] & X_{K-1} \ar[r] & \Sigma X_K^{s+r},
}
\]
and the right-hand bottom map gives the commutative diagram
\[
\xymatrix{
X_{s+r-1} \ar[r]\ar[d] & X_{K-1} \ar[d]\\
\Sigma F_{s+r} \ar[r] & \Sigma X_K^{s+r}.}
\]
Here, $y$ is represented by the image of $\tilde x $ along the left vertical map. On the other hand, $g_* x$ is represented by the image of $\tilde x $ along the top horizontal map. The image of either to the bottom corner is a class detecting $\delta_*[g_* x]$, proving (2).

The converse is straightforward.
\end{proof}

\part{$K(2)$-Local Computations at $p=2$}\label{part:two}

In this part of the paper, we get specific and work at $n=p=2$. The main goal  isto compute $\kappa_2$. To do this, we 
need to use the full arsenal of $K(2)$-local chromatic homotopy theory. Even this is not enough, as we
will also develop new tools, and call an academic family of collaborators. Every step of the computation has
required either new theory and deep computations. 

So, dear reader, get ready for a long and strenuous hike through the forest of intricacies that is $K(2)$-local 
chromatic homotopy theory at the prime 2!

\begin{figure}
\center 
\includegraphics[width=\textwidth]{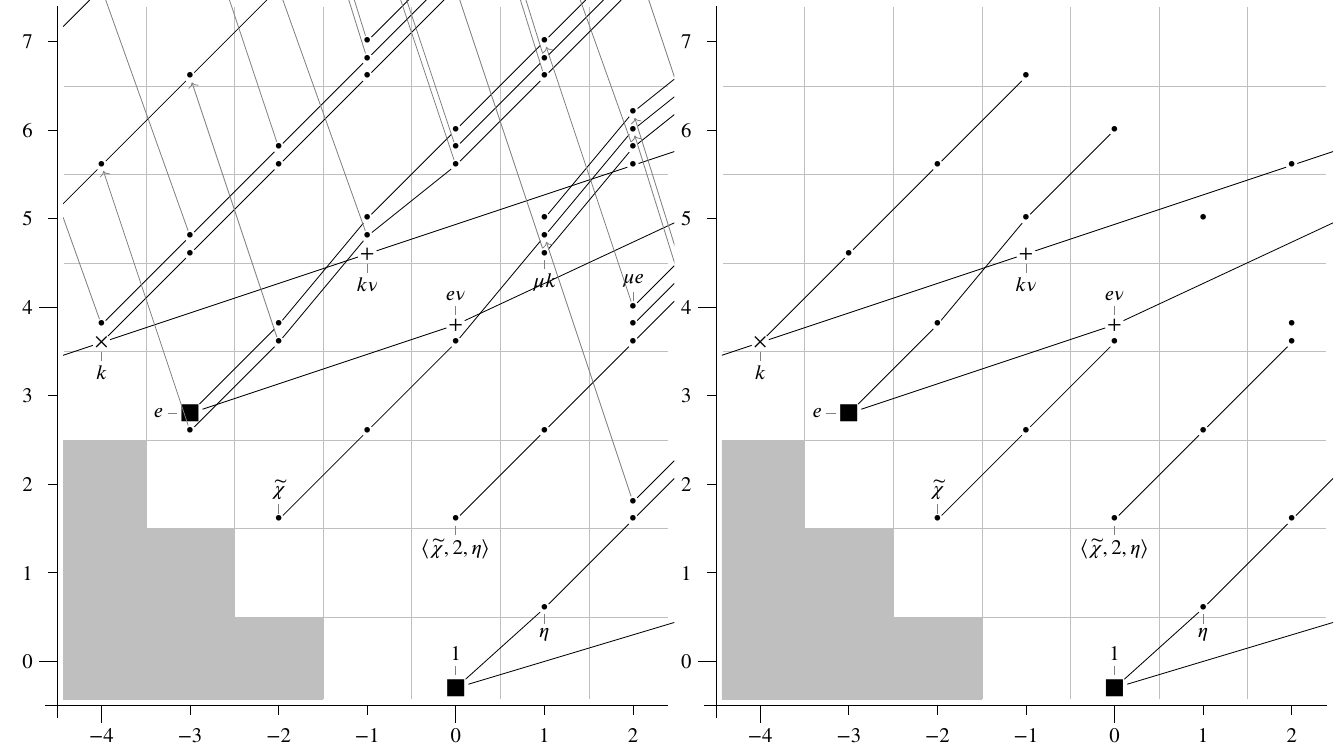}
\includegraphics[width=\textwidth]{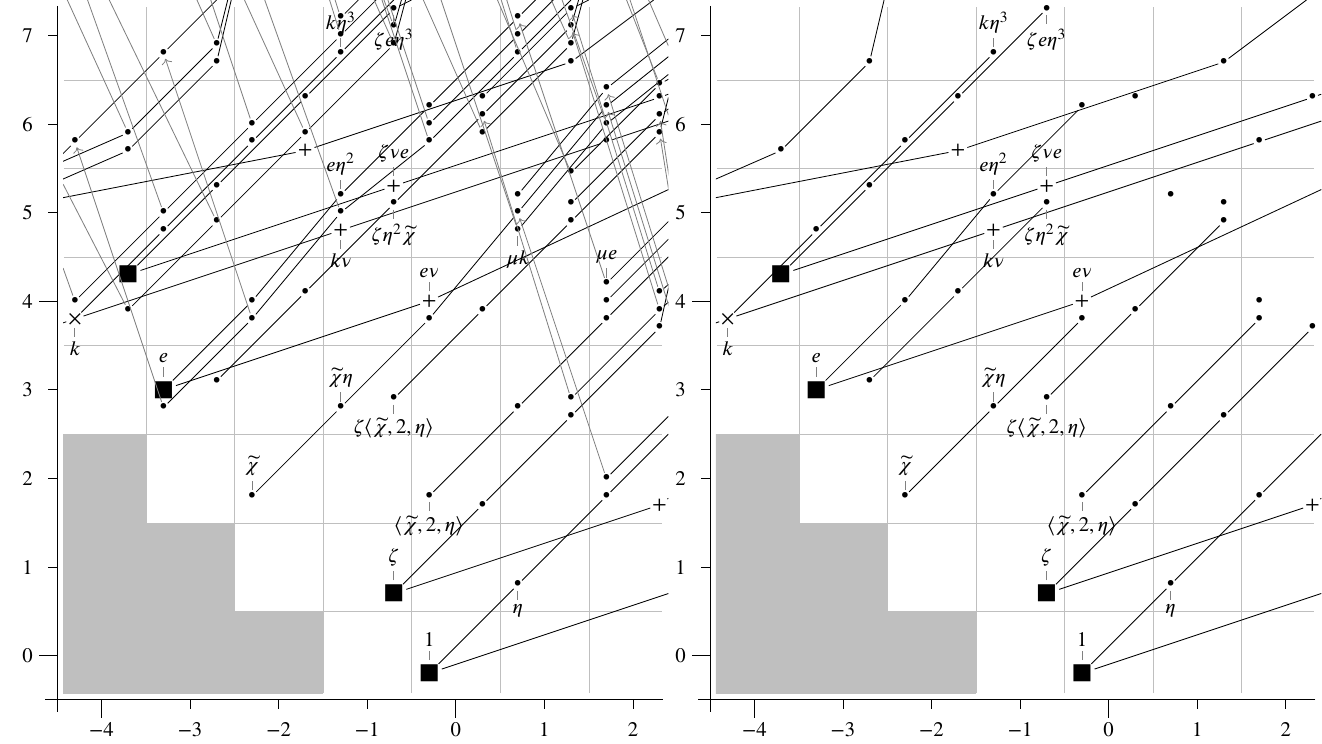}
\captionsetup{width=\textwidth}
\caption{The $E_3$ and $E_5$-pages of the homotopy fixed point spectral sequences $H^s(\GG_2^1, E_t) \Rightarrow \pi_{t-s}E^{h\GG_2^1}$ (top) and $H^s(\GG_2, E_t) \Rightarrow \pi_{t-s}E^{h\GG_2}$ (bottom). Here, $\blacksquare=\Z_{2}$, $\bullet=\Z/2$, $+=\Z/4$ and $\times=\Z/8$.
If, instead, we let $\blacksquare =\W$, $\bullet=\W/2$, $+=\W/4$ and $\times=\W/8$, then this is the $E_3$ and $E_5$-pages for $H^s(\mathbb{S}_2^1, E_t) \Rightarrow \pi_{t-s}E^{h\SS_2^1}$ (top) and $H^s(\mathbb{S}_2, E_t) \Rightarrow \pi_{t-s}E^{h\mathbb{S}_2}$ (bottom). These computations were done in \cite{BBGHPScoh}.}
%\label{fig:HFPSSG21}
\label{fig:HFPSSG2}
\end{figure}

 % !TEX root = pic-master.tex
 
\section{The subgroup filtration of $\kappa_2$}\label{rem:finite-subgrps}

In this section, we discuss the various subgroups of $\GG_2$ at $p=2$ and specify which subgroup filtration we will use to compute $\kappa_2$.

We customarily choose $\Gamma_2$ to be the formal  group a
supersingular elliptic curve $C$ defined over $\FF_2$; this has the advantage of direct access to the geometry
of elliptic curves, although this geometry will be tacit in this paper, subsumed into the body of work that has led to this 
paper (e.g. \cite{BeaudryTowards, BobkovaGoerss, BGH}). We then have an inclusion of the automorphism groups
\[
\xymatrix{
\Aut(\FF_4,C) \ar[r]^-{\subseteq} &\Aut(\FF_4,\Gamma_2) = \GG_2.
}
\]
The structure of $\Aut(\FF_4,C)$ is well understood. See, for example, \cite[Appendix A]{silverman}. 
We have $\Aut(\FF_4,C) \cong \Aut(C/\FF_4) \rtimes \Gal$, where we abbreviate the Galois group of $\FF_4/\FF_2$ as $\Gal$, and there is an isomorphism
\[
Q_8 \rtimes C_3 \cong \Aut(C/\FF_4),
\]
where $Q_8$ is the quaternion group of order $8$ and $C_3 = \Aut(Q_8)$ is cyclic of order $3$. This semidirect product is the binary tetrahedral group. 

\begin{defn}[{\bf The finite subgroups}]\label{defn:finite-subgrps} Let $p=2$ and $n=2$. Define the following finite
subgroups of $\GG_2$
\begin{align*}
G_{24} &= Q_8 \rtimes C_3 \cong \Aut(C/\FF_4) \\
G_{48} &= G_{24} \rtimes \Gal \cong \Aut(\FF_4,C) \\
C_4 &\subseteq G_{24} = \hbox{any of the cyclic groups of order $4$ in $Q_8$}\\
C_2 &\subseteq Q_8 = \hbox{the center of $Q_8$} \\
C_6 &= C_2 \times C_3 \subseteq G_{24}\ .
\end{align*}
\end{defn}

The groups $G_{48}$ and $G_{24}$ are maximal finite subgroups in $\GG_2$ and $\SS_2$ respectively.
See \cite{HennCent} for more on the finite subgroups of $\GG_2$. 

All of the homotopy fixed point spectral sequences $E^{hF}$ where $F$ are the subgroups of \Cref{defn:finite-subgrps}
have been heavily studied and, indeed, much of what follows depends on and is driven by that
knowledge. See \cite[Section 2]{BobkovaGoerss} for a pithy summary. Note that all of the choices
of $C_4$ are conjugate in $G_{24}$, so that the fixed point spectrum $E^{hC_4}$ is independent of the choice.

We record here two additional facts about the subgroups of $\GG_2$ which we use
later.

First, there is an open subgroup $K \subseteq \SS_2$ and a semi-direct product decomposition
\begin{equation}\label{eq:K-split-S2}
K \rtimes G_{24} \cong \SS_2. 
\end{equation}
In particular, $G_{24}$ is a quotient of $\SS_2$ as well as a subgroup. This can be found in many references;
for example, see \cite{BeaudryRes}. 

As in \cref{defn:zeta}, we also a have the closed subgroup $\GG_2^1$ which is the kernel of
$\zeta \colon \GG_2 \to \ZZ_2$. Similarly, $\SS_2^1$ is the kernel of the restriction of $\zeta$ to $\SS_2$.
If we let $\pi \in \SS_2$ be any class so that $\zeta(\pi) \in \ZZ_2$ is a topological generator, then $\pi$ defines a
section of $\zeta:\SS_2 \to \ZZ_2$ and we get semi-direct product decompositions
\[
\GG_2 \cong \GG_2^1\rtimes \ZZ_2 \quad\text{ and } \quad \SS_2 \cong \SS_2^1 \rtimes \ZZ_2,
\]
where $\SS_2^1 = \SS_2 \cap \GG_2^1$.

Any finite subgroup of $\SS_2$ is automatically in $\SS_2^1$; in particular, both $G_{24}$ and $\pi G_{24}\pi^{-1}$
are in $\SS_2^1$. They are evidently conjugate in $\SS_2$, but not in $\SS_2^1$. It is customary
to write 
\begin{equation}\label{eq:g24-prime}
G'_{24} =\pi G_{24}\pi^{-1}.
\end{equation}

In what follows, we will use a filtration of $\kappa_2$ coming from a sequence of subgroups of $\GG_2$ and the constructions of \cref{lem:incl-kappa-K}. 

\begin{defn}\label{defn:subgroup-filt} The \emph{subgroup filtration} of $\kappa_2$ is
\[
\kappa(\GG_2^1) \subseteq \kappa(G_{48}) \subseteq \kappa_2.
\]
\end{defn}

We can now give a road map for what is to come.

First, the associated graded of the subgroup filtration will be
\begin{align*}
\kappa(\GG_2^1)&\cong \Z/8 \oplus (\Z/2)^{ 2}\\
\kappa(G_{48})/\kappa(\GG_2^1) &\cong \Z/2\\
\kappa_2/ \kappa(G_{48}) &\cong \Z/8.
\end{align*}
The results giving these filtration quotients can be found in \Cref{thm:kappag21}, \Cref{thm:Qfirst}, and \Cref{thm:first-quotient}. Further, we show
that this filtration splits by comparing the subgroup filtration with the descent filtration, namely the one arising from the
Adams-Novikov Spectral Sequence. See \Cref{defn:descent-filt}. Thus, $\kappa_{2,r} \subseteq \kappa_2$
is the subset of elements so that $d_q(\iota_X) = 0$ for $q < r$ in the Adams-Novikov
Spectral Sequence for $X$ and for any choice of $\GG_2$-invariant generator
$\iota_X$ of $E_0X$.  If $G$ is a closed subgroup of $\GG_2$, let $\kappa_{r}(G)$ denote
the intersection $\kappa(G)\cap \kappa_{2,r}.$ 

Thus, we will have a diagram, where all the arrows are inclusions
\begin{equation}\label{eq:the-filtered-diag-kappa}
\xymatrix{
\kappa_2 & \ar[l] \kappa(G_{48}) & \ar[l] \kappa(\GG_2^1)\\
\kappa_{2,5} \ar[u]& \ar[l]  \ar[u] \kappa_5(G_{48}) &  \ar[l]  \ar[u] \kappa_5(\GG_2^1)\\
\kappa_{2,7}  \ar[u] &  \ar[l]  \ar[u] \kappa_7(G_{48}) & \ar[l]  \ar[u] \kappa_7(\GG_2^1).
}
\end{equation}
In the end we will show that the diagram of \eqref{eq:the-filtered-diag-kappa} maps isomorphically 
to following diagram of groups, again with all arrows inclusions.
\begin{equation}\label{eq:the-filtered-diag-kappa1}
\xymatrix{
(\ZZ/8)^2 \oplus (\ZZ/2)^3 & \ar[l] \ZZ/8 \oplus (\ZZ/2)^3 & \ar[l] \ZZ/8 \oplus (\ZZ/2)^2\\
(\ZZ/8)^2 \oplus \ZZ/2 \ar[u]& \ar[l]  \ar[u] \ZZ/8 \oplus \ZZ/2 &  \ar[l]_-=  \ar[u] \ZZ/8 \oplus \ZZ/2\\
(\ZZ/2)^2  \ar[u] &  \ar[l]  \ar[u] \ZZ/2 & \ar[l]_-=  \ar[u] \ZZ/2\ .
}
\end{equation}

We can then filter this diagram first vertically and then horizontally. Using \Cref{prop:kappaG48-filt} and 
\Cref{prop:kappa-filt} the associated graded becomes 
\begin{center}
\renewcommand{\arraystretch}{1.5}
\begin{tabular}{| c | c | c | c |} 
 \hline
\multicolumn{1}{|l|}{$s$} & \multicolumn{3}{c |}{} \\
\hline
\hline
3&& $\ZZ/2\{\wchi \eta\}$ & $\ZZ/2\{\zeta \langle \wchi, 2, \eta\rangle\}$\\ 
 \hline
5& $\ZZ/4\{k\nu\}$ &  & $\ZZ/4\{\zeta e\nu\} \oplus \ZZ/2\{\zeta\wchi\eta^2\}$\\
 \hline
7& $\ZZ/2\{k\eta^3\}$ & & $\ZZ/2\{\zeta e\eta^3\}$\\
 \hline
\end{tabular}
\renewcommand{\arraystretch}{1}
\end{center}

We have labelled the various generators here, using the following convention. Suppose $X \in \kappa_{2,r}$ and $
\iota_X \in H^0(\GG_2,E_0X) \cong \ZZ_2$ is a $\GG_2$-invariant generator. Then note that \Cref{defn:descent-filt} 
implies an inclusion 
\[ \kappa_{2,r}/\kappa_{2,r+1}  \longr E_r^{r,r-1}(S^0), \]
given by the formula
$
d_r(\iota_X) = \phi_r(X)\iota_X$.

Thus we have labelled the elements of the table by their names in $E_r^{r,r-1}(S^0)$, which in turn have names inherited 
from $ E_2^{r,r-1}(S^0) \cong H^r(\GG_2,E_{r-1})$.  The relevant part of the group cohomology $E_2^{r,r-1}(S^0) \cong 
H^r(\GG_2,E_{r-1})$, the $d_3$-differentials, and the resulting page $E_5^{r,r-1}(S^0)$ were computed in 
\cite{BBGHPScoh}, and are summarized in \Cref{fig:HFPSSG2}.

\begin{rem} It is instructive to compare these results with what happened at $p=3$ and height $n=2$. See
\cite{GHMRPicard}. There the analogous subgroup filtration collapses to
\[
\kappa(\GG_2^1) = \kappa(G_{24}) \subseteq \kappa_2
\]
and we had $\kappa(\GG_2^1) \cong \ZZ/3$ and $\kappa_2/\kappa(\GG_2^1) \cong \ZZ/3$. The descent
filtration was particularly simple
\[
0 = \kappa_{2,6} \subseteq \kappa_{2,5} = \kappa_2
\]
and
\[
\phi_5: \kappa_2 \longr H^5(\GG_2,E_4) \cong \ZZ/3 \times \ZZ/3
\]
was an isomorphism. This gave the splitting. 

The two summands of $\ZZ/8$ at $p=n=2$ are completely analogous to the two summands of $\ZZ/3$
at $p=3$; in both cases generators are detected in the decent filtration by $k\nu$ and $\zeta e\nu$, once we realize
the name of $\nu$ at $p=3$ is $\alpha$.

The existence of the extra summands of $\ZZ/2$ at $p=2$ is related to the fact that there is an extra homomorphism 
from $\SS_2$ to $\ZZ/2$ at the prime $2$, see \cite[Propositions 5.2, 5.3]{Henn-minicourse}. This is the same 
complication that gave rise to  the modification of the Chromatic Splitting Conjecture at $p=2$;
see \cite{BGH}. Finally, the class in $\kappa(G_{48})/\kappa(\GG_2^1) \cong \Z/2$ arises from
the $J$-construction and, as far as we know, cannot be produced in any other way. 
\end{rem}

 % !TEX root = pic-master.tex

\section{The $E^{h\GG_2^1}$-orientable elements of the exotic Picard group}\label{sec:subgroupsk2G21}

In this section we start with the calculation of $\kappa(\GG_2^1) \subseteq \kappa_2$, the subgroup of exotic invertible
elements which have an $E^{h\GG_2^1}$-orientation. The complexity of $H^\ast(\GG_2,E_\ast)$ makes this section technically forbidding. The key ideas are all in \cref{sec:third-graded}, and the key computational input comes from several sources. We refer to \cite{BBGHPScoh} for results about the cohomology $H^*(\GG_2,E_*)$ and $H^*(\GG_2^1, E_*)$. We also refer to \cite{BGH} for the computation of $\pi_* E^{h\GG_2^1}$ in a range. We freely use the notation established in these sources.

Recall that in \cref{prop:twist-constr}, we proved that
\[ \kappa(\GG_2^1) \cong H^1(\GG_2/\GG_2^1, 1+\KE)\]
where $\KE$ is the subgroup of $\pi_0E^{h\GG_2^1}$ of elements of positive Adams--Novikov filtration. 
We calculate the group $1+\KE$ in \Cref{prop:calc21} and show that the action of $\GG_2/\GG_2^1 \cong \ZZ_2$ is 
trivial  in \Cref{prop:inv-of-y} and \Cref{prop:enu-invariant}.  This allows us to determine $ \kappa(\GG_2^1)$ in 
\Cref{thm:kappag21}. In \cref{thm:comb-2-filts-fin}, we compute the descent filtration of $\kappa(\GG_2^1)$.
 
 \begin{lem}\label{lem:wandw0}
The unique non-zero class
\[
w \in \pi_{-2}E^{h\GG_2^1} \cong  \ZZ/2
\]
is detected by the non-zero class $\wchi \in H^2(\GG_2^1,E_0) \cong\Z/2$.
The class $w$ is the image of a class $w_0 \in \pi_{-2} L_{K(2)}S^0 $ of 
order $2$ under the map 
\[\pi_{-2} L_{K(2)}S^0 \longrightarrow \pi_{-2}E^{h\GG_2^1}.\]
Furthermore, 
\[\pi_{-2}L_{K(2)}S^0 \cong \ZZ/2 \oplus \ZZ/2,\] 
generated by $w_0$ and $\zeta w_0 \eta$. 
 \end{lem}
 \begin{proof}
By \cite[Corollary 9.1.8]{BGH} we have that 
\begin{align*}
\pi_{-1}E^{h\GG_2^1} &\cong \ZZ/2\\
\pi_{-2}E^{h\GG_2^1} &\cong \ZZ/2
\end{align*}
generated by elements detected by  $\wchi\eta \in H^3(\GG_2^1,E_2)$ and $\wchi \in H^2(\GG_2^1,E_0)$
respectively. The latter class detects $w$. It then follows from the fiber sequence 
\[ L_{K(2)}S^0 \to E^{h\GG_2^1} \to E^{h\GG_2^1}\] 
that we have an exact sequence
\[
0 \longrightarrow \ZZ/2 \longrightarrow \pi_{-2}L_{K(2)}S^0 \longrightarrow \ZZ/2 \longrightarrow 0.
\]
The generator of the subgroup in this extension is detected by $\zeta\wchi\eta \in H^4(\GG_2,E_2)$, and there is
a class $w_0 \in \pi_{-2}L_{K(2)}S^0 $ detected by $\wchi \in H^2(\GG_2,E_0)$ which generates the quotient. To prove the lemma, it suffices to show this sequence splits, i.e. that $w_0$ has order 2.

If the sequence is not split then the class in the kernel becomes trivial when passing to $\pi_{-2}L_{K(2)}V(0)$. 
But, by \cite[Theorem 8.2.6]{BGH} both classes are non-zero under the  map
\[
H^\ast(\GG_2,E_\ast) \longrightarrow v_1^{-1}H^\ast(\GG_2,E_\ast V(0)).
\]
By \cite[Theorem 8.3.5]{BGH} they detect non-zero classes in $\pi_\ast L_{K(1)}L_{K(2)}V(0)$.
Thus the sequence is split as needed. 
 \end{proof}
 
\begin{lem}[{\cite[Corollary 9.1.8, 9.1.9]{BGH}}]\label{lem:name-the-elements} 
There is an isomorphism
\[\pi_0E^{h\GG_2^1} \cong \Z_2\{1\} \oplus \Z/8\{x\} \oplus \Z/4\{y\},\]
where the generators $y $ and $x$ have the following description.
\begin{enumerate}[1)]
\item 
The class $y\in \pi_{0}E^{h\GG_2^1} $ can be chosen to be either of the two elements in the Toda bracket
$\langle w, 2, \eta \rangle$, and is detected by the Massey product
\[
\langle \wchi, 2, \eta \rangle \in H^2(\GG_2^1,E_2).
\]
The class $2y$ is then detected by $\wchi\eta^2 \in H^4(\GG_2^1,E_4)$.

\item The class $x \in \pi_{0}E^{h\GG_2^1}$ is detected by $e\nu \in H^4(\GG_2^1,E_4)$. The class $4x$ 
is detected by $e\eta^3 \in H^{6}(\GG_2^1, E_6)$.
\end{enumerate}
In particular,
\begin{equation}\label{eq:what-is-KE}
\KE \cong \ZZ/8\{ x\} \oplus \ZZ/4 \{ y \}.
\end{equation}
\end{lem}

\begin{cor}\label{cor:cohomologyG21filtrationvanishing}
In the spectral sequence
\[ E_2^{s,t}(\GG_2^1) = H^s(\GG_2^1,E_t) \Rightarrow \pi_{t-s}E^{h\GG_2^1},\]
we have
\begin{align*}
 E_{\infty}^{s,s}(\GG_2^1)  &= \begin{cases}
 \Z_2\{1\} & s=0 \\
 \Z/2\{\langle \wchi, 2,\eta\rangle\} & s=2 \\
 \Z/2\{\eta^2\wchi\}\oplus \Z/4 \{\nu e\} & s=4 \\
 \Z/2\{\eta^3e\}& s=6 \\
0 & \text{otherwise.}
\end{cases}
\end{align*}
Furthermore, $E_2^{0,0}(\GG_2^1)=E_{\infty}^{0,0}(\GG_2^1)$, and, for $0<s\leq 6$, we have $E_{s+1}^{s,s} = E_{\infty}^{s,s}(\G_2^1)$.
\end{cor}
\begin{proof}
The $E_5$-page of the spectral sequence was computed in \cite[Section 8]{BBGHPScoh}. See \cref{fig:HFPSSG2} for a figure illustrating that computation. 
The classes listed are permanent cycles by \Cref{lem:name-the-elements}. For degree reasons, they cannot be hit by differentials. Therefore, they survive to the $E_{\infty}$-page and account for all of the classes in the homotopy group $\pi_0E^{h\GG_2^1}$. It follows that there can be no non-zero classes in higher filtration at $E_{\infty}$. The claim about the stabilization of $E_{s+1}^{s,s}(\GG_2^1)$ is immediate from the computation in \cite{BBGHPScoh}. 
\end{proof}

\begin{prop}\label{prop:calc21} There is an isomorphism
\[
1+\KE \cong \ZZ/8 \oplus (\ZZ/2)^2.
\]
\end{prop}

\begin{proof}  In \cite[Corollary 9.1.9]{BGH}
it is shown that
$x^2 = 0$ and 
\[
y^2 = 2y + 2\l x
\]
for some integer $\l$ which only depends on its congruence class modulo $4$. Further,
$xy$ 
has Adams-Novikov filtration at least $6$, so we have
\[xy = 4\varepsilon x\]
for some $\varepsilon=0$ or $1$. Then
\begin{align*}
(1+x)^n &= 1 +nx\\
(1+y - \l x)^2&= 1. 
\end{align*} 
Thus $1+\KE \cong \ZZ/8 \oplus (\ZZ/2)^2$ generated by $1+x$, $1+ y - \l x$, and $1+2y$. 
\end{proof}

The next step is to show that the $\ZZ_2 \cong \GG_2/\GG_2^1$ action on $1+ \KE$ is trivial; this is equivalent to
showing that the action on $\KE$ is trivial. 

\begin{prop}\label{prop:inv-of-y} The class $y\in \pi_{0}E^{h\GG_2^1}$ is invariant under the 
action of $\GG_2/\GG_2^1$. 
\end{prop}

\begin{proof} The class $w$ of \cref{lem:wandw0}  is the image of a class $w_0$ of 
order $2$ under the map $\pi_{-2} L_{K(2)}S^0 \longrightarrow \pi_{-2}E^{h\GG_2^1}$. Once we have $w_0$ we can
form the Toda bracket $\langle w_0, 2, \eta \rangle$. So, by Part (1) of \Cref{lem:name-the-elements}, we see that 
$y$ is in the image of
\[
\pi_{0} L_{K(2)}S^0 \longrightarrow \pi_{0}E^{h\GG_2^1} .\qedhere
\]
\end{proof}

Our next goal is to prove the invariance of $x$. 

\begin{rem}\label{lem:inv-enu-1} 
In \cite[Lemma 7.2]{BBGHPScoh}, we proved that the action of $\GG_2/\GG_2^1$ on $H^*(\GG_2^1,E_t)$ is trivial in the range $0\leq t <12$. 
\end{rem}

The triviality of the action of $\GG_2/\GG_2^1$ in cohomology implies that the action of $\GG_2/\GG_2^1$ on $E_\infty^{s,s}(\GG_2^1)$ is trivial, but this is only the associated graded of $\pi_0E^{h\GG_2^1}$. In particular, this result is not enough to show the class $x$ is invariant, because $4x$ is detected in higher
filtration. So, we will use the technique of truncated spectral sequences, from \Cref{subsec:trun-ss} to discuss the invariance of $x$.

As in \Cref{subsec:ANSS} and \cite{DH}, the spectral  sequence 
\[
E_2^{s,t}=H^s(\GG_2^1,E_t) \Longrightarrow \pi_{t-s}E^{h\GG_2^1}
\]
is constructed as the $K(2)$-local $E$-based Adams-Novikov Spectral Sequence; that is, as
the Bousfield-Kan Spectral Sequence of the augmented cosimplicial $K(2)$-local spectrum 
\[
E^{h\GG_2^1} \to E^{\bullet +1} \wedge E^{h\GG_2^1} =A^{\bullet}. 
\]
Using the partial totalization functors $\mathrm{Tot}_n$, this writes $E^{h\GG_2^1}$ as a homotopy inverse limit 
of a tower of fibrations $\bX = \{X_s \}$
\[
\cdots \longrightarrow X_s \longrightarrow X_{s-1} \longrightarrow \cdots \longrightarrow X_1 \longrightarrow X_0
\]
with $X_s = \mathrm{Tot}_s(E^{\bullet +1} \wedge E^{h\GG_2^1})$. The limit $X $ is then $ E^{h\GG_2^1}$.
We will work with the truncated tower $\bX_{\leq 6}$ with $X_6 = \holim \bX_{\leq 6}$. See \Cref{subsec:trun-ss} for the 
definitions and notation. 

\begin{lem}\label{lem:truncate-at-6} In the Adams-Novikov tower $\bX$ under $X=E^{h\GG_2^1}$, the map
\[
\pi_0E^{h\GG_2^1} \to \pi_0X_6
\]
is an injection. 
\end{lem}

\begin{proof} In the Adams-Novikov spectral sequence
\[
H^s(\GG_2^1,E_t) \Longrightarrow \pi_{t-s}E^{h\GG_2^1}
\]
we have $E_\infty^{s,s} = 0$ for $s > 6$ by \cref{cor:cohomologyG21filtrationvanishing}. 
Now use \Cref{lem:truncated-e2}. 
\end{proof}

\begin{prop}\label{prop:enu-invariant} Let $x \in \pi_0E^{h\GG_2^1}$ be any class detected by
$e\nu \in H^4(\GG_2^1,E_4)$. Then $x$ is invariant under the action of $\GG_2/\GG_2^1$. 
\end{prop}

\begin{proof} Truncation is natural in towers of fibrations and the cobar construction
$X \to E^{\bullet +1} \wedge X $ is  natural in $X$. Let $\psi \colon E^{h\GG_2^1} \to E^{h\GG_2^1}$  be the
map induced by the action of a topological generator $\psi \in \GG_2/\GG_2^1$. This induces a map
\[
\psi: \bX_{\leq 6} \longr \bX_{\leq 6}
\]
of truncated towers and  hence induces a map of truncated spectral sequences. By \Cref{lem:truncate-at-6} it is sufficient
to show that $x \in \pi_0X_6$ is invariant; that is, $\psi(x) = x$. 

Te class $x$ is detected by 
\[e\nu \in H^4(\GG_2^1,E_4) \cong E_{2,\leq 6}^{4,4}(A^{\bullet}),\] 
and by Part (2) of \Cref{lem:name-the-elements}, we know that $4x$ is detected by
\[
e\eta^3 \in H^6(\GG_2^1,E_6) \subseteq E_{2,\leq 6}^{6,6}(A^{\bullet}).
\]

The advantage of working with $X_6$ is that this class represented by $e\eta^3$ is in fact a product of $\eta^3$ and a class detected by $e$. Indeed, we now turn to the class $e \in H^3(\GG_2^1,E_0)$, which is invariant by \Cref{lem:inv-enu-1}. It was shown in \cite[Proposition 8.3]{BBGHPScoh} that $d_3(e) = 0$ in the Adams-Novikov spectral sequence. By naturality, the same is true in the truncated spectral sequence, where longer differentials on $e$ are not possible. Thus there is a class $z \in \pi_{-3}X_6$ detected by $e$. We do not know that $z$ is invariant, but we may conclude
\[
\psi(z) = z+b
\] 
where $b$ has filtration $5$.

Now we have that $z\nu$ and $x$ are two classes of order $8$ in $\pi_0X_6$ which are detected by
$e\nu \in H^4(\GG_2^1,E_4)$. If we write
\[
v = z\nu-x 
\]
then $v \in E_{2,\leq 6}^{6,6}(A^{\bullet}) $ is torsion. But by \Cref{lem:truncated-e2}, we have an inclusion
\[H^6(\GG_2^1,E_6)\subseteq E_{2,\leq 6}^{6,6}(A^{\bullet}), \] 
whose cokernel is a summand of $\pi_\ast(A^{7})$, so is torsion-free. Thus $v$ must be detected in $H^6(\GG_2^1,E_6)$. The group $H^6(\GG_2^1,E_6)$ 
is invariant under the action of $\GG_2/\GG_2^1$, by \cite[Lemma 7.2]{BBGHPScoh}. Since this is the top filtration of $\pi_0X_6$, any class in $\pi_0 X_6$ detected by an element of $H^6(\GG_2^1,E_6) \subseteq E_{2,\leq 6}^{6,6}(A^\bullet)$ must be invariant. We conclude that $v$ is invariant, thus $z\nu$ is invariant if and only if $x$ is.

So finally, we show that $z\nu$ is invariant to complete the proof. Since $\psi(z) = z+b$ for $b$ detected by a class in $H^5(\GG_2^1,E_2)$, and $\nu $ is invariant, we have that $\psi(z\nu) = z\nu + b\nu$. We computed in \cite[Table 2]{BBGHPScoh} that all classes in $H^5(\GG_2^1,E_2)$ are multiples of $\eta$, thus $b$ must be a multiple of $\eta$. See also \cref{fig:HFPSSG2}. Since $\eta\nu = 0$, we get that $b\nu = 0$, and so $\psi(z\nu) = z\nu$ as needed. 
\end{proof} 

Now we combine the above results into the following consequence.

\begin{thm}\label{cor:pi0isfixed} The action of $\GG_2/\GG_2^1$ on 
$\pi_0E^{h\GG_2^1}$  is trivial. In particular, the map  $\pi_0 L_{K(2)}S^0 \to \pi_0E^{h\GG_2^1}$ defined by the unit of $\pi_0E^{\GG_2^1}$ is surjective.
\end{thm}

\begin{proof}
The first claim follows from \Cref{prop:inv-of-y} and \Cref{prop:enu-invariant} and the fact that the unit is fixed. The second claim follows from the fact that the image of the unit map on $\pi_0$ is the kernel of $\psi-1$ for $\psi$ a topological generator of $\GG_2/\GG_2^1$.
\end{proof}

As an aside, this also allows us to compute $\pi_{-1}L_{K(2)}S^0$.

\begin{cor}\label{cor:pi-1}
There is an isomorphism
\[\pi_{-1}L_{K(2)}S^0 \cong  \ZZ_2\{\zeta\} \oplus \ZZ/8\{\zeta x\} \oplus \ZZ/4\{\zeta y\} \oplus \Z/2\{\eta w_0\}. \]
\end{cor}

\begin{proof} The
fiber sequence 
\begin{equation}\label{eq:stand-fibre-seq}
\xymatrix{
L_{K(2)}S^0 \ar[r] &  E^{h\GG_2^1}  \ar[r]^-{\psi-1} &  E^{h\GG_2^1}
}
\end{equation}
gives an exact sequence
\[
\xymatrix{
   \pi_0E^{h\GG_2^1} \ar[r]^-\partial & \pi_{-1}L_{K(2)}S^0 \ar[r] & \pi_{-1}E^{h\GG_2^1}
}
\]
which we prove is split short exact. For injectivity of $\partial$, we use that the kernel is the image of $\psi-1$ acting on $ \pi_0E^{h\GG_2^1} $, which is zero by \cref{cor:pi0isfixed}. 
By \cite[Corollary 9.1.8]{BGH}, we know that
$\pi_{-1}E^{h\GG_2^1} \cong \ZZ/2$ generated by $ w \eta$, which is the image of $w_0 \eta$ by \Cref{lem:wandw0}. Since $w_0\eta$ has order $2$, the sequence is split. Note that the map $\partial$ is multiplication by $\zeta$, giving the claim.
\end{proof}

Combining the results proved above, we get the following explicit identification of the group of $E^{h\GG_2^1}$-orientable elements of $\kappa_2$.

\begin{thm}\label{thm:kappag21}The twisting construction $\alpha \mapsto X(\alpha)$ of \cref{defn:super-e-defn} gives an isomorphism
\[
 \ZZ/8 \oplus (\ZZ/2)^2\cong 1 + \KE \cong \kappa(\GG_2^1).
\]
\end{thm}
\begin{proof}
By  \Cref{cor:pi0isfixed}, the action of $\GG_2/\GG_2^1$ on $1 + \KE$ is trivial. Thus \cref{prop:twist-constr} gives the isomorphism between $\kappa(\GG_2^1)$ and $1 + \KE$. The identification of $1 + \KE$ is  \cref{prop:calc21}.
\end{proof}

Now we turn to identifying the descent filtration on $\kappa(\GG_2^1)$. To do that, we will use \Cref{lem:filtcompare}, which requires us to check one more technical condition. We check that condition in \Cref{prop:zetacompinjection}, for which we'll use the following input.

\begin{lem}\label{lem:keta3d5cycle}
In the homotopy fixed point spectral sequence for $G=\GG_2$ and $\GG_2^1$, the class 
\[\eta^3 k \in H^{7}(G, E_{6})\]
is a $d_5$-cycle. 
\end{lem}
\begin{proof}
In \cite[\S8]{BBGHPScoh}, we compute the $E_5$-page $E_{5}^{s,t}(G)$ of both homotopy fixed point spectral sequences in a range. See Figures 12 and 13 of that reference. In particular, we show that $k \in E_2^{4,0}(G)$ is a $d_3$-cycle. Therefore,
\[d_5(\eta^3 k) = \eta^3 d_5(k).\]
But 
\[ d_5(k) \in E_5^{9, 4}(G) \cong \begin{cases} \Z/4\{\nu k^2\} \oplus \Z/2\{\eta^2 ek\}  & G = \GG_2^1\\
 \Z/4\{\nu k^2, \zeta\nu e k\} \oplus \Z/2\{\eta^2 ek\}  & G=\GG_2.
\end{cases}\]
Since $\eta^3\nu=0$ and we prove that $\eta^5 ek =  d_3(\mu e k)$, it follows that 
\[d_5(\eta^3k) \in \eta^3(E_5^{9, 4}(G))=0. \qedhere\] 
\end{proof}

\begin{prop}\label{prop:zetacompinjection}
For $2\leq s\leq 7$, the sequence
\[0 \to E_{s}^{s-1,s-1}(\GG_2^1) \xrightarrow{\zeta}   E_{s}^{s,s-1}(\GG_2) \xrightarrow{i_*} E_{s}^{s,s-1}(\GG_2^1) \to 0   \] 
is exact, where $i_*$ is the restriction.
\end{prop}
\begin{proof}
The triviality of the action of $\ZZ_2\cong \GG_2/\GG_2^1$ (see \cref{lem:inv-enu-1}) gives such an exact sequence on $E_2$-pages. In  \cite[\S8]{BBGHPScoh}, we computed the $d_3$-differentials in a range. See \cref{fig:HFPSSG2} above. From that explicit determination of $E_{r}^{*,t}(\GG_2)$ and $E_{r}^{*,t}(\GG_2^1)$ for $2\leq r \leq 5 $ and $0 \leq t \leq 10$, we conclude that we have an exact sequence
\[0 \to E_{r}^{s-1,t}(\GG_2^1) \xrightarrow{\zeta}   E_{r}^{s,t}(\GG_2) \xrightarrow{i_*} E_{r}^{s,t}(\GG_2^1)  \to 0  \] 
 in that same range.

In particular, our sequence is exact at the $E_s$-page for $2\leq s\leq 5$.  Since the $E_6$ and $E_7$ pages agree, it remains to check that this is also true at the $E_7$-page. At $E_5$, we have the short exact sequence
\[\xymatrix@R=1pc{
0 \ar[r]& E_{5}^{6,6}(\GG_2^1)\ar@{=}[d] \ar[r]^-{\zeta}   & E_{5}^{7,6}(\GG_2) \ar[r]^-{i_*}  \ar@{=}[d]  & E_{5}^{7,6}(\GG_2^1)  \ar@{=}[d] \ar[r] &0   \\
0 \ar[r]&  \Z/2\{\eta^3e\}  \ar[r]  &  \Z/2\{\zeta\eta^3e, \eta^3 k \} \ar[r]  &   \Z/2\{\eta^3 k \}  \ar[r] & 0 .
}\] 
From \Cref{lem:name-the-elements}, we know that $\eta^3e$ is a permanent cycle in the $\GG_2^1$-spectral sequence, and from \Cref{cor:pi0isfixed}, we get that it remains a permanent cycle in the $\GG_2$-spectral sequence. Hence $\zeta\eta^3e$ is also a permanent cycle. From \cref{lem:keta3d5cycle}, $\eta^3k$ is a $d_5$-cycle, hence it survives to the $E_7$-page.

The $d_5$ differentials that have $E_5^{6,6}(\GG_2^1)$, $E_5^{7,6}(\GG_2)$, and $E_5^{7,6}(\GG_2^1)$ as targets originate in
\[ E_{5}^{1,2}(\GG_2^1) = \Z/2\{\eta\}, \quad E_{5}^{2,2}(\GG_2) = \Z/2\{\langle \wchi, 2,\eta\rangle, \eta\zeta\}, \quad E_5^{2,2}(\GG_2^1) = \Z/2\{ \langle \wchi, 2,\eta\rangle\} ,\]
respectively. However, $\eta$, $\eta\zeta$ and $\langle \wchi, 2,\eta\rangle$ are permanent cycles. For $\langle \wchi, 2,\eta\rangle$, see \cref{lem:name-the-elements} above. Therefore, the exact sequence remains the same at the $E_7$-page.
\end{proof}

Finally, we apply \cref{lem:filtcompare} to study the filtration on $\kappa(\GG_n^1)$.
\begin{thm}\label{thm:comb-2-filts-fin} In the filtration
\[
0 \subseteq \kappa_7(\GG_2^1) \subseteq \kappa_5(\GG_2^1) \subseteq \kappa_3(\GG_2^1) = \kappa(\GG^1_2) \cong
\ZZ/8 \times (\ZZ/2)^2
\]
we have isomorphisms
\[
\xymatrix@R=10pt{
\kappa(\GG_2^1)/ \kappa_5(\GG_2^1) \ar[r]_-\cong^-{\phi_3} & \ZZ/2\{\zeta \langle \wchi, 2, \eta\rangle\} \\
\kappa_5(\GG_2^1)/ \kappa_7(\GG_2^1) \ar[r]_-\cong^-{\phi_5} & \ZZ/4\{\zeta e \nu\} \times \ZZ/2\{\zeta\wchi\eta^2\}\\
\kappa_7(\GG_2^1) \ar[r]_-\cong^-{\phi_7} & \ZZ/2\{\zeta e\eta^3\}.
}
\]
\end{thm}
\begin{proof}
We apply  \cref{lem:filtcompare} with $N=7$. Conditions (i) and (ii) follow from \cref{cor:cohomologyG21filtrationvanishing}, while condition (iii) is checked in \cref{prop:zetacompinjection}. The identification of the filtration quotients follows from \cref{cor:cohomologyG21filtrationvanishing}.
\end{proof}

% !TEX root = pic-master.tex

\section{Duality Resolutions}\label{subsec:TDSS}

A crucial tool for analyzing the rest of $\kappa_2$ is the Duality Resolution of the half sphere $E^{h\SS_2^1}$.
In this section we recall the basic material.

The algebraic duality resolution, first introduced
by Goerss--Henn--Mahowald--Rezk and developed in detail in \cite{BeaudryRes}, is a device for isolating the 
contributions of finite subgroups $F$ to the cohomology of $\SS_2^1$;
similarly the topological duality resolution of \cite{BobkovaGoerss} isolates the contributions of
$\pi_\ast E^{hF}$ to $\pi_\ast E^{h\SS_2^1}$. Some of the basic finite subgroups of $\GG_2$ 
were defined and discussed in \Cref{rem:finite-subgrps}. See \Cref{defn:finite-subgrps} and \eqref{eq:g24-prime} in particular. 

There  is a more recent variant of these ideas that works directly with $\GG_2^1$, rather than 
through $\SS_2^1$. See \cite{BBHRes}. However, we will appeal to  the results of \cite{BeaudryTowards}, \cite{BGH},
and \cite{BBGHPScoh} which use the older resolution.

The {\emph Algebraic Duality Resolution} is an augmented exact complex of continuous
$\ZZ_2[[\SS_2^1]]$-modules
\begin{equation}\label{eq:dual-res-ssn-mod}
0 \to \ZZ_2[[\SS_2^1/G'_{24}]] \to \ZZ_2[[\SS_2^1/C_6]] \to \ZZ_2[[\SS_2^1/C_6]] \to
\ZZ_2[[\SS_2^1/G_{24}]] \to \ZZ_2 \to 0.
\end{equation}
For use below, let us write
\begin{equation}\label{eq:groupsF}
F_{-1}=\mathbb{S}_2^1, \quad F_0 =G_{24}, \quad F_1=F_2=C_6, \quad F_{3}=G_{24}' .
\end{equation}
Thus $\ZZ_2[[\SS_2^1/F_p]]$ is a typical term in the Duality Resolution. 

We then induce \eqref{eq:dual-res-ssn-mod} up to a sequence of $\ZZ_2[[\GG_2]]$-modules
\begin{align}\label{eq:dual-res-ggn-mod}
0 \to \ZZ_2[[\GG_2/G_{24}]] \to &\ZZ_2[[\GG_2/C_6]] \to \ZZ_2[[\GG_2/C_6]]\\
&\to \ZZ_2[[\GG_2/G_{24}]] \to \ZZ_2[[\GG_2/\SS_2^1]] \to 0.\nonumber
\end{align}
Since $G_{24}'$ is conjugate to $G_{24}$ in $\GG_2$, we have dropped the distinction. We write
\[
d=d_p:\ZZ_2[[\GG_2/F_{p}]] \to \ZZ_2[[\GG_2/F_{p-1}]],\qquad 0 \leq p \leq 3.
\]

We next apply $\Hom_{\ZZ_2}(-,E_\ast)$ to the sequence \eqref{eq:dual-res-ggn-mod} to obtain an
exact sequence of Morava modules
\begin{align*}
0 \to  \map(\GG_2/\SS_2^1,E_\ast) \to  &\map(\GG_2/G_{24},E_\ast) \to \map(\GG_2/C_6, E_\ast)\\
&\to \map(\GG_2/C_6,E_\ast) \to \map(\GG_2/G_{24},E_\ast) \to 0.\nonumber
\end{align*}
By \cref{rem:mormodulehom} this is isomorphic to an exact sequence of Morava 
modules
\begin{align}\label{eq:dual-res-mor-mod}
0 \to  E_\ast E^{h\SS_2^1} \to  &E_\ast E^{hG_{24}}  \to E_\ast E^{hC_6} \to E_\ast E^{hC_6}  \to E_\ast E^{h\G_{24}}  \to 0.
\end{align}

The project of the Topological Duality Resolutions from \cite{ghmr} or \cite{BobkovaGoerss} is to realize 
such sequences as maps of spectra.

\begin{rem}\label{rem:d-is-delta} There is some ambiguity here; for example, the Morava module
$\map(\GG_2/G_{24},E_\ast)$ is $24$-periodic, but the spectrum $E^{hG_{24}}$ is only $8\cdot 24 = 192$
periodic.  In the language of \Cref{defn:alg-periodicity} we can produce an algebraic periodicity class 
\[
\Delta \in H^0(G_{24},E_{24})
\]
as the discriminant of a deformation of the unique supersingular elliptic curve over $\FF_4$. Then $\Delta^8$
is a permanent cycle and
\[
\pi_{192}E^{hG_{24}} \longr H^0(G_{24},E_{192})
\]
is an isomorphism. Thus $\Delta^8$ detects a topological periodicity class. We will use 
these classes in our constructions below.
\end{rem}

The main theorem of \cite{BobkovaGoerss} produces a sequence of spectra
\begin{equation}\label{eq:top_dual-res-def}
E^{h\SS_2^1} \to E^{hG_{24}} \to E^{hC_6} \to E^{hC_6} \to \Sigma^{48}E^{hG_{24}}
\end{equation} 
realizing the sequence \eqref{eq:dual-res-mor-mod} of Morava modules, and called the {\emph Topological Duality Resolution}. All the maps in \eqref{eq:top_dual-res-def}
are algebraic in the sense of \Cref{def:alg-mor} for $d=\Delta$. Furthermore, all compositions are zero and all
Toda brackets  in \eqref{eq:top_dual-res-def} contain zero, so the sequence refines to a tower. 

\begin{defn}[{\bf Duality Spectral Sequences}]\label{not:all-the-duality-ss} For $F_s$ as in \eqref{eq:groupsF}, we get the following spectral sequences.
\begin{enumerate}

\item If $X$ is any spectrum, then the spectral sequence of the tower
gives the {\emph Topological Duality spectral sequence} for $X$
\[
E_1^{s,t}(X) = \pi_t(E^{hF_s} \wedge X) \Longrightarrow \pi_{t-s}(E^{h\SS_2^1} \wedge X). 
\]

\item Let $M$ be a Morava module. The sequence of $\ZZ_2[[\mathbb{S}_2^1]]$-modules \eqref{eq:dual-res-ssn-mod}
and the Schapiro Isomorphism gives the {\emph Algebraic Duality spectral sequence}
\[
E_1^{p,q}(M) = H^p(F_q,M) \Longrightarrow H^{p+q}(\SS_2^1,M).
\]
for any appropriate $\ZZ_2[[\mathbb{S}_2^1]]$-module $M$.  
\end{enumerate}
\end{defn}

We now have two ways of calculating $\pi_\ast E^{h\SS_2^1}$ from $H^\ast(F_p,E_\ast)$ encoded in the two ways
around the following square
\[
\xymatrix@C=40pt{
H^s(F_p,E_t  X) \ar@{=>}[d]_{ADSS} \ar@{=>}[r]^-{HFPSS} & \pi_{t-s}(E^{hF_p}\wedge X) \ar@{=>}[d]^{TDSS} \\
H_c^{s+p}(\SS_2^1,E_t X) \ar@{=>}[r]_-{ANSS} & \pi_{t-s-p}(E^{h\SS_2^1}\wedge X).
}
\]
We hope the acronyms are clear. In general these two methods have a relatively complicated relationship, but under 
certain hypotheses we can deduce some information. The hypotheses of the following result are crafted to ensure
there are no exotic jumps of filtration in related Adams-Novikov Spectral Sequences. The proof
is exactly the same as for \cite[Lemma 2.5.11]{BGH}.

\begin{lem}\label{lem:chasing-detectors} Let $x \in \pi_{n} (E^{h\SS_2^1}\wedge X)$. Suppose
\begin{enumerate}

\item the class $x$ is detected by $\alpha \in H^{p}(\SS_2^1,E_tX)$ in the ANSS, so $n=t-p$;

\item the class $x$ is detected by $y \in \pi_{t}(E^{hF_p}\wedge X)$ in the TDSS; and,

\item the class $y$ is detected by $\beta \in H^{0}(F_p,E_tX)$ in the HFPSS.
\end{enumerate}
Then $\alpha$ is detected by $\beta$ in the ADSS.
\end{lem}

 % !TEX root = pic-master.tex

\section{An upper bound for the $G_{48}$-orientable elements}

We continue to examine the subgroup filtration
\[
\kappa(\GG_2^1) \subseteq \kappa(G_{48}) \subseteq \kappa_2.
\]
of $\kappa_2$ given in \Cref{defn:subgroup-filt}. In this section we give an upper bound on the filtration quotient
$\kappa(G_{48})/\kappa(\mathbb{G}_2^1)$. This group is the group of elements in $\kappa_2$ which have an
$E^{hG_{48}}$-orientation, but cannot be given  an $E^{h\GG_2^1}$-orientation. We will show that
there is an injective homomorphism $\kappa(G_{48})/\kappa(\mathbb{G}_2^1) \to \ZZ/2$, hence this group can 
be at most of order $2$. In the next section we will show that this group has order 2, by constructing and examining a 
non-zero element in $\kappa(G_{48})/\kappa(\mathbb{G}_2^1)$. In some sense, these two sections contain some of
the main innovations of this paper, as it is here we depart considerably from the program laid out in
\cite{GHMRPicard} at the prime $3$. 

Recall from \Cref{defn:subgroupfilt} that $\kappa(K) \subseteq \kappa_2$ is the subgroup of invertible spectra
$X$ with an orientation class $z \in \pi_0(E^{hK} \wedge X)$; then $z$ extends to an equivalence
$\varphi_z: E^{hK} \simeq E^{hK} \wedge X  $  of $E^{hK}$-modules and defines a $\GG_2$-invariant
generator $\iota_X \in E_0X$. As in \Cref{defn:descent-filt}, we introduce a homomorphism
\begin{equation}\label{eq:phi31}
\phi_3^1: \kappa_2 \to  H^{3}(\GG_2^1, E_2) \cong \Z/2\{\wchi\eta\}
\end{equation}
by the equation
\[
d_3(\iota_X) = \phi_3^1(X)\iota_X
\]
where $d_3$ is the differential in the homotopy fixed point spectral sequence
\[
H^s(\GG_2^1,E_tX) \Longrightarrow \pi_{t-s}(E^{h\GG_2^1} \wedge X).
\]
We will prove that $\kappa(\GG_2^1)$ is the kernel of the composition
\[
\xymatrix{
\kappa(G_{48}) \ar[r]^-{\subseteq} & \kappa_2 \ar[r]^-{\phi_3^1} & H^{3}(\GG_2^1, E_2) \cong \Z/2\{\wchi\eta\}
}
\]
and thus obtain an injective group  homomorphism $ \kappa(G_{48})/\kappa(\GG_2^1) \rightarrow \Z/2$. 
The names of various elements in cohomology were explained in Section 7. See \Cref{fig:HFPSSG2}.
 
The main idea is this. Suppose we are  given $X\in \kappa(G_{48})$. If we can factor the orientation class  
\[
\xymatrix{
&E^{h\GG_2^1} \wedge X \ar[d]\\
S^0 \ar@{-->}[ur] \ar[r]_-z & E^{hG_{48}} \wedge X
}
\]
then we would have $X \in \kappa(\GG_2^1)$. Note that $\kappa(\GG_2^1) \cong \kappa(\SS_2^1) $ by \Cref{prop:galois-doesnt-matter}, so it is equivalent to lift to $E^{h\SS_2^1}$. We will examine the obstructions to this latter lifting the using the
Topological Duality resolution.

\subsection{Untwisting the Topological Duality Resolution} \label{sec:untwistduality}

A key observation for this section is that if $X \in \kappa(G_{48})$ then, the first two pages of 
the Topological Duality spectral sequences for $S^0$ and for $X$ agree. The claim for the first page is clear, and the point of this 
subsection is to prove the statement for the second pages.

The Topological Duality resolution was defined and discussed in \Cref{subsec:TDSS}. In the next result,
the groups $F_i \subseteq \GG_2$ are the finite subgroups that appear in the Duality
resolution: $F_0=F_3=G_{24}$ and $F_1=F_2=C_6$. The resolution itself is
\[
\xymatrix{
E^{h\SS_2^1} \ar[r] & E^{hG_{24}} \ar[r]^-{d} & E^{hC_6} \ar[r]^-{d}  &
	E^{hC_6}  \ar[r]^-{d} & \Sigma^{48}E^{hG_{24}}\ .
}
\]
To prove our untwisting result, we need some facts about the Hurewicz homomorphism.

\begin{lem}\label{lem:ag-duality-hurew}
Let $H_1, H_2 \in \{C_6, G_{24}\}$.
The Hurewicz map
\[ \pi_tF(E^{hH_1}, E^{hH_2}) \to  \Hom_{\Mor}(  E_0E^{hH_1}, E_tE^{hH_2})\]
is 
\begin{enumerate}
\item injective if $t\equiv 0$ modulo $48$, and 
\item bijective if $t=0$.
\end{enumerate}
\end{lem}

\begin{proof}
By \cite[Proposition 2.7]{ghmr} we have a commutative diagram
\[\xymatrix{
\pi_tE[\![\GG_2/H_1]\!]^{hH_2} \ar[r] \ar[d]^-{\cong}& (E_t[\![\GG_2/H_1]\!])^{H_2} \ar[d]^-{\cong}\\
\pi_tF(E^{hH_1}, E^{hH_2}) \ar[r] & \Hom_{\Mor}(E_0E^{hH_1}, E_tE^{hH_2})
}\]
where the horizontal maps are the Hurewicz homomorphism. The top arrow is a limit of Hurewicz maps 
\begin{align}\label{eq:hurF}
\pi_t E^{hF} \to H^0(F, E_t)
\end{align}
for various subgroups $F \subseteq G_{24}$ (see \cite[Sec. 2]{BobkovaGoerss}, \cite{BBHS}, \cite{tbauer}, 
\cite{tmfbook} and \cite{MR} for computation of $\pi_* E^{hF}$). The maps \eqref{eq:hurF} are all injective under our 
condition that $t\equiv 0\mod 48$. Indeed, any class in $\pi_{t} E^{hF}$ has Adams--Novikov filtration $0$ in these 
cases. Furthermore, when $t=0$, the Hurewicz maps  \eqref{eq:hurF}  are isomorphisms for $F \subseteq G_{24}$. This 
proves the claims.
\end{proof}

Now let $X \in \kappa(G_{48})$ and let
\[
z \in \pi_0(E^{hG_{48}}\wedge X)
\]
be a choice of $E^{hG_{48}}$-orientation. For $0 \leq i \leq 3$ we get an induced orientation
\[
z \in \pi_0(E^{hF_i} \wedge X)
\]
and hence an equivalence $\varphi_i : E^{hF_i} \to E^{hF_i} \wedge X$ of $E^{hF_i}$-modules. 

 \begin{prop}\label{prop:untwist-the-TDSS-1} The following diagram commutes up to homotopy 
 \[
\xymatrix@C=30pt{
E^{hG_{24}} \ar[r]^-{d} \ar[d]_{\varphi_0} & E^{hC_6} \ar[r]^-{d} \ar[d]_{\varphi_1}  &
	E^{hC_6}  \ar[r]^-{d} \ar[d]^{\varphi_2} & \Sigma^{48}E^{hG_{24}} \ar[d]^{\Sigma^{48} \varphi_3}\\
E^{hG_{24}} \wedge X \ar[r]^-{d\wedge 1_X} & E^{hC_6} \wedge X \ar[r]^-{d\wedge 1_X} &
	E^{hC_6} \wedge X   \ar[r]^-{d\wedge 1_X} & \Sigma^{48}E^{hG_{24}} \wedge X.
}
\]
\end{prop}

\begin{proof} By \Cref{lem:ag-duality-hurew} and the fact that $E^{hF}\wedge X \simeq E^{hF}$ for $F=C_6, G_{24}$,
we need only check that the diagram commutes after we apply $E_\ast(-)$.
This follows from \Cref{lem:orientations-nat-2}.
\end{proof}

Note we have made no claim about extending \cref{prop:untwist-the-TDSS-1} to the augmentation from
$E^{h\SS_2^1} \to E^{hG_{24}}$; indeed, the existence of a map completing the diagram
\[
\xymatrix{
E^{h\SS_2^1} \ar[r] \ar@{-->}[d] & E^{hG_{24}} \ar[d]^{\varphi_0}\\
E^{h\SS_2^1} \wedge X \ar[r] & E^{hG_{24}} \wedge X
}
\]
 is equivalent to the assertion that 
$X \in \kappa(\SS_2^1)$.

As an immediate consequence of  \cref{prop:untwist-the-TDSS-1}, we have the following conclusion.
\begin{cor}\label{cor:tdss-agrees}  Let $X \in \kappa(G_{48})$ and let $z \in \pi_0(E^{hG_{48}}\wedge X)$
be a choice of an $E^{hG_{48}}$-orientation. Then $z$ determines an isomorphism between the $E_2$-term of the 
Topological Duality spectral sequence for $S^0$ and the $E_2$-term of the Topological Duality
spectral sequence for $X$.  
\end{cor}

\subsection{Analyzing the obstruction} Let $X \in \kappa(G_{48})$ and let $z \in \pi_0(E^{hG_{48}} \wedge X)$
be an $E^{hG_{48}}$-orientation. If $F \subseteq G_{48}$ is any subgroup, we write
$z \in \pi_0(E^{hF} \wedge X)$ for the induced $E^{hF}$- orientation. If $F$ is the trivial subgroup we will also write
$\iota_X \in E_0X$ for the induced $\GG_2$-invariant generator.

We now examine the obstructions to lifting $z:S^0 \to E^{hG_{24}} \wedge X$
up the duality resolution tower
\[
\xymatrix{
&E^{h\SS_2^1} \wedge X\ar[d]\\
&X_2 \ar[r] \ar[d]& \Sigma^{46} E^{hG_{24}}\\
&X_1 \ar[r] \ar[d]& \Sigma^{-1} E^{hC_{6}}\wedge X\\
S^0 \ar[r]_-z \ar@{-->}[ur] \ar@{-->}[uur] \ar@{-->}[uuur]&
		E^{hG_{24}} \wedge X \ar[r] & E^{hC_{6}}\wedge X\ .
}
\]
This is equivalent to computing $d_r(z)$ in the Topological Duality spectral sequence
\[
E_1^{s,t}(X) = \pi_tE^{hF_s}\wedge X \Longrightarrow \pi_{t-s}E^{h\SS_2^1}\wedge X.
\]
Note that this is a spectral sequence of $\WW$-modules as $\pi_0E^{h\mathbb{S}_2^1}$ acts on the spectral sequence.

We first record some results of \cite{BGH} about the Topological Duality spectral sequence for $S^0$. See Remark 9.1.5, Theorem 9.1.7, and
especially Figure 6 of \cite{BGH}, which we also included in \cref{fig:ADSSE2}.

\begin{lem}[{\cite[\S 9]{BGH}}]\label{lem:TDSSsphereSmall}
In the Topological Duality spectral sequence for $S^0$, we have
\begin{enumerate}

\item $E_2^{0,0}(S^0) \cong \W$ generated by a class detecting the unit element of $\pi_0E^{h\SS_2^1}$.

\item $E_2^{2,0}(S^0) \cong \WW/2$ generated by a class $\overline{b}_0$ detecting the class $\wchi$. 

\item $E_2^{2,1}(S^0) \cong \WW/2$ generated by class $\overline{b}_0\eta$ detecting the class $\wchi\eta$.

\item $E_2^{3,2}(S^0) = 0$.
\end{enumerate}
\end{lem} 

\begin{figure}[h]
\center
\includegraphics[width=0.5\textwidth]{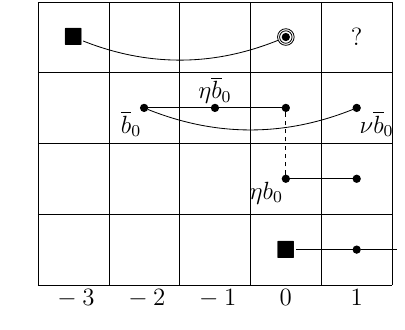}
\captionsetup{width=\textwidth}
\caption{The $E_2$-page of the Topological Duality spectral sequence for $S^0$. The vertical axis is the
Topological Duality spectral sequence filtration $s$ and the horizontal axis is $t-s$. Further,
$\blacksquare = \WW$, $\bullet=\F_4$, and the circled bullet represents $E_0^{3,3} \cong \WW/8$.
Horizontal lines are $\eta$-multiplications and curved lines are $\nu$-multiplications.}
\label{fig:ADSSE2}
\end{figure}

\begin{prop}\label{prop:keylemobstruct}
Let $X \in \kappa(G_{48})$ and let $E_r^{s,t}(X)$ be the Topological Duality spectral sequence
\[
E_1^{s,t}(X) = \pi_t(E^{hF_s}\wedge X) \Longrightarrow \pi_{t-s}(E^{h\SS_2^1}\wedge X).
\]
Let $z \in \pi_0(E^{hG_{24}}\wedge X) = E_1^{0,0}$ be the induced $E^{hG_{24}}$-orientation.
Then 
\begin{enumerate}
\item $d_1(z) = 0$,
\item
$d_2(z) \in E_2^{2,1} = \WW/2\{\eta\overline{b}_0z\}$,
\item if $d_2(z) = 0$, then $z$ is a permanent cycle and $X \in \kappa(\GG_2^1) \cong \kappa(\SS_2^1)$, and
\item if $d_2(z) \ne 0$, then $\pi_{-1}(E^{h\SS_2^1} \wedge X) = 0$.
\end{enumerate}
\end{prop}

\begin{proof} \Cref{prop:untwist-the-TDSS-1} implies that we have an isomorphism
\[
E_2^{\ast,\ast}(S^0) \cong E_2^{\ast,\ast}(X)
\]
sending $a$ to $az$. Then (1)--(3) follow from \Cref{lem:TDSSsphereSmall} and \Cref{fig:ADSSE2}.
For point (4)  note that $\pi_{-1}(E^{h\SS_2^1} \wedge X)$ is a $\pi_0E^{h\SS_2^1}$-module
and hence a $\WW$-module. The smallest non-zero $\WW$ module is $\FF_4$; hence
if $d_2(z) \ne 0$, we must have $\pi_{-1}(E^{h\SS_2^1} \wedge X) = 0$.
\end{proof}

We also have that following result, which we will use in the proof of \Cref{prop:keypropd2meansd3}.

\begin{lem}\label{lem:keylemobstruct-1} Let $X \in \kappa(G_{48})$.  Let $z \in \pi_0(E^{hG_{24}}\wedge X)
= E_1^{0,0}$ be the induced $E^{hG_{24}}$-orientation and $\iota_X \in E_0X$ the induced
$\GG_2$-invariant generator. Then we have
\[
\pi_{-2}E^{h\SS_2^1} \wedge X \cong \WW/2.
\]
The generator $y$ is detected by
\[
\overline{b}_0\, z \in E_2^{0,2}(X)
\]
in the Topological Duality spectral sequence and by
\[
\wchi\,\iota_X \in H^2(\SS_2^1,E_0X)
\]
in the Adams-Novikov spectral sequence. 
\end{lem}

\begin{proof} All but the last statement follow from \Cref{lem:TDSSsphereSmall} and \Cref{fig:ADSSE2}. 
The last statement follows from \Cref{lem:chasing-detectors} with $\beta = \overline{b}_0\,z$
and the fact that $\overline{b}_0$ detects $\wchi$ in the Algebraic Duality spectral sequence, as shown in \cite[Lemma 5.2.10]{BGH}.
\end{proof}

\Cref{prop:keylemobstruct} defines a function
\[
\kappa(G_{48}) \longrightarrow \WW/2
\]
sending $X$ to $d_2(z)$.  Because the Topological Duality spectral sequence does not have good
multiplicative properties it is not clear if this function is a homomorphism. We would also like to cut down the image
to $\ZZ/2$. To do all this, we relate it to a differential in the Adams--Novikov spectral sequence. 

\begin{prop}\label{prop:keypropd2meansd3}
Let $X \in \kappa(G_{48})$ and let $z \in \pi_0(E^{hG_{24}}\wedge X)$ be an $E^{hG_{24}}$-orientation induced by an $E^{hG_{48}}$-orientation
of $X$. Let 
\[\iota_X \in H^0(\GG_2,E_0X) \subseteq H^0(\GG_2^1,E_0X) \subseteq E_0X\] 
be the $\GG_2$-invariant generator determined
by $z$. 

Then $d_2(z) \neq 0$ in the Topological Duality spectral sequence computing $\pi_*(E^{h\SS_2^1}\wedge X)$ if and only if
\[
d_3(\iota_X) = \wchi \eta\, \iota_X \in H^3(\GG_2^1,E_2) \cong \ZZ/2\{\wchi \eta \iota_X\} 
\]
in the Adams-Novikov spectral sequence computing $\pi_*(E^{h\mathbb{G}_2^1}\wedge X)$.
\end{prop}

\begin{proof}  We use \Cref{lem:keylemobstruct-1}. Let $y \in \pi_{-2}(E^{h\SS_2^1}\wedge X)$ be the class detected
by $\overline{b}_0 z$ in the Topological Duality spectral sequence and $\wchi\iota_X$ in the Adams-Novikov 
spectral sequence. Suppose
\[
d_2 (z) = \alpha\, \overline{b}_0\eta z \in E_2^{2,1}(X)  \cong \W/2\{\overline{b}_0\eta z\}, \qquad 0 \ne \alpha \in \WW/2.
\]
Then point (4) of \Cref{prop:keylemobstruct} gives that
\[
y \eta =0 \in  \pi_{-1}(E^{h\mathbb{S}_2^1}\wedge X).
\]
However, 
\[
0 \ne  \wchi  \eta\iota_X \in H^3(\SS_2^1,E_2X) \cong H^3(\SS_2^1,E_2).
\]
See \cref{fig:HFPSSG2}.
Because $\wchi \eta \iota_X$ detects $y\eta = 0$, we must have (for degree reasons) that $d_3(\iota_X)$ is a non-zero 
multiple of $\wchi\eta \iota_X$. 

Since we have a map of spectral sequences
\begin{align}\label{eq:diagramS2G2}
\xymatrix{
H^s(\GG_2^1,E_tX) \ar@{=>}[r]\ar[d]  & \pi_{t-s}E^{h\GG_{2}^{1}} \wedge X\ar[d]\\
H^s(\SS_2^1,E_tX) \ar@{=>}[r] & \pi_{t-s}E^{h\SS_2^1} \wedge X.\\
}
\end{align}
and the map on $E_2$-terms is an injection onto the Galois invariants we exactly have
\[
d_3(\iota_X) = \wchi\eta\, \iota_X
\]
as needed. 

For the converse, if $d_2(z)=0$, then $X \in \kappa(\GG_2^1)$ by \cref{prop:keylemobstruct} so $\iota_X$
is a permanent cycle in the Adams-Novikov Spectral Sequence for $E^{h\mathbb{G}_2^1}\wedge X$  by 
\Cref{prop:detect-kappa-ANSS}.
\end{proof}

For the next result, we recall that the homomorphism $\phi_3^1$ was defined in \eqref{eq:phi31} ({\it{c.f.}} 
\Cref{defn:descent-filt}). It is given by the formula
\[
d_3(\iota_X) = \phi_3^1(X)\iota_X
\]
where $\iota_X \in E_0X$ is any $\GG_2$-invariant generator and $d_3$ is the differential in the spectral sequence
\begin{align}\label{eq:tmp24-1}
H^s(\GG_2^1, E_tX) \Longrightarrow \pi_{t-s} E^{h\GG_2^1}\wedge X.
\end{align}

\begin{thm}\label{thm:upper-bound-kg48-etc}
The composite
\[
\kappa(G_{48}) \to   \kappa_2 \xrightarrow{\phi_3^1}   H^{3}(\GG_2^1, E_2) \cong \Z/2\{\eta \wchi \} 
\]
 induces an injective homomorphism
\[
\kappa(G_{48})/\kappa(\GG_2^1) \to \Z/2 \ .
\]
\end{thm}

\begin{proof} Let $X \in \kappa(\GG_2^1)$. Then by \Cref{prop:detect-kappa-ANSS} the class
$\iota_X$ is a permanent cycle in the spectral sequence of \eqref{eq:tmp24-1}.
Therefore, $\phi_3^1(\iota_X)=0$ and so $\kappa(\GG_2^1)$ is in the kernel of $\phi_3^1$.

Now, suppose that $X \in \kappa(G_{48})$ is in the kernel of the composite. Then $d_3(\iota_X)=0$ in the
spectral sequence of \eqref{eq:tmp24-1}.  Then \Cref{prop:keypropd2meansd3} implies that $d_2(z)=0$ in the Topological Duality spectral sequence for $X$. By  \Cref{prop:keylemobstruct},
 $X \in \kappa(\mathbb{G}_2^1)$.
\end{proof}

 % !TEX root = pic-master.tex

\section{The $G_{48}$-orientable elements of the Picard group}\label{sec:second-existence}

In this section we compute $\kappa(G_{48})$, the group of $E^{hG_{48}}$-orientable elements in $\kappa_2$. The main result
is \Cref{prop:kappaG48}. We begin by finding and analyzing a non-trivial element in
$\kappa(G_{48})/\kappa(\GG_2^1)$; that is,
we will find a spectrum $Q$ which has an $E^{hG_{48}}$-orientation that cannot be refined to
an $E^{h\GG_2^1}$-orientation. We will use \Cref{thm:upper-bound-kg48-etc}; thus
we need to find $Q \in \kappa(G_{48})$ so that $d_3(\iota_Q) \ne 0$ in the spectral sequence
\[
E_2(\GG_2^1, Q)=H^s(\GG_2^1,E_tQ) \Longrightarrow \pi_{t-s}(E^{h\GG_2^1} \wedge Q).
\]
To do this we will use the $J$-construction as in \Cref{sec:j-construction} and some classical $C_2$-equivariant homotopy theory.  
  
Let $\sigma$ be the one dimensional real sign representation of $C_2$. We have an isomorphism
\[
\ZZ[\sigma]/(\sigma^2-1) \cong RO(C_2)
\]
describing the real representation ring of $C_2$. The augmentation $RO(C_2) \to \ZZ$ sends a 
representation to its virtual dimension; the augmentation ideal $I(C_2)$ is of rank $1$ over $\ZZ$ 
generated by $\sigma -1$. 

Now take $\chi \colon \GG_2 \to (\Z_2/4)^{\times} \cong C_2$ to be the surjective
homomorphism of \cref{defn:chi-defined}, whose kernel we called $\GG_2^0$. Given any virtual representation
$V \in RO(C_2)$, we get an action of $\GG_2$ on $S^V$ by restriction along $\chi$. 
Recall from \Cref{exam:rep-spheres} and \Cref{prop:basic-j-3} that we get a homomorphism
\[ J: RO(C_2) \to \Pic_2,\]
defined by the formula
\[
J(V) = (E \wedge S^V)^{h\GG_2}.
\]
This uses the diagonal action of $\GG_2$ on the smash product.

Applying \Cref{prop:basic-j-1}, we have that if $K \subseteq \GG_2$ is a closed subgroup, then 
\[ E^{hK} \wedge J(V) = E^{hK} \wedge (E \wedge S^V)^{h\GG_2} \simeq (E\wedge S^V)^{hK},\]
so that in particular, if $K$ is in the kernel of $\chi$, then we obtain an equivalence
\[ E^{hK} \wedge J(V) \simeq \Sigma^{\dim V} E^{hK}.\]

Let $V \in I(C_2)$ and let $j = j_V \in \pi_0S^V \cong \ZZ$ a chosen generator. Define $\iota = \iota_V \in E_0S^V$
to the the image of $j$ under the  Hurewicz map. We have an isomorphism of Morava modules
\[
\pi_\ast (E \wedge J(V)) \cong \pi_\ast (E \wedge S^V) \cong E_\ast \otimes_{E_0} E_0S^V
\]
where $\GG_2$ acts diagonally on $E_\ast \otimes_{E_0} E_0S^V$. In particular, 
if $V \in I(C_2)^2 = (2(\sigma-1))$, $\iota_V$ is a $\GG_2$-invariant generator and $J(V) \in \kappa_2$. 
Compare \Cref{rem:basic-j-1b}.

\begin{lem}\label{lem:G48-reps}
If $V \in I(C_2)$, then there is an equivalence of $E^{hG_{48}}$-modules 
\[
E^{hG_{48}} \simeq E^{hG_{48}} \wedge J(V).
\]
If $V \in I(C_2)^2$, then $J(V) \in \kappa(G_{48})$.  
\end{lem} 

\begin{proof} Note that $G_{24} \subseteq \SS_2$ is in the kernel of the determinant (\Cref{defn:det-defined-here}), 
since there are no non-trivial group homomorphisms $G_{24} \to \ZZ_2^\times$. (See, for example,
Remark 5.1.6 of \cite{BGH} for this last fact.) Then $G_{48}$ is in the kernel of $\chi$, and the first statement
follows. The second  statement follows from the first by \Cref{lem:old-is-new}, since $\iota_V$
gives $E_\ast J(V)$ a $\GG_2$-invariant generator by the above discussion.
\end{proof} 

Now we come to the star of this section.

\begin{defn}\label{defn:Qdefined} Let $Q \in \kappa(\GG_2^1)$ be the spectrum
\begin{equation*}\label{eq:Qdefined}
Q =J (2\sigma-2) = (E \wedge S^{2\sigma-2})^{h\GG_2} \ . 
\end{equation*}
\end{defn}

By \Cref{lem:G48-reps}, we know that $Q$ is an element of $\kappa(G_{48})$, and we will show that its class in the quotient $  \kappa(G_{48})/\kappa(\GG_2^1)$
is non-trivial. Using similar methods, we will also analyze $2Q \simeq J(4\sigma-4)$ and will find that it is zero in $\kappa_2$.

For any virtual $C_2$ representation $V$, there is a $C_2$-equivariant map
\[
S^V \longr E^{h\GG_2^0} \wedge S^V,
\]
hence a map
\[
(S^V)^{hC_2} \longr (E^{h\GG_2^0} \wedge S^V)^{hC_2} \simeq J(V).
\]
In the next subsection, we will use standard methods in equivariant stable homotopy theory to give a partial analysis of 
$(S^V)^{hC_2}$, which will help us analyze $ Q$ and apply \Cref{thm:upper-bound-kg48-etc}.

\subsection{Some $C_2$-homotopy theory for representation spheres.}\label{sec:stunted} In this section we write down some classical
homotopy theory for the homotopy fixed points $(S^V)^{hC_2}$, where $V$ is a virtual representation
of the cyclic group of order $2$.

There is a small and  geometric model for $EC_2$; namely, $S^\infty = \cup S^m$ with the antipodal action.
In a coincidence forced on us by the fact that $C_2$ is a very small group, $S^\infty$ is equivariantly homeomorphic
to the geometric realization of the standard simplicial model for $EC_2$ obtained from the bar construction. Thus for any $C_2$-spectrum $X$, its homotopy fixed points can be expressed as
\[
X^{hC_2} \simeq \Tot(F(C_2^{\bullet + 1},X)) \simeq F(\Sigma_+S^\infty,X).
\]
Indeed, for each $m$ there is an equivalence
$ \Tot_m(F(C_2^{\bullet + 1},X)) \simeq F_{C_2}(\Sigma^\infty_+S^m,X)$. We could use either of these descriptions to construct the homotopy fixed point spectral sequence
\[
H^s(C_2,\pi_tX) \Longrightarrow \pi_{t-s}X^{hC_2},
\]
as the Bousfield-Kan spectral sequence of the respective towers giving $X^{hC_2}$ as a homotopy limit
\[
X^{hC_2} \simeq  \holim  \Tot_m(F(C_2^{\bullet + 1},X)) \simeq \holim_m F_{C_2}(\Sigma^\infty_+S^m,X).\]

If $V$ is a virtual $C_2$-representation, then $S^V$ is dualizable with 
Spanier-Whitehead dual $S^{-V}$; thus we have
\[
(S^{V})^{hC_2} \simeq \holim_m F(\Sigma_+^\infty S^m \wedge_{C_2} S^{-V},S^0).
\]
From this we get that the basic objects of study are the spectra
\[
\Sigma_+^\infty S^m \wedge_{C_2} S^{n\sigma}, \quad n,m\in \Z, \ m\geq 0
\]
and their Spanier-Whitehead duals. As above, $\sigma$ is the sign representation. 

In the following we will confuse the (pointed) projective spaces $\RP^n$ with their suspension spectra. We will do this throughout,
letting context indicate whether we are working with the space or the spectrum.

Let $\xi$ be the tautological line bundle  over $\RP^{\infty}$ and let $T$ denote the Thom spectrum functor; we
arrange our conventions for this functor so that if $\gamma$ is a bundle of virtual dimension $k$, then the Thom class
is in $H^k(T(\gamma),\FF_2)$. By \cite[Theorem 1.8 of Chapter 16]{Husemoller}, we have that for all $n \in \ZZ$, the Thom spectrum of $n\xi$ is identified as
\begin{align}\label{eq:thom}
\Sigma^\infty_+ S^\infty \wedge_{C_2} S^{n\sigma} \simeq T(n\xi) \simeq \RP^{\infty}_{n}.
\end{align}
Here, if $n \geq 0$, then $\RP^\infty_n = \RP^\infty/\RP^{n-1}$ is the (suspension spectrum of) truncated projective space; if $n < 0$, this formula is the 
definition of $\RP^\infty_n$. Note that if $n=0$, $\RP^{\infty}_0 \simeq \Sigma^\infty_+\RP^\infty$, while for $n=1$ we have $\RP^{\infty}_1 \simeq \Sigma^\infty \RP^\infty $, which we write as $\RP^\infty$ according to the convention above.

If $\xi_m$ denotes the restriction of $\xi$ to $\RP^m$, then 
\begin{align}\label{eq:thom-m}
\Sigma^\infty_+ S^m \wedge_{C_2} S^{n\sigma}  \simeq T(n\xi_m) \simeq \RP^{m+n}_n,
\end{align}
where again this formula may be needed to define the truncated projective space. Note again that
$H_k(\RP^{m+n}_n,\FF_2) \ne 0$ only for $n \leq k \leq m+n$. Colloquially, we say that this spectrum
has bottom cell in dimension $n$ and top cell in dimension $m+n$.

The virtual bundle $\xi_m -1$ on $\RP^m$ has finite order; indeed, this order is a number $c_m$ related
to the Adams vector field number. The original source for this result is \cite{AdamsVectorFields}, but it can be found
conveniently in \cite[Chapter 16, Theorem 12.7]{Husemoller}, with an explicit formula in \cite[Chapter 16, Remark 11.1]{Husemoller}. This gives \emph{James Periodicity} for truncated projective spaces
\begin{equation}\label{eq:james-per}
\Sigma^{c_m}\RP^{m+n}_n \simeq \RP^{c_m+ m+n}_{c_m+n}.
\end{equation}
In particular, the following classical result will help calculate differentials in the $C_2$-homotopy fixed point spectral sequence of a suitable representation sphere.
\begin{lem}\label{lem:james-shift} For all $n$, there are equivalences of spectra 
\[
\Sigma^4 \RP_n^{n+2} \simeq \RP_{n+4}^{n+6}, \quad \quad \Sigma^4 \RP_n^{n+3} \simeq \RP_{n+4}^{n+7}, \quad \text{ and }\quad \Sigma^8 \RP_n^{n+7} \simeq \RP_{n+8}^{n+15}.
\]
In addition, we have an equivalence of spectra
\[
\RP^3 \simeq \RP^2 \vee S^3.
\]
\end{lem}

\begin{proof}  By \cite[Chapter 16, Theorem 12.7]{Husemoller} the virtual bundles $\xi_2-1$ on $\RP^2$, as well as $\xi_3-1$ on $\RP^3$ have order $c_2=4 = c_3$, while $\xi_7 -1$ on $\RP^7$ has order $c_m = 8$. So, the first three statements follow from James Periodicity \eqref{eq:james-per}. 
For the last statement, note
that $\RP^3$ is the manifold underlying the Lie group $SO(3)$; 
hence $\RP^3$ is stably parallelizable
and stably the top cell splits off. That $\RP^3$ is stably parallelizable could also be proved using the fact that
the stable tangent bundle of $\RP^3$ is isomorphic to $4\xi_3$, hence trivial.
\end{proof}

We will need the following formula for the Spanier-Whitehead dual of $\RP^k_n$
\begin{equation}\label{eq:SW-duals-projs}
D\RP^k_n = F(\RP^k_n,S^0) \simeq \Sigma \RP^{-n-1}_{-k-1},
\end{equation}
which can found as \cite[Theorem 6.1]{Atiyah}, or can easily be proved given what we've said so far.
Indeed, if $n=0$ this follows from Atiyah Duality and the fact that the stable normal bundle of $\RP^k$ is
$-(k+1)\xi+1$. For general $n$ and $k$ we use James Periodicity, as in \eqref{eq:james-per}. 

The above material may be assembled to give a formula for the homotopy fixed points. We will only need it for
$n=0$, $2$, and $4$, but the result is easy to state for all $n$. 

\begin{prop}\label{prop:basic-homotopy-fixed}  Let $n \in \ZZ$. Then the $\Tot$-tower of the cosimplicial spectrum
$F_{C_2}(C_2^{\bullet + 1},S^{n\sigma})$ is equivalent to the tower of fibrations
\[
\xymatrix{
\cdots \ar[r] & \Sigma\RP^{n-1}_{n-m-1} \ar[r] & \Sigma\RP^{n-1}_{n-m} \ar[r] & \cdots \ar[r] &\Sigma\RP^{n-1}_{n-2} \ar[r] &
\Sigma\RP^{n-1}_{n-1}\\
&S^{n-m} \ar[u]&S^{n-m+1}\ar[u]&&S^{n-1}\ar[u]&S^{n} \ar[u]^\simeq
}
\]
where the maps from the successive fibers are given by the inclusion of the bottom cell. We have
\[
\Tot_m F_{C_2}(C_2^{\bullet + 1},S^{n\sigma}) \simeq \Sigma \RP_{n-m-1}^{n-1}
\]
and there is an equivalence
\[
(S^{n\sigma})^{hC_2} \simeq \holim_i \Sigma \RP^{n-1}_{-i}\ . 
\]
\end{prop}

\begin{proof} 
By the discussion at the beginning of this subsection, we have an equivalence \[ \Tot_m(F(C_2^{\bullet + 1},S^{n\sigma})) \simeq F_{C_2}(\Sigma^\infty_+S^m,S^{n\sigma}).\] We can analyze the right-hand side using \eqref{eq:thom-m} to obtain
\begin{align*}
F_{C_2}(\Sigma_+^\infty S^m,S^{n\sigma}) &\simeq F(\Sigma_+^\infty S^m \wedge_{C_2} S^{-n\sigma},S^0) \simeq F(\RP^{m-n}_{-n},S^0),
\end{align*}
while \eqref{eq:SW-duals-projs} identifies this with $ \Sigma \RP_{n-m-1}^{n-1} $ as needed.
\end{proof}

\begin{rem}\label{rem:whatif-add-triv} So far these formulas have only been for multiples of the sign representation,
but we have $S^{n\sigma+k} = S^{n\sigma} \wedge S^k$ with the trivial action on $S^k$. This forces a shift 
into the $\Tot$-tower; for example, we have
\[
(S^{n\sigma+k})^{hC_2} \simeq \holim_i \Sigma^{1+k} \RP^{n-1}_{-i}\ .
\]
\end{rem}

{\textbf{Now let $n$ be even and we implicitly $2$-complete all spectra.}}

The requirement that $n$ be even implies that the action of $C_2$ on $H_n(S^{n\sigma},\Z_2)$ is trivial, so the
action on $\pi_\ast S^{n\sigma}\cong \pi_\ast S^n$ is also trivial. 
Choose a generator
\[
j_n \in H^0(C_2,\pi_nS^{n\sigma}) \cong \ZZ_2.
\]
We are interested in the fate of this class in the homotopy fixed point spectral sequence
\begin{equation}\label{eq:C2-hfpss}
E_2^{s,t}=H^s(C_2,\pi_tS^{n\sigma}) \Longrightarrow \pi_{t-s}(S^{n\sigma})^{hC_2}.
\end{equation}
This is the spectral sequence for the tower of fibrations in \Cref{prop:basic-homotopy-fixed}. 
Since the class $j_n$ is represented by the generator of
\[
\pi_n\Sigma\RP^{n-1}_{n-1} \cong \pi_n S^n \cong \ZZ_2,
\]
differentials on it come down to whether or not the projection map 
\[
\Sigma\RP^{n-1}_{n-m-1} \to S^n
\]
to the top cell has a splitting.

\newcommand{\wchihfs}{g}
\newcommand{\chihfs}{h}
\newcommand{\ahfs}{\alpha}
 
We now give a calculation of the $E_2$-term  of \eqref{eq:C2-hfpss} in a range. Since we are assuming $n$ 
is even, the $E_2$-term is a free module of
rank one on $j_n \in H^0(C_2,\pi_nS^{n\sigma})$ over the ring
$
H^\ast(C_2,\pi_\ast S^0) $. 
Let $\eta \in \pi_1S^0$ and $\nu \in \pi_3S^0$ be the standard generators and let $\chihfs \in H^1(C_2,\ZZ/2)$ 
and $\wchihfs \in H^2(C_2,\ZZ_2)$ be the generators of the group cohomology rings. Then $\wchihfs$ is the Bockstein on $\chihfs$, and reduces
to $\chihfs^2 \in H^2(C_2,\ZZ/2)$. Finally, let $\ahfs = \chihfs\eta \in H^1(C_2,\pi_1S^0)$, and $\beta= h\nu^2 \in H^1(C_2,\pi_3 S^0)$.
The following is now a standard calculation. Compare \Cref{fig:HC2S0}.

\begin{figure}[h]
\center
\includegraphics[width=\textwidth]{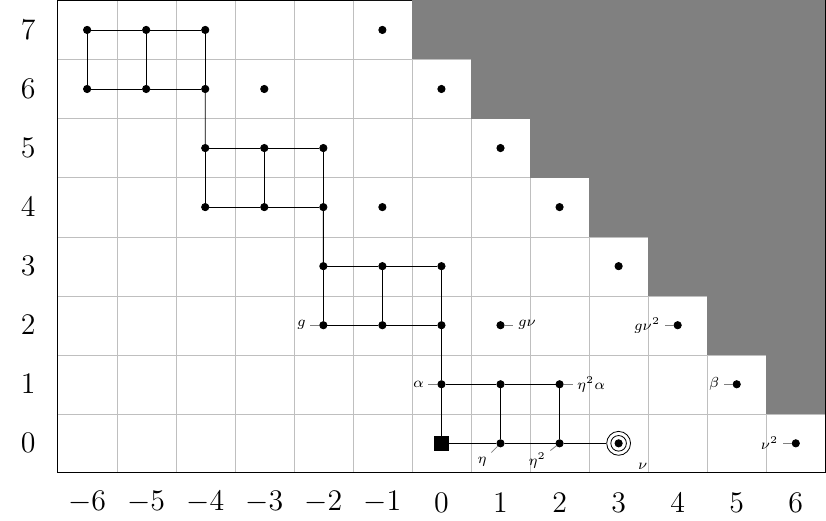}
\captionsetup{width=\textwidth}
\caption{The cohomology ring $H^s(C_2,\pi_t S^0)$ in a range. The vertical axis is $s$ and the horizontal axis is $t-s$.}
\label{fig:HC2S0}
\end{figure}

\begin{prop}\label{lem:low-comps-fpsp} There is a map of bigraded rings
\begin{align*}
\varphi \colon \ZZ_2[\wchihfs, \eta, \ahfs,\nu, \beta]
\longr\ H^\ast(C_2,\pi_\ast S^0),
\end{align*}
where the  $(s,t)$-bidegrees of the generators are
\[
|\wchihfs|=(2,0) \quad |\eta| = (0,1) \quad |\ahfs| =(1,1)\quad |\nu|=(0,3) \quad |\beta| = (1,6).
\] 
The ideal
\[I=(2\eta,  \ 8\nu, \ 2\nu^2,  \ \eta \nu, \ \eta^3-4\nu, \ 2\ahfs, \ 2\wchihfs, \ \eta^2 \wchihfs-\ahfs^2, \ \nu \ahfs,  \ 2\beta) \]
 is in the kernel of $\varphi$. 
The map
\[\ZZ_2[\wchihfs, \eta, \ahfs,\nu, \beta]/I \longr\ H^\ast(C_2,\pi_\ast S^0) \]
induced by $\varphi$
is an isomorphism in bidegrees $(s,t)$ with $t \leq 6$. 
 \end{prop}

We now come to our first calculations with \eqref{eq:C2-hfpss}, which will be key to \Cref{thm:Qfirst} and \Cref{thm:2Q}.

The following proposition will be useful because the homotopy fixed point spectral sequences for $C_2$ acting on $S^V$ for any representation $V$ are modules over the spectral sequence for the $C_2$ acting trivially on $S^0$.

\begin{prop}\label{prop:zeroSphere}
Let $n \equiv 0 \mod 8$. Then in the homotopy fixed point spectral sequence
\[
E_2^{s,t}=H^s(C_2,\pi_tS^{n\sigma}) \Longrightarrow \pi_{t-s}(S^{n\sigma})^{hC_2}
\]
we have $d_r(g^2 j_n)=0$ for $r\leq 3$.
\end{prop}

\begin{proof} Write $n=8k$. We use the tower of \Cref{prop:basic-homotopy-fixed}. The class $g^2j_n$
is the residue class of the bottom cell of $\Sigma\RP^{8k-1}_{8k-5}$ and we are asking if it is in the 
image of the quotient map
\[
\Sigma\RP^{8k-1}_{8k-8} \longr \Sigma\RP^{8k-1}_{8k-5}.
\]
From \Cref{lem:james-shift} we have that this map is a (de-)suspension of the 
quotient map
\[
\RP^{7}_{0} \longr \Sigma\RP^{7}_{3}.
\]
If $S^3 \to \RP^3$ splits off the top the cell then the composition
\[
S^3 \longr \RP^3 \longr \RP^{7}_{0} \longr \Sigma\RP^{7}_{3}.
\]
is non-zero in homology, as needed. 
\end{proof}

\begin{prop}\label{prop:first-truncate} Let $n \equiv 2$ modulo $4$. Then in the homotopy
fixed point spectral sequence
\[
E_2^{s,t}=H^s(C_2,\pi_tS^{n\sigma}) \Longrightarrow \pi_{t-s}(S^{n\sigma})^{hC_2}
\]
we have differentials
\[
d_2(j_n) =\eta \wchihfs j_n \quad \text{ and } \quad d_2(\wchihfs j_n) = 0.
\]
\end{prop}

\begin{proof} Write $n=4k+2$. We can analyze these differentials as obstructions to lifting along the tower in \Cref{prop:basic-homotopy-fixed}. Note that we have a splitting $\RP^{4k+1}_{4k} \simeq S^{4k} \vee S^{4k+1}$, which gives a lift of $j_n$ as a map $S^{4k+2 } \to \Sigma \RP^{4k+1}_{4k}$, and $d_2(j_n)$ is the obstruction to lifting this further to $\Sigma\RP^{4k+1}_{4k-1}$. Now note that 
\[
0 \ne \mathrm{Sq}^2: H^{4k-1}( \RP^{4k+1}_{4k-1},\F_2) \longr  H^{4k+1} (\RP^{4k+1}_{4k-1},\F_2).
\]
This implies that
\[
\pi_{4k+2}\Sigma\RP^{4k+1}_{4k-1} \longr \pi_{4k+2}\Sigma\RP^{4k+1}_{4k+1} 
\]
is not surjective, since the non-trivial $\mathrm{Sq}^2$ implies that the top cell of $\RP^{4k+1}_{4k-1}$ does not split off. This gives $d_2(j_n) \ne 0$ and, by \Cref{lem:low-comps-fpsp} there is only one non-zero class in
$H^2(C_2,\pi_{n+1}S^{n\sigma})$. 

For the second statement, let $x: S^{4k} \to \Sigma \RP^{4k+1}_{4k-1}$ be the inclusion of the bottom
cell. Then we are asking if $x$ factors through $\Sigma\RP_{4k-3}^{4k+1}$. To see that it does, note that the 
composition
\[
\RP^{4k-1}_{4k-3} \to \RP^{4k+1}_{4k-3} \to  \RP^{4k+1}_{4k-1}
\]
is non-zero in $H_{4k-1}(-,\FF_2)$. Now use that
\[
\RP^{4k-1}_{4k-3} \simeq \Sigma^{4k-4}\RP^3 \simeq S^{4k-1} \vee \Sigma^{4k-4}\RP^2
\]
to see that $x$ can be factored as the resulting splitting $S^{4k} \to \Sigma \RP^{4k-1}_{4k-3} $ followed by the inclusion $\Sigma \RP^{4k-1}_{4k-3}  \to \Sigma \RP^{4k+1}_{4k-3} $.
\end{proof}

By a similar method, we get the following result, crucial to \Cref{thm:2Q}. 

\begin{prop}\label{prop:second-truncate} Let $n \equiv 4$ modulo $8$. Then in the homotopy
fixed point spectral sequence
\[
H^s(C_2,\pi_tS^{n\sigma}) \Longrightarrow \pi_{t-s}(S^{n\sigma})^{hC_2}
\]
we have $d_r(j_n) = 0$ for $r \leq 3$ and
\[
0 \ne d_4(j_n) =\nu \wchihfs^2 j_n.
\]
Moreover, $d_r(g^2 j_n) =0$ for $r\leq 3$.
\end{prop}

\begin{proof} Write $n= 8k+4$. Proving that $d_r(j_n)=0$ for $r\leq 3$ amounts to lifting the bottom cell $S^{8k+4} \to \Sigma\RP_{8k+3}^{8k+3}$ to $\Sigma \RP_{8k}^{8k+3}$. But by \Cref{lem:james-shift}, we have
\[
\RP^{8k+3}_{8k} \simeq \Sigma^{8k}\RP^3_0 \simeq S^{8k} \vee \Sigma^{8k}\RP^2 \vee S^{8k+3},
\]
so there are no obstruction to the needed lift.
For the $d_4$ calculation, notice that 
\[
0 \ne \mathrm{Sq}^4: H^{8k-1} (\RP^{8k+3}_{8k-1},\F_2) \longr  H^{8k+3}( \RP^{8k+3}_{8k-1},\F_2).
\]
It follows
\[
\pi_{8k+4}\Sigma\RP^{8k+3}_{8k} \longr \pi_{8k+4}\Sigma\RP^{8k+3}_{8k+3} \cong \ZZ_2
\]
is surjective but
\[
\pi_{8k+4}\Sigma\RP^{8k+4}_{8k-1} \longr \pi_{8k+4}\Sigma\RP^{8k+3}_{8k+3}
\]
is not, implying that $d_4(j_n) \ne 0$. By \Cref{lem:low-comps-fpsp}  
\[
\nu\wchihfs^2 j_n \in H^4(C_2,\pi_{n+3}S^{n\sigma})
\]
is the only non-zero class. Since $d_4(j_n) \ne 0$, the class $\wchihfs^2\nu j_n$ cannot be the image of 
$d_2$ or $d_3$ and the result follows. 

To see that $d_r( g^2 j_n) = 0$ for $r\leq 3$, note from \Cref{prop:zeroSphere} that $g^2$ is a $d_3$-cycle in the homotopy fixed point spectral sequence for the trivial $C_2$-action on $S^0$. Using the module structure over that spectral sequence, we get the claim.
\end{proof}

\subsection{A non-trivial element in $\kappa_2(G_{48})/\kappa(\GG_2^1)$} \label{sec:theQ}
After the above interlude on stunted projective spaces and the calculation of differentials in some classical $C_2$-homotopy fixed points, we are ready to return to the study of our invertible spectrum $Q$ from \Cref{defn:Qdefined}.

We set the stage for the proof of \Cref{thm:Qfirst}, to connect it to the material in \Cref{sec:stunted}. We will use the same set-up for \Cref{thm:2Q} as well.

Let $V \in I(C_2)$. In \Cref{rem:compare-ss}
we used the map $S^V \to E \wedge S^V$ to construct a map of augmented cosimplicial spectra
\begin{equation}\label{eq:ss-to-be-trunc}
\xymatrix{
S^V \ar[r] \ar[d] & F(C_2^{\bullet +1},S^V)^{C_2}\ar[d]\\
E \wedge S^V \ar[r] & F_c(\GG_2^{\bullet +1},E \wedge S^V)^{\GG_2}
}
\end{equation}
Suppose $V = n\sigma - n$. Let $Y_m = \Tot_mF_c(\GG_2^{\bullet +1},E \wedge S^V)^{\GG_2}$.
Then, by \Cref{prop:basic-homotopy-fixed} and \Cref{rem:whatif-add-triv} we have a commutative diagram of towers
\[
\xymatrix{
\Sigma^{1-n}\RP^{n-1}_{n-m-1} \ar[r] \ar[d] & Y_m \ar[d]\\
\Sigma^{1-n}\RP^{n-1}_{n-m} \ar[r] & Y_{m-1}\ ,
}
\]
which gives the diagram of spectral sequences
\begin{equation}\label{eq:mapOfSS}
\xymatrix{
E_2^{s,t}(C_2,S^V) = H^s(C_2,\pi_tS^V) \ar@{=>}[r]\ar@<7ex>[d] & \pi_{t-s}(S^V)^{hC_2} \ar[d]\\
E_2^{s,t}(\GG_2,E\wedge S^V)  = H^s(\GG_2,E_tS^V) \ar@{=>}[r] & \pi_{t-s}J(V).
}
\end{equation}
Now suppose $V \in I(C_2)^2$ and fix a choice of $C_2$-invariant generator $j_V \in \pi_0S^V$, which defines
a $\GG_2$-invariant generator $\iota_V \in E_0S^V$. Then we
have a commutative diagram in group cohomology
\[
\xymatrix
{H^*(C_2, \pi_*S^0) \ar[rr] \ar[d]^-{\cong}_-{(j_V)_*}  & & H^*(\GG_2, E_*) \ar[d]^-{\cong}_{(\iota_V)_*}   \\
H^*(C_2, \pi_*S^{V}) \ar[rr] & & H^*(\GG_2, E_*S^{V}).
}
\]

Note that the bottom spectral sequence is isomorphic to the Adams--Novikov spectral sequence for $J(V)$, by \Cref{prop:basic-j-00}. We will be interested in the faith of $\iota_V$. 

To be specific, recall the map
\[
\phi_3^1: \kappa(G_{48}) \to H^3(\GG_2^1,E_2) \cong \ZZ/2
\]
from \eqref{eq:phi31} ({\it{c.f.}} \Cref{defn:descent-filt}). For the element $Q = J(2\sigma-2) $ from \Cref{defn:Qdefined} of $\kappa_{G_{48}}$, we can understand $\phi_3^1(Q)$ by first analyzing the faith of $\iota_{2\sigma-2}$ in the homotopy fixed point spectral sequence $E_r^{*,*}$ through \eqref{eq:mapOfSS}. Then we can restrict to the Adams--Novikov spectral sequence for $E^{h\GG_2^1}\wedge Q$, which is the $\GG_2^1$ homotopy fixed point spectral sequence for $E\wedge S^{2\sigma -2}$ and study the outcome.

In the next subsection, a similar argument will determine that $2 Q $, equivalent to $ J(4\sigma - 4)$ by \Cref{prop:basic-j-2}, is a trivial element in the Picard group. In that case, we will be checking that $\iota_{4\sigma-4}$ cannot support differentials.

In \Cref{thm:upper-bound-kg48-etc}, we proved that $\phi_3^1$ gives an injection
\[
\kappa(G_{48})/\kappa(\GG_2^1) \longr   \ZZ/2,
\]
and now we are ready to prove surjectivity.

\begin{thm}\label{thm:Qfirst} We have
\[
0 \ne \phi_3^1(Q) \in H^3(\GG_2^1,E_2) \cong \ZZ/2
\]
and hence a short exact sequence
\[
\xymatrix{
0 \ar[r] & \kappa(\GG_2^1) \ar[r] & \kappa(G_{48}) \ar[r]^-{\phi_3^1} & \ZZ/2 \ar[r] & 0
}
\]    
\end{thm}

\begin{proof}[Proof of \Cref{thm:Qfirst}.] Let $V = 2\sigma -2$. 
We will show that in the  spectral sequence $E_r^{s,t}(\GG_2,E\wedge S^V)$, 
there is a non-trivial differential
\[
d_3(\iota_V ) =\eta \wchi\iota_V. 
\]
Because the restriction of $\eta \wchi\iota_V$ is non-zero in $H^*(\GG_2^1, E_*S^V)$, this forces the same differential in the homotopy fixed point spectral sequence $E_r^{s,t}(\GG_2^1,E\wedge S^V)$.  The result then follows from \Cref{thm:upper-bound-kg48-etc}. 

From here on, we write $j = j_V$, $\iota = \iota_V$, we abbreviate $E_r^{s,t}(S^V) = E_r^{s,t}(C_2,S^V)$ and  $E_r^{s,t}(E\wedge S^V) = E_r^{s,t}(\GG_2,E\wedge S^V) $, and we refer to \Cref{lem:low-comps-fpsp} for the definitions of elements in $H^\ast(C_2,\pi_\ast S^{V})$. 

We apply the truncation construction of \Cref{subsec:trun-ss} to the map of towers \eqref{eq:ss-to-be-trunc}, and let
\[
f_\ast \colon  E_{*, \leq 3}^{*,*}(S^{V})  \to E_{*, \leq 3}^{*,*}(E \wedge S^{V})
\]
be the resulting map of truncated spectral sequences. We truncate up to $3$ since we are interested in $d_3(\iota)$. Then we have 
a commutative diagram
\begin{equation}\label{eq:trun-diag-e2}
\xymatrix{
H^s(C_2,\pi_tS^{V}) \ar[r] \ar[d] &  H^s(\GG_2,E_tS^{V}) \ar[d]\\
E_{2, \leq 3}^{s,t}(S^{V})  \ar[r]^-{f_\ast} & E_{2, \leq 3}^{s,t}(E \wedge S^{V}),
}
\end{equation}
where the vertical maps are injections for $s \leq 3$ and isomorphisms for $s < 3$. See \Cref{lem:truncated-e2}. 
The truncated spectral sequence $E_{*, \leq 3}^{*,*}(S^{V})$ converges to $\pi_\ast \Sigma^{-1}\RP^1_{-2}$
by \Cref{prop:basic-homotopy-fixed} and \Cref{rem:whatif-add-triv}. 

From \Cref{prop:first-truncate} we get that $\wchihfs j$ is a permanent cycle in the truncated spectral sequence
\[
E_{2, \leq 3}^{*,*}(S^{V}) \Longrightarrow \pi_{t-s}\Sigma^{-1}\RP^1_{-2}
\]
detecting a non-zero element of $\pi_{-2} \Sigma^{-1}\RP^1_{-2}$ which we will denote by $[\wchihfs j]$. 
Furthermore, also by \Cref{prop:first-truncate} we have the differential
\[
d_2(j) =\eta \wchihfs  j
\]
and so we conclude that $\eta [\wchihfs j]=0$ in $\pi_{-1} \Sigma^{-1} \RP_{-2}^1$.

We now turn to some corresponding elements in the truncated spectral sequence
\[
E_{2, \leq 3}^{*,*}(E \wedge S^{V}) \Longrightarrow \pi_{t-s}\Tot_3(Q).
\]
Referring to \eqref{eq:trun-diag-e2} we see that on $E_2$-terms, $f_*(\wchihfs j)$ is the class
\[
\wchi\iota \in H^2(\GG_2,E_0S^{V}),\
\]
which is immediate from the definition of the class $\wchi \in H^2(\GG_2, E_0)\cong H^2(\GG_2,E_0S^V) $ as the Bockstein on $\chi$; see \Cref{defn:chi-defined} or \cite[(1.7)]{BBGHPScoh}.
Now, the class $\wchi\iota \in E_{r,\leq 3}^{*,*}(E\wedge S^{V})$ is not hit by a differential for degree reasons; see \cite[Table 3]{BBGHPScoh}. Therefore, it detects a non-trivial class in $[\wchi\iota] \in \pi_{-2}\Tot_3(Q)$.

Since $\eta [\wchihfs j]=0$ we must have $\eta [\wchi \iota]= 0$.  But $ \eta  [\wchi \iota]$ is detected by 
\[
\eta\wchi\iota  \in H^3(\GG_2,E_2S^{V}) \cong H^3(\GG_2, E_2)
\]
again by \cite[\S7]{BBGHPScoh}, or \cref{fig:HFPSSG2}.
So, $\eta\wchi\iota$ must be killed by a differential
and $d_3(\iota ) =\eta \wchi\iota $  is the only possibility. 
\end{proof}

\subsection{The group $\kappa(G_{48})$}	
In this section we apply the same type of analysis as in the previous section in order to prove that $Q \in \kappa(G_{48})$ has order 2. This will imply a splitting 
\[ \kappa(G_{48}) \cong \kappa(\GG_2^1) \oplus \ZZ/2. \]
While the idea of the argument is essentially the same as for \Cref{thm:Qfirst}, it is somewhat harder to accomplish as there are longer differentials to keep track of.

We start with the diagram of augmented cosimplicial spectra from \eqref{eq:ss-to-be-trunc}, with $V = 4\sigma -4$. Again we abbreviate
\begin{align*}
 j &= j_{4\sigma-4} \in H^0(C_2,\pi_0S^{4\sigma - 4})\\
 \iota &= \iota_{4\sigma-4} \in H^0(\GG_2,E_0S^{4\sigma-4}).
\end{align*} 

The goal is to show that $\iota$ is a $d_7$-cycle in the $\GG_2$ spectral sequence. Recalling \Cref{defn:descent-filt}, this would imply that $2Q$ is an element of $\kappa_{2,8}$, so altogether in $\kappa_8(\GG_2^1)$, which is trivial by \Cref{thm:comb-2-filts-fin}.

\begin{figure}
\center
\includegraphics[width=\textwidth]{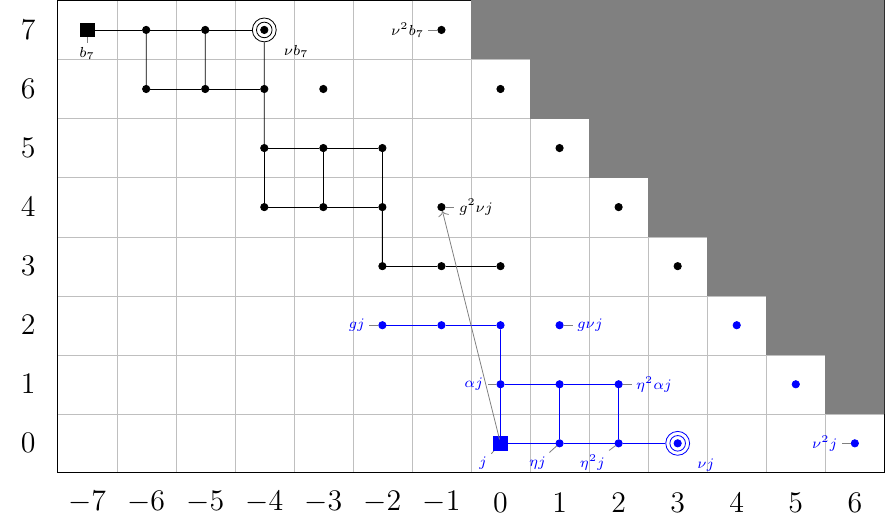}
\captionsetup{width=\textwidth}
\caption{The spectral sequence for the cofiber sequence of towers $\bX^7_3 \longr \bX_{\leq 7} \longr \bX_{\leq 2}$. Classes in {\color{blue} blue} are in the spectral sequence for $\bX_{\leq 2}$, classes in black are in that of $\bX^7_3$. The classes combined give the $E_2$-page of the spectral sequence of $\bX_{\leq 7}$. A $\blacksquare$ is a copy of $\Z_{2}$, a $\bullet$ is a $\Z/2$. The circled $\bullet$ represent $\Z/8$. We have only drawn one differential in the spectral sequence of $\bX_{\leq 7}$.}
\label{fig:HC2S0hfpss4sig4_ALL}
\end{figure}

The first step is an examination of some useful truncations of the $C_2$-fixed point spectral sequence of $S^{4\sigma -4}$, whose outcome is \Cref{lem:chasing-truncs}. Since we are ultimately interested in $d_r(\iota)$ for $r \leq 7$, we will study the 7-truncation. By \Cref{prop:second-truncate} we know that in the homotopy fixed point spectral sequence
\[
H^s(C_2,\pi_t S^{4\sigma-4}) \Longrightarrow \pi_{t-s}(S^{4\sigma-4})^{hC_2}
\]
we have $d_4(j) = \nu g^2 j$; we will study a relative truncation in order to analyze the implications of this differential.

The Tot tower $\bX$ for $(S^{4\sigma-4})^{hC_2}$ is a $(-4)$-desuspension of the tower given in
\Cref{prop:basic-homotopy-fixed}. The $7$-truncation $\bX_{\leq 7}$ of this tower is given by the tower
\[
\xymatrix@C=20pt{
\Sigma^{-3}\RP^{3}_{-4}\ar[r] & \Sigma^{-3}\RP^{3}_{-3} \ar[r] & \Sigma^{-3}\RP^{3}_{-2} \ar[r] & \Sigma^{-3}\RP^{3}_{-1} \ar[r] & \cdots \ar[r] &
 \Sigma^{-3}\RP^{3}_{3}\\
S^{-7}\ar[u] & S^{-6} \ar[u]&S^{-5}\ar[u]&S^{-4}\ar[u]&&S^{0} \ar[u]^\simeq
}
\]
and the relative truncation $\bX^7_3$ is given by
\begin{equation}\label{eq:X36}
\xymatrix{
\Sigma^{-3}\RP^{0}_{-4} \ar[r] & \Sigma^{-3}\RP^{0}_{-3} \ar[r] & \Sigma^{-3}\RP^{0}_{-2} \ar[r] & \Sigma^{-3}\RP^{0}_{-1} \ar[r] & \Sigma^{-3}\RP^{0}_{0}\\
S^{-7} \ar[u]& S^{-6} \ar[u]&S^{-5}\ar[u]&S^{-4}\ar[u]&S^{-3} .\ar[u]^\simeq
}
\end{equation}
We have a fiber sequence of very short towers
\[
\bX^7_3 \longr \bX_{\leq 7} \longr \bX_{\leq 2}
\]
which, after taking inverse limits (which amounts to taking the top space in these finite towers) gives the
standard cofiber sequence
\[
 \Sigma^{-3}\RP^0_{-4} \to \Sigma^{-3}\RP^3_{-4} \to \Sigma^{-3}\RP^3_{1} \to  \Sigma^{-2}\RP^0_{-4}.
\]
For example, $\Tot(\bX^7_3) \simeq \Sigma^{-3}\RP^0_{-4}$. Note $\Sigma^{-3}\RP^3_{1} = \Sigma^{-3}\RP^3 \simeq S^0 \vee \Sigma^{-3}\RP^2$. 

From \Cref{lem:truncated-e2}, for $N = 2$ we get an identification
\[
E_2^{s,t}(\bX_{\leq 2}) =
\begin{cases} 
H^s(C_2,\pi_tS^0),& s \leq 2;\\
0,&s > 2,
\end{cases}
\]
while for $N=7$, we get 
\[ 
E_2^{s,t}(\bX_{\leq 7}) =
\begin{cases} 
H^s(C_2,\pi_tS^0),& s \leq 6;\\
\pi_tS^0, & s = 7;\\
0,&s > 7.
\end{cases}
\]
In addition, by \Cref{lem:truncated-e2-bis}, we have
\[
E_2^{s,t}(\bX_3^7) =
\begin{cases} 
H^s(C_2,\pi_tS^0),& 3 \leq s \leq 6;\\
\pi_tS^0, & s = 7;\\
0,&\mathrm{otherwise}.
\end{cases}
\]
\Cref{fig:HC2S0hfpss4sig4_ALL} gives a concise summary of the $E_2$-terms of the respective spectral sequences: the part in blue depicts the spectral sequence for $\bX_{\leq 2}$, the part in black depicts the one for $\bX_3^7$, and everything together gives the spectral sequence for $\bX_{\leq 7}$. Differentials are not depicted in this figure, other than the $d_4$-differential on $j$, amounting to the image of $j$ under the boundary map.

Using the nomenclature of \Cref{lem:low-comps-fpsp} we now have the following lemma, illustrated in \Cref{fig:HC2S0hfpss4sig4_ALL}.

\begin{lem}\label{lem:chasing-truncs}\label{lem:boundaryuptofiltration} In the spectral sequence
\[
E_2^{s,t}(\bX_{\leq 2}) \Longrightarrow \pi_{t-s}\Sigma^{-3}\RP^3
\]
the class $j$ is a permanent cycle detecting a homotopy class $[j_2]\in \pi_0\Sigma^{-3}\RP^3$, which
is non-zero in homology.

The class $g^2 j$ is a non-trivial permanent cycle in the spectral sequence 
\[
E_2^{s,t}(\bX_{3}^7) \Longrightarrow \pi_{t-s}\Sigma^{-3}\RP^{0}_{-4}.
\]

The image $\delta_*[j_2]$ of $[j_2]$ in $\pi_{-1} \Sigma^{-3}\RP^0_{-4}$ is detected by the
class $ [\nu g^2 j]$ in the spectral sequence
\[
E_2^{s,t}(\bX_3^7) \Longrightarrow \pi_{t-s}\Sigma^{-3}\RP^{0}_{-4}.
\]
More precisely, $\delta_*[j_2] = \nu ( [g^2 j] + a[\nu b_7])$, where $a \in \ZZ/2$, and $b_7 \in E_2^{7,0}(\bX_3^7)$. 
\end{lem}

\begin{proof} 
We apply \Cref{lem:geo-bound} to the tower $\bX$ with $K=3$, $M=7$, $r=4$, $x = j$, and $y = \nu g^2 j$, as we know 
from \Cref{prop:second-truncate} that $d_4(j)=\nu g^2 j$. This gives us that $\delta_*[j_2]$ is indeed detected by $[\nu 
g^2 j]$. Since there is nothing in filtrations 5 and 6 contributing to $\pi_{-1} \Sigma^{-3}\RP_{-4}^0$, we conclude that
$\delta_*[j_2] = [ \nu  \wchihfs^2 j] + [z_7]$,
for some $z_7 \in E_2^{7,6}(\bX_3^7)$. 

First, we note that the class $[\nu g^2 j]\in \pi_{-1}\Sigma^{-3}\RP_{-4}^0$ is the $\nu$-multiple of $[g^2 j]$, which follows 
from the fact that $[g^2 j] $ is a non-trivial permanent cycle in the spectral sequence for $\bX_3^7$, again by 
combination of \Cref{prop:second-truncate} and \Cref{lem:geo-bound}. 

Second, $z_7 = a \nu^2 b_7$, where $b_7$ represents the bottom cell of $\Sigma^{-3} \RP_{-4}^0$, and $a \in \ZZ/2$. 
Note that $\nu b_7$ cannot be hit by a differential: while there are potentially non-trivial $d_2$ and $d_3$ with targets in 
bidegree $(s,t) = (7,3)$, they both have sources in groups of order at most 2, thus cannot conspire to kill the generator
$\nu b_7$ of $E_2^{7,3}(\bX_3^7) \cong \ZZ/8$. Consequently, $[z_7] = a \nu [\nu b_7]$, and all together,
$\delta_*[j_2] $ is a multiple of $\nu$.
\end{proof}

\begin{thm} \label{thm:2Q} The spectrum  $2Q = Q\wedge Q$ is trivial in $\kappa_2$.
\end{thm}

\begin{proof} 
Write
\[
A^\bullet = \{(E^{\bullet +1} \wedge S^{4\sigma-4})^{h\GG_2}\}
\]
for the cosimplicial spectrum giving the homotopy fixed point spectral sequence. 
Write $\bY_\bullet $ for the associated tower. As discussed in the beginning of this subsection, it suffices to show that $d_r(\iota) = 0$ for $r\leq 7$ in the associated spectral sequence, i.e. the $\GG_2$ homotopy fixed point spectral sequence.
For that, it suffices to show that $\iota$ is a permanent cycle in the truncated spectral sequence for $\bY_{\leq 7}$, and that is what we will prove.

Consider the truncated towers
\[
\bY_3^7 \to \bY_{\leq 7} \to \bY_{\leq 2}.
\]

Then \eqref{eq:ss-to-be-trunc} gives us a diagram
\begin{equation}\label{eq:diagr-trunc-xy}
\xymatrix{
\Sigma^{-3}\RP^0_{-4} \ar[r] \ar[d]_f & \Sigma^{-3}\RP^3_{-4} \ar[r] \ar[d]_f & \Sigma^{-3}\RP^3\ar[d]^f\\
Y_3^7 \ar[r] & Y_7 \ar[r] & Y_2
}
\end{equation}
where $Y_3^7 = \holim \bY_3^7$ and so on. 
We have a diagram of spectral sequences
\[
\xymatrix{
E_2^{s,t}(\bX_{\leq 2}) \ar[d]_{f_\ast} \ar@{=>}[r]& \pi_{t-s} \Sigma^{-3}\RP^3 \ar[d]^{f_\ast}\\
E_2^{s,t}(\bY_{\leq 2}) \ar@{=>}[r] & \pi_{t-s} Y_2.\\
}
\]
By \Cref{lem:chasing-truncs}, we have that $j \in H^0(C_2,S^{4\sigma - 4})$ is a permanent cycle in the top spectral sequence, giving a non-trivial class $[j_2] \in \pi_0 \Sigma^{-3} \RP^3$, which is in fact a choice of splitting of the top cell. This implies that $f_*(j) = \iota$ is also a permanent cycle in the bottom spectral sequence, giving a homotopy class of $[\iota_2] \in \pi_0 Y_2$.

We also have a diagram of spectral sequences
\[
\xymatrix{
E_2^{s,t}(\bX_3^7) \ar[d]_{f_\ast} \ar@{=>}[r]& \pi_{t-s} \Sigma^{-3}\RP^0_{-4} \ar[d]^{f_\ast}\\
E_2^{s,t}(\bY_3^7) \ar@{=>}[r] & \pi_{t-s} Y_3^7,\\
}
\]
and we next study the image of $\delta_*[j_2] \in \pi_{-1}\Sigma^{-3}\RP_{-4}^0$ under $f_*$. By \Cref{lem:boundaryuptofiltration}, $\delta_*[j_2]=\nu ( [\wchihfs^2 j] + a[\nu b_7]) $. 

First, note that since $f_*([\nu b_7]) $ has filtration 7, it must actually be zero as $E_2^{7,3}(\bY_3^7) = 0$. This implies that 
\[f_*(\delta_*[j_2] ) = \nu f_* ([\wchihfs^2 j]). \]

From Theorem 6.1.6 and Proposition 5.3.1 of \cite{BGH} we have that
\[
0 = \wchi^2 \in  H^4(\GG_2,\W)\cong H^4(\GG_2, E_0)
\]
and hence 
\[
0 = \wchi^2 \iota = f_\ast (\wchihfs^2 j) \in H^4(\GG_2,E_0S^{4\sigma-4}) = E_2^{4,0}(\bY_3^7).
\]
This implies that $f_\ast [\wchihfs^2 j]$ must be detected in filtration greater than $4$ in the spectral sequence for $Y_3^7$. Thus $f_\ast [\wchihfs^2 j]$ could be detected in $E_2^{s,s-4}(\bY_3^7)$, with $4<s\leq 7$. By \Cref{lem:truncated-e2-bis}, and since $\pi_1 E = 0 = \pi_3 E $, the only such group which is non-zero is the case $s=6$, given by
\[ E_2^{6,2}(\bY_3^7) = H^6(\GG_2, E_2 ).\]
Thus $f_\ast [\wchihfs^2 j]$ must have filtration 6. 
From \cite[Table 2]{BBGHPScoh}, we know that $\nu$-multiplication 
\[H^6(\GG_2, E_2) \xrightarrow{\nu} H^7(\GG_2, E_6)\] 
is trivial (since every class in $H^6(\GG_2, E_2)$ is a multiple of $\eta$), thus $\nu f_\ast [\wchihfs^2 j] = f_\ast ( \nu  [\wchihfs^2 j]  )$ must be in higher filtration, i.e. in filtration at least 8. However, there is nothing in filtration 8 in the homotopy of $Y_3^7$, so $f_\ast ( \nu  [\wchihfs^2 j]  ) = 0$.

Therefore, the image of $[\iota_2 ] := f_*[j_2] $ in $\Sigma Y_3^7$ is 
\[ \delta_*[\iota_2] = \delta_*f_*[j_2] = f_*(\nu [\wchihfs^2 j]) = 0.\]
Now we apply the second part of \Cref{lem:geo-bound} to conclude that $\iota$ is a permanent cycle in the spectral sequence for $Y_7$ as needed.
\end{proof} 

Now the following calculation of $\kappa(G_{48})$ is immediate.

\begin{cor}\label{prop:kappaG48} The short exact sequence 
\[
\xymatrix{
0 \ar[r] & \kappa(\GG_2^1)  \ar[r] & \kappa(G_{48})  \ar[r]^-{\phi_3^1} & \ZZ/2  \ar[r] & 0
}
\]
splits and there is  an isomorphism
\[
\kappa(G_{48}) \cong \kappa(\GG_2^1) \oplus \Z/2 \cong \ZZ/8 \oplus (\ZZ/2)^3.
\]
\end{cor}

We end this section with an analysis of the descent filtration of $\kappa(G_{48})$.

\begin{thm}\label{prop:kappaG48-filt} In the filtration
\[
0 \subseteq \kappa_7(G_{48}) \subseteq \kappa_5(G_{48}) \subseteq \kappa_3(G_{48}) = \kappa(G_{48}) \cong
\ZZ/8 \times (\ZZ/2)^3
\]
we have isomorphisms
\begin{align*}
\kappa_{5}(G_{48}) &= \kappa_5(\GG_2^1) \cong \ZZ/8 \oplus \ZZ/2\\
\kappa_{7}(G_{48}) &= \kappa_7(\GG^1_{2}) \cong \ZZ/2\\
\kappa_s(G_{48}) &= \kappa_s(\GG_2^1) = 0,\qquad s > 7. 
\end{align*}
Furthermore,
\[
\xymatrix@R=10pt{
\kappa(G_{48})/ \kappa_5(G_{48}) \ar[r]_-\cong^-{\phi_3} &
\ZZ/2\{\eta\wchi\} \times \ZZ/2\{\zeta \langle \wchi, 2, \eta\rangle\} \\
\kappa_5(G_{48})/ \kappa_7(G_{48}) \ar[r]_-\cong^-{\phi_5} & \ZZ/4\{\zeta e \nu\} \times \ZZ/2\{\zeta\wchi\eta^2\}\\
\kappa_7(G_{48}) \ar[r]_-\cong^-{\phi_7} & \ZZ/2\{\zeta e\eta^3\}.
}
\]
\end{thm}

\begin{proof} Consider the diagram
\[
\xymatrix{
0 \ar[r]& \kappa_5(\GG_2^1)\ar[d]^{\cap} \ar[r] & \kappa(\GG_2^1)\ar[r]^-{\phi_3}\ar[d]^{\cap} &
H^3(\GG_2,E_2) \ar[d]^{=}\\
 0 \ar[r]& \kappa_5(G_{48}) \ar[r] & \kappa(G_{48})\ar[r]^-{\phi_3} & H^3(\GG_2,E_2)
}
\]
where both $\phi_3$ mean their restrictions to $\kappa(G_{48})$ and $\kappa(\GG_2^1)$.
Both horizontal sequences are exact by \cref{lem:desfilt}. By \cite[\S7]{BBGHPScoh} (see also \cref{fig:HFPSSG2}),
\[
H^3(\GG_2,E_2) \cong \Z/2\{\zeta \langle \wchi, 2, \eta\rangle\}\oplus \Z/2\{\eta\widetilde{\chi}\}.
\]
From \cref{thm:comb-2-filts-fin}, we have that $\phi_3(\kappa(\GG_2^1)) = \Z/2\{\zeta \langle \wchi, 2, \eta\rangle\}$ and \cref{thm:Qfirst} shows that
$\phi_3(Q) = \eta\widetilde{\chi}$. By chasing in the above diagram,
we see that
\[
\kappa_5(G_{48}) = \kappa_5(\GG_2^1).
\] 
We then conclude that $\kappa_s (G_{48}) = \kappa_s(\GG_2^1)$ for $s \geq 5$, and 
the calculation of $\phi_5$ and $\phi_7$ follows from \cref{thm:comb-2-filts-fin}.
\end{proof}

 % !TEX root = pic-master.tex

\section{Picard elements detected by $E^{hG_{48}}$, and the calculation of $\kappa_2$}

Now that we have determined the subgroup $\kappa(G_{48}) \subseteq \kappa_2$, it remains to compute the quotient and study the extension in order to understand $\kappa_2$ completely.

In this section, we will show that $\kappa_2/\kappa(G_{48}) $ is $\ZZ/8$ in two steps: first, we show that the upper 
bound on this quotient is $\ZZ/8$ in \Cref{prop:upper-bnd-top}, and then we exhibit a generator of order 8 in 
\Cref{thm:first-quotient}. Then, in \Cref{kappa25:split} we show that the extension problem is trivial and finally, we 
conclude the paper by analyzing the descent filtration on $\kappa_2$ in \Cref{thm:kappa}.

The elements of $\kappa_2$ detected by $\tmf$ are those $X\in \kappa_2$ which do not have an $\tmf$-orientation in 
the sense of \Cref{defn:subgroupfilt}. We will be looking at the obstructions for having such an orientation. By 
\Cref{prop:detect-kappa-ANSS}, those obstructions come in the form of differentials on a $\GG_2$-invariant generator 
\[\iota_X \in H^0(\GG_2, E_0X )\subseteq H^0(G_{48}, E_0 X) \] 
in the homotopy fixed point spectral sequence 
\begin{align}\label{eq:G48-X-SS}
E_2^{s,t}(G_{48},X)\cong H^s(G_{48}, E_tX) \Rightarrow \pi_{t-s} (\tmf \wedge X).
\end{align}
We have an isomorphism of this $E_2$-page with $E_2^{s,t}(G_{48},S^0)$ determined by the choice of $\iota_X$. To 
understand the possible fates of $\iota_X$, we'll need a detailed analysis of the homotopy fixed point spectral sequence 
for $\tmf$, over which \eqref{eq:G48-X-SS} is a module. We recall some details about that spectral sequence in the first 
subsection.

\subsection{The homotopy fixed point spectral sequence for $\tmf$.} 

The complete calculation of homotopy fixed point spectral sequence $E_*^{*,*}(G_{48},S^0)$ can be found in \cite{DuanKongLiLuWang}. This is closely related to any calculation of the homotopy groups of various version of 
topological modular forms spectrum, since the $K(2)$-localization of $tmf$ at $p=2$ is $\tmf$. See, for example, 
\cite{tmfbook,tbauer,TmfAt2,BrunerRognes-tmf}. A basic point is that this is a spectral sequence of rings; in particular the 
differentials satisfy the Leibniz rule.

The cohomology ring $H^\ast(G_{48},E_\ast)$ is described in Theorem 2.15 and the accompanying Figure
5 of \cite{BobkovaGoerss} and we content ourselves with a summary here. There are elements
\begin{align*}
c_4 &\in H^0(G_{48}, E_8) &c_6 \in H^0(G_{48}, E_{12})\\
\Delta& \in H^0(G_{48}, E_{24}) &j\in H^0(G_{48}, E_0)
\end{align*}
so that 
\[
E_2^{0,t} = H^0(G_{48}, E_t) = \ZZ_2[[j]][c_4, c_6, \Delta^{\pm 1}]/(c_4^3-c_6^2-(12)^3\Delta, c_4^3-j\Delta)
\]
and 
\begin{equation}\label{eq:coho-g48}
H^\ast(G_{48},E_\ast) \cong
H^0(G_{48},E_\ast)[\eta,\nu,\mu,\epsilon,\kappa,\kappabar]/R,
\end{equation}
where $R$ is a rather elaborate list of relations. The classes
\begin{align*}
\eta &\in H^1(G_{48}, E_2) &\nu \in H^1(G_{48}, E_{4})\\
\epsilon& \in H^2(G_{48}, E_{10}) &\kappa\in H^2(G_{48}, E_{16})\\
\kappabar &\in H^4(G_{48}, E_{24})
\end{align*}
are named for the elements of $\pi_\ast S^0$ that they detect. The class $\mu \in H^1(G_{48}, E_{6})$ has
the property that $d_3(\mu) = \eta^4$. Crucial among these generators is $\kappabar \in
H^4(G_{48},E_{24})$, 
and one important property it has is that multiplication by
$\kappabar$ on $H^*(G_{48}, E_*)$ is onto in cohomological degree greater than 4 and has no annihilators of positive cohomological dimension.

If $r$ is even, then $d_r=0$. The odd differentials form a complicated but well-understood pattern. 
Crucially, $\Delta^8$ is a permanent cycle and the entire spectral
sequence is periodic of period 192 in degree $t$. 

Here and elsewhere
we will use the standard convention of writing elements by their names at $E_2$ even though they may
no longer be products at $E_r$. 

The $d_3$ differentials are determined by
\begin{equation}\label{eq:diffs48:3}
\begin{aligned}
d_3(\mu) &=\eta^4, & & &
d_3(c_6)&=c_4\eta^3
\end{aligned}
\end{equation}
and the fact that all other generators are $d_3$-cycles. 

The $d_5$ and $d_7$ differentials are determined by the formulas
\begin{equation}\label{eq:diffs48:5-7}
\begin{aligned}
&d_5(\Delta) =\overline{\kappa} \nu, & \\
& d_7(4\Delta) =\overline{\kappa}\eta^3, \quad &d_7(2\Delta^{2}) = \Delta\overline{\kappa} \eta^3, \\
& d_7(4\Delta^{3}) = \Delta^2\overline{\kappa} \eta^3,\quad\quad & d_7(\Delta^4) =\Delta^3 \overline{\kappa} \eta^3,
\end{aligned}
\end{equation}
the multiplicative structure, and the fact that they vanish on all the remaining generators. 

Next are the $d_9$ differentials, regarding which we need only the following
\begin{equation}\label{eq:diffs48:9-23}
\begin{aligned}
d_9(\Delta^2 \epsilon)&=\kappabar^2 \kappa\eta, & \quad\quad
d_9(\Delta^6 \epsilon)&=\Delta^4\kappabar^2 \kappa\eta.
\end{aligned}
\end{equation}

In fact, there are considerably more differentials, 
but those are not needed for our arguments below. We will focus here on the calculation in a  small range, which we present in \Cref{fig:hfpss-tmf}, and some key properties summarized in the following result. 

\begin{prop}\label{rem:notes-on-tmf-ss} In the homotopy fixed point spectral sequence
\begin{equation}\label{eq:tmf-hfpss}
E_2^{s,t}(G_{48}, S^0)=H^s(G_{48},E_t) \Longrightarrow \pi_{t-s}E^{hG_{48}}
\end{equation}
we have the following:
\begin{enumerate}

\item The non-zero differentials that originate on the  zero line are $d_3$, $d_5$ and $d_7$. 

\item The spectral sequence collapses at the $r=24$ page and
\[
E^{s,\ast}_\infty(G_{48},S^0) =0\quad\mathrm{if}\quad s \geq 23. 
\]
\item\label{item3:notes-on-tmf-ss} For $r>7$, we have $E_r^{r,r-1} (G_{48},S^0) = 0$.
\item\label{item4:notes-on-tmf-ss} For $s>0$, $E_\infty^{s,s} (G_{48},S^0) = 0$.
\item The  spectrum $E^{hG_{48}}$ is 192-periodic, with a class $\Delta^8\in \pi_{192} E^{hG_{48}}$ a periodicity
generator. Thus the composition 
\[
\Sigma^{192} E^{hG_{48}} \xra{\Delta^8 \wedge \iota} \tmf \wedge \tmf \xra{m} \tmf,
\]
where $m$ is the ring spectrum multiplication, is an equivalence.
\end{enumerate}
\end{prop}

\begin{figure}
\includegraphics[height = 0.9\textheight]{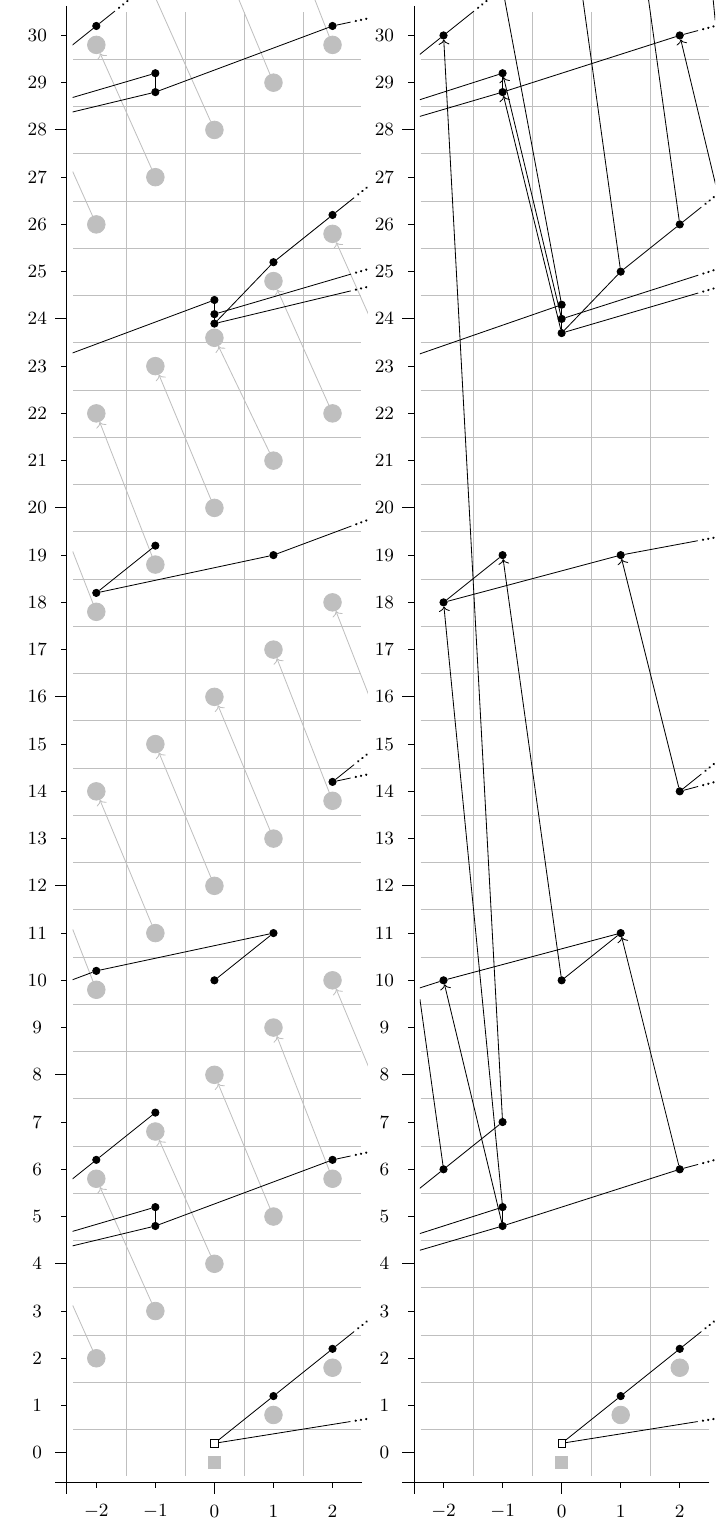}
\captionsetup{width=\textwidth}
\caption{The homotopy fixed point spectral sequence  \eqref{eq:tmf-hfpss} for $\pi_*E^{hG_{48}}$ in a small range. The $x$-axis is $t-s$ and the $y$-axis is $s$. The circles represent $j\FF_2[[j]]$ and bullets represent $\FF_2$. $n$ bullets connected by a vertical line represent $\ZZ/{2^n}$. Lines of slope 1 and 1/3 represent multiplication by $\eta$ and $\nu$, respectively.}
\label{fig:hfpss-tmf}
\end{figure}

\subsection{Establishing an upper bound for $\kappa_2/\kappa(G_{48})$}

In this subsection we will produce an injective homomorphism $\kappa_2/\kappa(G_{48}) \to \ZZ/8$, see \Cref{prop:upper-bnd-top}. 

So we fix $X \in \kappa_2$, let $\iota_X \in E_0X$ be a $\GG_2$-invariant generator for $X$, 
and we examine what can happen to $\iota_X$ in the spectral sequence \eqref{eq:G48-X-SS}. The first result handles $d_3$.  

\begin{lem}\label{lem:top-bound-1} In the homotopy fixed point spectral sequence
\[
H^s(G_{48},E_t X) \Longrightarrow \pi_{t-s}(E^{hG_{48}} \wedge X)
\]
we have $d_3(\Delta^k\iota_X) = 0$ for all $k$.
\end{lem}

\begin{proof} Since $\Delta$ is a $d_3$ cycle (see \eqref{eq:diffs48:3}), it suffices to prove that 
\[ d_3(\iota_X) \in H^3(G_{48}, E_4X) \cong \FF_2[[j]] \eta^3\frac{c_4 c_6}{\Delta} \iota_X\]
is zero. 
A priori, we have
\[
d_3(\iota_X) = f(j) \eta^3\frac{c_4 c_6}{\Delta} \iota_X
\]
for some power series $f(j) \in \FF_2[[j]]$. We apply $d_3$ again, and use the fact that $\frac{c_6^2}{\Delta} \equiv j \mod 2$ to obtain that $f(j)$ satisfies
\[
f(j) + j f(j)^2 = 0.
\]
This implies that $f(j)=0$.
\end{proof} 

\begin{rem}
The reader is encouraged to compare this computation of $d_3(\iota_X)$ with the computation of $d_3$ in the Picard spectral sequence in \cite[Theorem 8.2.2]{MathewStojanoska}.
\end{rem}

\begin{rem}\label{rem:setting-up-dr}\label{rem:diff-patt-top} Now that we know $d_3(\iota_X) = 0$,
\eqref{eq:coho-g48} and the differential patterns of \eqref{eq:diffs48:3}, \eqref{eq:diffs48:5-7} and \eqref{eq:diffs48:9-23} (depicted in \Cref{fig:hfpss-tmf}) give that
\[
d_5(\iota_X) = a\Delta^{-1}\kappabar\nu\iota_X,\qquad \text{ for some } a \in \ZZ/4
\]
and if $a = 0$, 
\[
d_7(\iota_X) = b\Delta^{-1}\kappabar\eta^3\iota_X, \qquad  \text{ for some }  b \in \ZZ/2.
\]
\end{rem}
We use this observation to prove the next result, which finds that some $\Delta$-multiple of $\iota_X$ must be a $d_7$-cycle.

\begin{lem}\label{lem:top-bound-2} Let $X \in \kappa_2$ and let $\iota_X \in E_0X$ be a $\GG_2$-invariant
generator. Then there is an integer $k$ so that 
\[
d_5(\Delta^k\iota_X) = 0 = d_7(\Delta^k\iota_X).
\]
\end{lem} 

\begin{proof} In the homotopy fixed point spectral sequence
for $E^{hG_{48}}$ we have $d_5(\Delta) = \kappabar\nu$, see \eqref{eq:diffs48:5-7}. Then
\[
d_5(\Delta^k) =
\begin{cases}
k\Delta^{k-1}\kappabar\nu,& k \not\equiv 0\ \mathrm{modulo}\ 4;\\
0, &k\equiv 0\ \mathrm{modulo}\ 4.
\end{cases}
\]
and
\[
d_7(\Delta^k) = \Delta^{k-1}\kappabar\eta^3,\qquad k \equiv 4\ \mathrm{modulo}\ 8.
\]
We use this and the differential pattern of \Cref{rem:diff-patt-top}. Suppose $d_5(\iota_X) = a \Delta^{-1}\kappabar\nu \iota_X$
with $0 \leq a \leq 3$; then $d_5(\Delta^{-a}\iota_X) = 0$, and
\[
d_7(\Delta^{-a}\iota_X) = c\Delta^{-a-1}\kappabar\eta^3\iota_X, \qquad c\in \ZZ/2.
\]
If $c=0$ we are done. If not, $d_5(\Delta^{-a-4}\iota_X) = 0$ as well and
\[
d_7(\Delta^{-a-4}\iota_X) = 0. \qedhere
\]
\end{proof}

Let $\iota_X$ be a $\GG_2$-invariant generator for $X$. By \Cref{lem:top-bound-2}, we can find an integer $k$ such that $\Delta^k \iota_X$ is a $d_7$-cycle in the homotopy fixed point spectral sequence $E_r^{s,t}(G_{48},X)$ \eqref{eq:G48-X-SS}. It turns out that this implies $\Delta^k \iota_X$ is a permanent cycle, which has the following consequence.

\begin{prop}\label{prop:all-are-modules}\label{lem:top-bound-3} 
Let $X \in \kappa_2$, and let $k$ be an integer such that $\Delta^k \iota_X$ is a $d_7$-cycle. Then $\Delta^k\iota_X$ a permanent cycle extending to an equivalence of $\tmf$-modules
$\Sigma^{24k} E^{hG_{48}} \simeq E^{hG_{48}} \wedge X$. 
\end{prop}

\begin{proof} 
To see that $\Delta^k \iota_X$ is a permanent cycle, we use an adaptation of the usual proof that $\Delta^8$ is permanent cycle in the spectral
sequence for $\tmf$. If $\Delta^k\iota_X$ is a $d_r$-cycle, then we have an isomorphism of
$E_r^{*,*}(G_{48},S^0)$-modules
\[
E_r^{*,*}(G_{48},X) \cong E_r^{*,*}(G_{48},S^{24k}) \cong E_r^{*,*-24k} (G_{48},S^{0}). 
\]
We know this is the case for $r=7$.
 The image of $\Delta^k\iota_X$ under $d_r$ then lives in the group
\[
E_r^{r,r+24k-1}(G_{48},X) \cong E^{r,r-1} _r(G_{48},S^0),
\]
but from \Cref{rem:notes-on-tmf-ss}\eqref{item3:notes-on-tmf-ss}, we know when $r> 7$, the groups $E_r^{r,r-1}(G_{48},S^0) $ vanish. Thus we conclude that we cannot have any further non-trivial differentials on $\Delta^k\iota_X$.
\end{proof} 

As a first corollary of this result, we get a bound on the descent filtration on $\kappa_2$.
\begin{cor}\label{lem:kappa8is0}
The subgroup $\kappa_{2,8} \subset \kappa_2$ is trivial.
\end{cor}
\begin{proof}
Suppose $X \in \kappa_{2,8}$, so that by definition, $d_r(\iota_X) = 0$ for $r\leq 7$ in the homotopy fixed point spectral sequence $E_r^{s,t}(\GG_2, X)$. Then we get that the image of $\iota_X$ in the homotopy fixed point spectral sequence $E_r^{s,t}(G_{48}, X)$ is also a $d_7$-cycle, but then \Cref{prop:all-are-modules} shows that $\iota_X$ is a permanent cycle in the latter spectral sequence. Hence, $X \in \kappa_{8}(G_{48}) = \kappa_{2,8}\cap \kappa(G_{48})$. But from \Cref{prop:kappaG48-filt}, we know that the group $\kappa_{8}(G_{48})$ is trivial, implying that $X$ is trivial.
\end{proof}

Note that since $\tmf$ is $24\times 8 = 192$-periodic, the integer $k$ in \Cref{prop:all-are-modules} is only
well-defined modulo $8$. 

\begin{prop}\label{prop:upper-bnd-top} Define a homomorphism $\kappa_2 \to \ZZ/8$ by sending $X$ to $k$, where
$E^{hG_{48}} \wedge X \simeq \Sigma^{24k} E^{hG_{48}}$. This gives an injection
\[
\kappa_2/\kappa(G_{48}) \longrightarrow \ZZ/8.
\]
\end{prop}

\begin{proof} Suppose $X \in \kappa_2$ goes to $0$; that is, we have an equivalence of $E^{hG_{48}}$-module
spectra $E^{hG_{48}} \simeq E^{hG_{48}} \wedge X$. Then \Cref{lem:old-is-new}, with the input of \Cref{rem:notes-on-tmf-ss}\eqref{item4:notes-on-tmf-ss} implies that $X $ is in $\kappa(G_{48})$. 
\end{proof} 

\begin{rem}
Let $\Pic^0(E^{hG_{48}})$ denote the Picard group of invertible $E^{hG_{48}}$-modules $N$ so that
$E_\ast N$ is in even degrees. This is equivalent to the Picard group $\Pic^0_{G_{48}}(E)$ of invertible $K(2)$-local $E$-modules
$M$ in  $G_{48}$-spectra so that $\pi_\ast M$ is in even degrees. Then $\pi_0M$ is an invertible $E_0$-module
equipped with a compatible action of $G_{48}$. Thus we have a map
\[
\Pic^0(E^{hG_{48}}) \to H^1(G_{48},E_0^\times).
\]
Using methods similar to \cite[Proposition 5.2]{Karamanov} at the prime 3, or the uncompleted version in \cite[Appendix B]{MathewStojanoska},
 one can show that the target is isomorphic to $\ZZ/12$ generated by
the image of $\Sigma^2 E^{hG_{48}}$. In particular, the map is onto and we have a short exact sequence
\[
0 \to K \to \Pic^0(E^{hG_{48}}) \to H^1(G_{48},E_0^\times) \to 0.
\]
Using the techniques of \cite{MathewStojanoska} and \cite{HeardMathewStojanoska},
and using the same input we used above, one can also show 
that $K$ is $\ZZ/8$ generated by $\Sigma^{24}E^{hG_{48}}$; in particular, $\Pic^0(\tmf) \cong \ZZ/96$
and hence $\Pic(E^{hG_{48}}) \cong \ZZ/192$ generated by $\Sigma E^{hG_{48}}$. \Cref{thm:first-quotient} below
shows that the  induced map $\kappa_2/\kappa(G_{48}) \to K$ is an isomorphism.

Note that $E^{hG_{48}}\simeq L_{K(2)}TMF$ and the analogous result for the Picard group of $TMF$-modules is \cite[Theorem 8.2.2]{MathewStojanoska}.
\end{rem}

\subsection{A generator for $\kappa_2/\kappa(G_{48})$.}
In the previous section we established that the quotient $\kappa_2/{\kappa(G_{48})}$ is at most $\ZZ/8$.
The goal of this section is to show that it is exactly $\ZZ/8$ by showing that there exists $P \in \kappa_2$ such
that $P \wedge\tmf\simeq \Sigma^{-24}\tmf$.

The invertible spectrum $P$ that we'll use for this purpose comes from Gross-Hopkins duality \cite{HopkinsGross,GrossHopkins2}. This seminal work implies that at each height $n$, the spectrum $I_n$ obtained as the Brown-Comenetz dual of the $n$-th monochromatic layer of the sphere, is an invertible object in the $K(n)$-local category; a more hands-on presentation of this fact can be found in \cite{StrickGrossHop}. For each $n$, there is a decomposition of $I_n$ as a smash product
\begin{align}\label{eq:define-Pn}
I_n \simeq S^{n^2-n}\wedge S\langle \det \rangle \wedge P_n, 
\end{align}
where $S\langle \det \rangle $ is the determinant sphere, as constructed in \cite{BBGS}, for example, and $P_n$ is an element of $\kappa_n$. The invertibility of $I_n$ then implies that for any $K(n)$-local spectrum $Z$, we have
\begin{align}\label{eq:Gross-Hopkins-relation}
 I_n Z = F(Z, I_n) \simeq F(Z,L_{K(n)}S^0) \wedge I_n \simeq D_nZ \wedge \Sigma^{n^2-n}S^0\langle \det\rangle \wedge P_n,
 \end{align}
where the smash product is understood to be $K(n)$-local.

The $P$ we consider at height $2$ is exactly this $P_2$. We can find a minimum for its order in $\kappa_2$ by studying the interplay between the Spanier-Whitehead dual $D_2\tmf = F(\tmf, L_{K(n)}S^0)$ and the Gross-Hopkins dual $I_2 \tmf = F(\tmf, I_2)$ of $\tmf$. 
The Spanier-Whitehead as well as Gross-Hopkins duals of $\tmf$ are well known.
\begin{thm}[\cite{Bobkova_DTMF}]\label{thm:DTMF} There is an equivalence
of $\tmf$-modules 
\[
D_2 \tmf\simeq\Sigma^{44} \tmf.
\]
\end{thm}
See also \cite[Theorem 13.25]{BGHS} for a different proof of \cref{thm:DTMF}.

The next result identifies the Gross-Hopkins dual of $\tmf$. It has a long history but has not been recorded in the literature in the precise form that we use, and here we present a proof based on calculations from \cite{Pham}.

\begin{thm}\label{thm:ITMF} There is an equivalence
of $\tmf$-modules 
\[
I_2 \tmf \simeq\Sigma^{22}\tmf.
\]
\end{thm}
\begin{proof}
Substituting the equivalence from \Cref{thm:DTMF} into  \eqref{eq:Gross-Hopkins-relation}, and using that
$\tmf\wedge S\langle \det \rangle\simeq \tmf$ by \cite[Corollary 3.11]{BBGS}, we get 
\begin{align}
I_2(\tmf) \simeq \Sigma^{46} P_2 \wedge \tmf.
\end{align}
By \Cref{lem:top-bound-2} and \Cref{lem:top-bound-3}, applied to $X =P_2 \in \kappa_2$, we conclude that there is an
integer $k$ such that $P_2\wedge \tmf \simeq \Sigma^{24k} \tmf $, which in turn implies that $I_2\tmf \simeq \Sigma^{46+24k} 
\tmf$. Since $\tmf$ is $192$-periodic the value of $k$ is only determined modulo $8$. The following 
discussion shows that $k\equiv -1$.

Let $A(1)$ be
a finite $2$-local spectrum with the property that $H^\ast (A(1),\FF_2)$ is free on a generator in degree
zero over the subalgebra of  the Steenrod algebra generated by $\mathrm{Sq}^1$ and $\mathrm{Sq}^2$.
We have some choice here, and we choose a version of $A(1)$ which is self-dual; in \cite{Pham} these
are called $A_1[10]$ and $A_1[01]$. Either choice will do. We will show that 
\[
I_2(\tmf) \wedge A(1) \simeq \Sigma^{46-24} \tmf \wedge A(1)
\]
and that if $k \not\equiv -1$ modulo $8$, then
\[
I_2(\tmf) \wedge A(1) \not\simeq \Sigma^{46+24k} \tmf \wedge A(1).
\]

We have that $DA(1) = F(A(1),S^0) = \Sigma^{-6}A(1)$. Since $A(1)$ is a type $2$ complex
$L_1(\tmf \wedge A(1)) \simeq \ast$ and
\[
I_2(\tmf \wedge A(1)) = I_{\ZZ/2^\infty} (\tmf \wedge A(1)).
\]
Here, $I_{\ZZ/2^\infty}X$ denotes the Brown-Comenetz dual of $X$, whose homotopy groups are $ (\pi_*X)^\vee = \Hom(\pi_*X, \ZZ/2^\infty)$.
Using \Cref{prop:all-are-modules},  \eqref{eq:Gross-Hopkins-relation}, and the fact that $A(1)$ is a finite complex,
we have
\begin{align*}
I_{\ZZ/2^\infty} (\tmf \wedge A(1))
&\simeq S^2 \wedge S^0\langle \det \rangle \wedge P \wedge \Sigma^{44}E^{hG_{48}} \wedge \Sigma^{-6} A(1)\\
& \simeq \Sigma^{40}\tmf \wedge P \wedge A(1)\\
& \simeq \Sigma^{40 + 24k} \tmf \wedge A(1)
\end{align*}
with $k$ to be determined. This gives an equation
\begin{align*}
[\pi_{-n}(\tmf \wedge A(1))]^\vee &\cong \pi_{n-40-24k}(\tmf \wedge A(1))\\
&\cong \pi_{n+152-24k}(\tmf \wedge A(1)).
\end{align*}
The last isomorphism uses the periodicity of $\tmf$.
Investigating \cite[Figures 22-25]{Pham} for low values of $n$ we see that this is only
possible if $k\equiv -1$ modulo $8$. Specifically, we see that if $1 \leq n \leq 15$, then
\[
\pi_{-n}(\tmf \wedge A(1)) = \pi_{192-n}(\tmf \wedge A(1)) = 0. 
\]
In the same range, $\pi_{n+152-24k}(\tmf \wedge A(1)) = 0$ only if $k=-1$. 
\end{proof}

\begin{rem}\label{rem:dualityproof}
There are several possible alternatives to this proof. For one, the calculations given above could be combined with a strategy analogous to the prime 3 version given in \cite[Proposition 2.4.1]{BehrensModular} to arrive at the same conclusion.
A different approach would be to cite Greenlees \cite[Example 4.4]{GreenleesTHHDuality}, for the fact that $tmf$ is Gorenstein self-dual, which is also done in much more detail in the newer \cite{BrunerRognes-tmf,BrunerGreenleesRognes}. Then one can apply \cite[Proposition 4.1]{GreenleesStojanoska} to obtain that the non-connective, non-periodic $Tmf$ is Anderson self-dual, and then $K(2)$-localize to obtain that $\tmf = L_{K(2)}Tmf$ is Gross-Hopkins self-dual as claimed.

The most conceptual proof would use Serre duality on a suitable cover of the compactified moduli stack of elliptic curves. The analogue at the prime 3 was done in \cite{Stojanoska}, while the algebraic calculations needed at $p=2$ are set up in \cite{TmfAt2}. 
\end{rem}

Combining these two results give the main result of this section.

\begin{theorem}\label{thm:first-quotient}
Let $P=P_2 \in \kappa_2$ be determined by the equation \eqref{eq:define-Pn} at $n=2$, i.e. 
\[
I_2 \simeq  S^{2}\wedge S^0\langle \det\rangle \wedge P.
\]
Then we have an equivalence of $\tmf$-modules
\[
\Sigma^{-24} \tmf \simeq \tmf \wedge P
\]
and, thus, an isomorphism
\[
\kappa_2/{\kappa(G_{48})} \cong \ZZ/8.
\]
\end{theorem}

\begin{proof}
We set $Z = \tmf$ in \eqref{eq:Gross-Hopkins-relation}, and then combine with
\Cref{thm:DTMF} and  \Cref{thm:ITMF} to get
\[
\Sigma^{22}\tmf \simeq  \Sigma^{44}\tmf \wedge \Sigma^2 S^0\langle \det \rangle \wedge P.
\]
Then we note that $S^0\langle \det \rangle \wedge \tmf \simeq \tmf$ since $G_{48}$
is in the kernel of the determinant map; see \cite[Corollary 3.11]{BBGS}. This then reduces to the
indicated equivalence.

The final isomorphism follows from \Cref{prop:upper-bnd-top}, since $8$ is the smallest integer $m$ such that $\tmf \wedge P^{\wedge m} \simeq \tmf$.
\end{proof}

\begin{rem}\label{rem:PasPhi5}
In the homotopy fixed point spectral sequence $E_r^{*,*}(G_{48},P)$, we have that $\iota_P$ is a $d_3$-cycle and $d_5(\iota_P) $ is a generator of $E_5^{5,4}(G_{48}, P) \cong \ZZ/4 \{ k\nu \}$.
Compare \Cref{rem:diff-patt-top}.
\end{rem}

\begin{rem}
A similar comparison of the effect of $P_2$ in $E^{hC_2}$ is done in \cite[Theorem 6.6]{HeardLiShi}, but that result shows that $C_2$ can only see that the order of $P_2$ is divisible by 2.
\end{rem}

\begin{rem}\label{rem:as-close-as-we-get}
It is natural to ask whether we can construct an element of $\kappa_2$ which maps to a generator of
$\kappa_2/\kappa(G_{48})\cong \Z/8$ using the $J$ construction of \cref{sec:j-construction}. We have not been able 
to do quite this much,
but here is what we \emph{can} do towards that goal.
 
Recall from \eqref{eq:K-split-S2} that $\mathbb{S}_2$ has an open normal subgroup $K$ with the property that the composition
\[
G_{24} \longr \SS_2 \longr \SS_2/K
\]
is an isomorphism, which produces a splitting $ \mathbb{S}_2 \cong K \rtimes G_{24} $. Unfortunately, $K$ is not Galois-invariant, so we cannot write $G_{24}$ or $G_{48}$ as the quotient of $\GG_2$. Still, we can use the $J$-construction for the quotient $\SS_2/K \cong G_{24}$ to produce an invertible $E^{h\SS_2}$-module, which if we knew is Galois invariant, would be a generator of the quotient $\kappa_2/\kappa(G_{48}).$

Let $\mathbb{H}$ be real quaternion algebra; this a four dimensional irreducible representation of $Q_8$
\[
\mathbb{H} \cong \R\{1,i,j,k\}.
\]
The action of $C_{3}$ on $Q_8$ which permutes $i,j,k$ extends $\mathbb{H}$ to a $G_{24}$-representation. We can let 
$\SS_2$ act on $\mathbb{H}$ via the quotient $q \colon \mathbb{S}_2 \to \mathbb{S}_2/K\cong G_{24}$.
Let $V=\mathbb{H}-4$ and let
\[
J(V) = (E \wedge S^{V})^{h\mathbb{S}_2} \simeq (E^{hK} \wedge S^{V})^{hG_{24}},
\]
compare \eqref{eq:defnJQFK}; however, this is not quite the same construction as  \eqref{eq:defnJQFK}, since $\SS_2/K$ is not a quotient of $\GG_2$. Nonetheless, $J(V)$ is a $K(2)$-local invertible $E^{h\SS_2}$-module, and
\[
E^{hK} \wedge J(V)  \simeq E^{hK}.
\]
Furthermore, by \cref{prop:shearing-iso}, we have an equivalence
\[
E^{hG_{24}} \wedge_{E^{h\mathbb{S}_2}}  (E \wedge S^V)^{h\mathbb{S}_2} \simeq (E \wedge S^V)^{hG_{24}}
\]
In Propositions 13.22 and 13.23 \cite{BGHS}, all elements of the form $(E \wedge S^V)^{hG_{24}}$ in
the Picard group $\Pic(E^{hG_{24}})$ of invertible $E^{hG_{24}}$-modules are computed. This is done by showing that the $J$ homomorphism factors as\footnote{There is a slight gap 
in \cite[Proposition 3.23 (2)]{BGHS}. We only know that the right vertical map is an inclusion, as there is no proof in the 
literature that $\Pic(E^{hG_{24}}) \cong \Z/192$.} 
\[
\xymatrix{ RO(G_{24})\ar[d]^-{\psi} \ar[r]^-{J} & \Pic(E^{hG_{24}}) . \\
\Z \oplus \Z/8 \ar@{->>}[r] & (\Z \oplus \Z/8)/ (24,1) \cong \ZZ/192 \ar[u]_-{\subseteq}
}
\]
Here, the map
\[
\psi \colon RO(G_{24}) \to \Z \oplus \Z/8
\]
is defined by $\psi(W) =(\dim W, \lambda(W))$ for a certain characteristic class $\lambda$ which has been normalized 
so that $\lambda(\mathbb{H})=1$. Therefore, $\psi(V)=(0,1)$. 
From this, it follows that
\[
(E \wedge S^V)^{hG_{24}} \simeq \Sigma^{-24}E^{hG_{24}} .
\]
If we could find an $X \in \kappa_2$ such that $E^{h\SS_2} \wedge X \simeq J(V)$ then 
$X$ would be a generator for $\kappa_2/\kappa(G_{48})$. 
This can be phrased as the question of whether $J(V) \in \Pic(E^{h\SS_2}) $ is Galois invariant. 
\end{rem}

\subsection{The group $\kappa_2$}
In this final section we show that the extension 
\begin{equation}\label{eq:one-to-be-split}
\xymatrix{
0 \ar[r] & \kappa(G_{48}) \ar[r] & \kappa_2 \ar[r] & \Z/8\ar[r] & 0
}
\end{equation}
from \Cref{thm:first-quotient} is split, and we relate these groups to the descent filtration.
We begin with the following intermediate results, which examine the descent filtration and decompose the
subgroup $\kappa_{2,5} \subseteq \kappa_2$. Recall that $\kappa_5(G_{48}) = \kappa(G_{48}) \cap \kappa_{2,5}$.  

\begin{prop}\label{kappa25:SES} There is a short exact sequence
\[
0\rightarrow \kappa_5(G_{48}) \rightarrow \kappa_{2,5} \rightarrow \Z/8 \rightarrow 0,
\]
where $\kappa_{2,5}\rightarrow \Z/8$ is the composite
$\kappa_{2,5} \hookrightarrow \kappa_{2}\twoheadrightarrow \kappa_2/\kappa(G_{48})\cong \Z/8. $
\end{prop}

\begin{proof} 
Consider the following commutative diagram
\[
\xymatrix{
& \kappa_5(G_{48})\ar@{^{(}->} [d]\ar[r]& \kappa(G_{48})\ar@{^{(}->} [d] \ar[r]^-{\phi_3} &
\ar[d]^{=} H^3(\GG_2,E_2) \rightarrow 0\\
0 \ar[r]& \kappa_{2,5} \ar[r]& \kappa_2 \ar[r]^-{\phi_3} &  H^3(\GG_2,E_2),
}
\]
The lower sequence is exact by \cref{lem:desfilt}. That the top $\phi_3$ is onto follows from \Cref{prop:kappaG48-filt}, and the rest of the top sequence is exact because $\kappa_5(G_{48}) = \kappa(G_{48}) \cap \kappa_{2,5}$. In particular, the bottom $\phi_3$ must be onto as well, and then the Snake Lemma gives the needed isomorphism
\[
\kappa_{2,5}/\kappa_5(G_{48})\cong \kappa_2/\kappa(G_{48})\cong \Z/8. \qedhere
\]
\end{proof}

In order to give a splitting of \eqref{eq:one-to-be-split}, we will first split its restriction from \Cref{kappa25:SES}.
\begin{prop}\label{kappa25:split} The short exact sequence 
\[
0 \rightarrow \kappa_{5}(G_{48})\rightarrow \kappa_{2,5}\rightarrow \Z/8\rightarrow 0
\]
splits.
\end{prop}

\begin{proof} From \cref{prop:kappaG48-filt}, we have that
$
\kappa_5(G_{48})\cong \Z/8\oplus \Z/2.
$
The result will follow if we prove that $\kappa_{2,5}$ has exponent $8$. 

\cref{lem:desfilt} provides us with an exact sequence
\[
0
\rightarrow \kappa_{2,7}\rightarrow \kappa_{2,5} \rightarrow E_5^{5,4}(\GG_2,S^0).
\]
The right most term $E_5^{5,4}(\GG_2,S^0)$ has exponent $4$; see \cref{fig:HFPSSG2}.
Therefore, it suffices to prove that $\kappa_{2,7}$ has  exponent $2$. 
For this, we note that the map
\[
\phi_7 \colon \kappa_{2,7} \longr E_7^{7,6}
\] 
is injective as a consequence of \Cref{lem:kappa8is0}. But $E^{7,6}_7(\GG_2,S^0)$ has exponent $2$, from \cref{fig:HFPSSG2}. So, it follows that $\kappa_{2,7}$
has exponent $2$ and we are done.\end{proof}

The final calculation of $\kappa_2$ is now an immediate corollary.

\begin{Theorem}\label{thm:kappa} The short exact sequence
\[\xymatrix{
0 \ar[r] & \kappa(G_{48}) \ar[r] & \kappa_2 \ar[r] & \Z/8\ar[r] & 0.
}
\]
splits and gives an isomorphism
\[
\kappa_2 \cong \kappa(G_{48}) \oplus \ZZ/8 \cong (\ZZ/8)^2 \oplus (\ZZ/2)^3.
\]
\end{Theorem}

\begin{proof} That the short exact sequence splits follows from \cref{kappa25:SES} and \cref{kappa25:split}.
For the calculation of $\kappa(G_{48})$, see \Cref{prop:kappaG48}.
\end{proof}

We end by recording the descent filtration on $\kappa_2$.

\begin{thm}\label{prop:kappa-filt} In the filtration
\[
0 \subseteq \kappa_{2,7} \subseteq \kappa_{2,5} \subseteq \kappa_{2,3} = \kappa_2 \cong
(\ZZ/8)^2 \times (\ZZ/2)^3
\]
we have isomorphisms
\begin{align*}
\kappa_{2,5} & \cong (\ZZ/8)^2 \times \ZZ/2\\
\kappa_{2,7} & \cong (\ZZ/2)^2\\
\kappa_{2,s} & = 0,\qquad s > 7. 
\end{align*}
Furthermore,
\[
\xymatrix@R=10pt{
\kappa_2/ \kappa_{2,5} \ar[r]_-\cong^-{\phi_3} &
\ZZ/2\{\wchi\eta\} \times \ZZ/2\{\zeta \langle \wchi, 2, \eta\rangle\} \\
\kappa_{2,5}/ \kappa_{2,7}\ar[r]_-\cong^-{\phi_5} & \ZZ/4\{k\nu\} \times \ZZ/4\{\zeta e \nu\} \times \ZZ/2\{\zeta\wchi\eta^2\}\\
\kappa_{2,7} \ar[r]_-\cong^-{\phi_7} &  \ZZ/2\{k\eta^3\} \times  \ZZ/2\{\zeta e\eta^3\}.
}
\]
\end{thm}

\begin{proof} 
The group $\kappa_{2,5}$ is of the given form by \Cref{kappa25:split} and \cref{prop:kappaG48-filt}. The proof of \Cref{kappa25:SES} also gave the isomorphism of $\kappa_{2}/\kappa_{2,5}$ as claimed.

Now consider the diagram
\[
\xymatrix@C=13pt{
0 \ar[r] & \kappa_5(G_{48})\ar[d]^{\phi_5} \ar[r] & \kappa_{2,5}\ar[r]\ar[d]^{\phi_5} & \Z/8\ar[r]\ar[d] & 0\\
 0 \ar[r]& \Z/4 \{\zeta e \nu \}\times \Z/2 \{\zeta \wchi \eta^2 \} \ar[r] & \im(\phi_5) \ar@{^(^->}[d] \ar[r]& E_5^{5,4}(G_{48},S^0)\cong \Z/4\{ k \nu\} \ar[r] & 0,\\
&& E_5^{5,4}(\GG_2,S^0)
}
\]
where the left vertical map is surjective by \Cref{prop:kappaG48-filt}, and the right vertical map is surjective by \Cref{rem:PasPhi5}. This gives the image of the middle $\phi_5$, which is the quotient $\kappa_{2,5}/\kappa_{2,7}$ as claimed.

The Snake Lemma for the kernels in the above diagram now gives a short exact sequence
\[
0 \rightarrow \Z/2\cong\kappa_{7}(G_{48})\rightarrow \kappa_{2,7} \rightarrow \Z/2\rightarrow 0,
 \]
 which has to split since $\phi_7: \kappa_{2,7} \to E_7^{7,6}(\GG_2,S^0)$ is injective and $E_7^{7,6}(\GG_2,S^0)$ has exponent 2. In fact, $E_5^{7,6}(\GG_2,S^0) \cong \ZZ/2 \{ k\eta^3\} \times \ZZ/2 \{ \zeta e\eta^3 \}$, so we conclude that $E_5^{7,6}(\GG_2,S^0) \cong E_7^{7,6}(\GG_2,S^0)$ and that $\phi_7$ must be onto.

Finally, we have that $\kappa_{2,s}$ is trivial for $s>7$ by \Cref{lem:kappa8is0}.
\end{proof}

\begin{rem}\label{rem:phi-not-surj}
Note that this proof shows that $\phi_5$ does not surject onto the group $E_5^{5,4}(\GG_2,S^0)$, which was computed in \cite{BBGHPScoh}, see \Cref{fig:HFPSSG2}. 
Namely, the class $\eta^2 e$ is not in the image of $\phi_5$.
\end{rem}

\bibliographystyle{alphaurl}
\bibliography{pic-bib}

\end{document}